\documentclass{article}

\usepackage{Jared-MK}
\usepackage{graphicx} 
\usepackage{mathtools} 
\usepackage{stmaryrd} 
\usepackage{hyperref}
\usepackage{thmtools}
\usepackage{thm-restate}
\usepackage{cleveref}
\usepackage{authblk}

\declaretheorem[name=Remark,numberwithin=section]{remark}

\newtheorem{assumption}{Assumption}
\newcommand{\eps}{\epsilon}

\newcommand{\p}{\partial}
\newcommand{\n}{\nabla}
\newcommand{\bM}{\mathbf{M}}
\newcommand{\bF}{\mathbf{F}}
\newcommand{\bI}{\mathbf{I}}

\newcommand{\ZZ}{\mathcal{Z}}
\newcommand{\Znr}{\mathcal{Z}_{n,\mathrm{rel}}}
\renewcommand{\AA}{\mathcal{A}}
\newcommand{\tPhi}{\tilde{\Phi}}
\newcommand{\tPsi}{\tilde{\Psi}}
\newcommand{\tX}{\tilde{X}}

\newcommand{\tZ}{\tilde{Z}}
\newcommand{\CC}{\mathcal{C}}
\newcommand{\Cyl}{\text{Cyl}}

\renewcommand{\C}{\mathcal{C}}
\newcommand{\F}{\mathbf F}
\newcommand{\mf}{\mathbf}
\newcommand{\mb}{\mathbb}
\newcommand{\mc}{\mathcal}

\newcommand{\oB}{\overline{B}}
\newcommand{\Id}{\mathrm{Id}}
\newcommand{\B}{\overline{\mathbf B}}
\newtheorem{theo}{Theorem}[]

\newtheorem{lemme}[theo]{Lemma}
\newtheorem{remarque}[theo]{Remark}

\newtheorem*{theorem*}{Theorem}

\newcommand{\Core}{\text{Core}}
\newcommand{\comp}{\text{Comp}}

\title{Infinitely Many Surfaces with Prescribed Mean Curvature in the Presence of a Strictly Stable Minimal Surface}

\author[1]{Pedro Gaspar}
\author[2]{Jared Marx-Kuo}
\affil[1]{Facultad de Matem\'aticas, Pontificia Universidad Cat\'olica de Chile, Avenida Vicuña Mackenna 4860, Santiago, Chile}
\affil[2]{Department of Mathematics, Rice University, Houston, TX 77005, USA}
\affil[]{pedro.gaspar@uc.cl \ ${}^2$ jm307@rice.edu}

\date{}
\begin{document}

\maketitle

\begin{abstract}
	\noindent We construct infinitely many distinct hypersurfaces with prescribed mean curvature (PMC) for a large class of prescribing functions when $(M^{n+1}, g)$ is a closed smooth manifold containing a minimal surface that is strictly stable (or more generally, admits a contracting neighborhood). In particular, we construct infinitely many distinct PMCs when $H_n(M, \Z_2) \neq 0$, or if $(M, g)$ does not satisfy the Frankel property. 
    Our construction synthesizes ideas from Song's construction of infinitely many minimal surfaces in the non-generic setting, Dey's construction of multiple constant mean curvature surfaces, and Sun--Wang--Zhou's min-max construction of free boundary PMCs.
\end{abstract}

\tableofcontents

%





%
%
\section{Introduction} \label{Introduction}
For a closed Riemannian manifold $(M^{n+1}, g)$, the \emph{p-widths} are a collection of geometric invariants \nl 
$\{\omega_p(M, g)\}_{p\in\N}$ introduced by \cite{gromov2002isoperimetry, Gromov, gromov2010singularities}. These values serve as a nonlinear analog of the spectrum of the Laplacian but for the area functional. The $p$-widths play an essential role in many significant breakthroughs in the study of minimal hypersurfaces, 
and we refer the reader to \cite{Almgren, Pitts, MN2014Willmore, MarquesNevesMultiplicity, MarquesNevesPositive,  LiokumovichMarquesNeves, irie2018density, ZhouMultiplicity, gaspar2018allen, gaspar2019weyl, chodosh2020minimal, DeyCMCs, stevenssun} for the historical developments of this program.

The reason for the utility of the volume spectrum in the study of minimal hypersurfaces is the fact that (when $3 \leq n+1 \leq 7$) each $p$-width equals the weighted area of a collection of \emph{min-max} minimal hypersurfaces: namely, there are disjoint, connected, smooth, closed, embedded, minimal hypersurfaces $\{\Sigma_{p,j}\}_{j=1}^{N(p)}$ and positive integers $\{m_{p,j}\}_{j=1}^{N(p)}$ so that
\begin{equation}
\label{eqn:rep_intro} \omega_p(M, g) = \sum_{j=1}^{N(p)} m_{p,j}\mathrm{Area}(\Sigma_{p,j}). 
\end{equation}
An essential feature of the $p$-widths is the corresponding Weyl law due to Gromov, Guth, Liokumovich--Marques--Neves:
\begin{theorem}[\cite{LiokumovichMarquesNeves, Guth, Gromov}]
For $\omega_p(M,g)$ as above on a compact Riemannian manifold, $(M^{n+1}, g)$, there exists $a(n) > 0$ so that 
\[
\lim_{p \to \infty} \omega_p(M,g) p^{-\frac{1}{n+1}} = a(n) \mathrm{Vol}_g(M)^{\frac{n}{n+1}}
\]
\end{theorem}
\noindent The Weyl law was an important tool in the resolution of Yau's conjecture \cite{yau1982problem}: Any closed three-dimensional manifold must contain an infinite number of immersed minimal surfaces. The resolution of this conjecture was actually \textit{stronger}:
\begin{theorem}[\cite{chodosh2020minimal} \cite{song2018existence} \cite{marques2019equidistribution} \cite{MarquesNevesPositive} \cite{irie2018density} \cite{ZhouMultiplicity}] \label{YauResolution}
	On any closed manifold $(M^{n+1}, g)$ with $3 \leq n + 1 \leq 7$, there exist infinitely many embedded minimal hypersurfaces.
\end{theorem}
%
\noindent We also note the work of Li \cite{li2023existence} who showed the existence of infinitely many embedded minimal hypersurfaces with optimal regularity in $(M^{n+1},g)$ when $n + 1 \geq 8$ and $g$ is a generic metric, as well as the work of Li-Wang \cite{LiWang} who showed the generic regularity of minimal hypersurfaces in dimension $n+1=8$.  \nl 
\indent One may hope to prove a similar conjecture for hypersurfaces with constant mean curvature (CMC) or prescribed mean curvature (PMC):
%
\begin{conjecture} \label{PMCconjecture}
 For any $h \in C^{\infty}(M)$, there exists infinitely many distinct hypersurfaces, $\{Y_p\}$, with $H_{Y_p} = h\Big|_{Y_p}$. 
\end{conjecture}
\noindent The main difficulty towards conjecture \ref{PMCconjecture} is that there is no analogy of the $p$-widths for the prescribed mean curvature functional: given $\Omega \in \CC(M)$ a Cacciopoli set and $h \in C^{\infty}(M)$, we define
\begin{align} \nonumber
\AA^h &: \CC(M) \to \R \\ \label{PMCFunctionalDef}
\AA^h(\Omega) &= \text{Per}(\Omega) - \int_{\Omega} h
\end{align}
when $h \equiv c > 0$ we denote this as $\AA^c$. Indeed, the fact that $\omega_p \neq 0$ for all $p$, can be seen by the non-trivial topology of $\mathcal{Z}_n(M; \Z_2)$, as it is weakly homotopic to $\R \P^{\infty}$ and $\pi_1(\mathcal{Z}_n(M; \Z_2)) = \pi_1(\R \P^{\infty}; \Z_2) = \Z_2$. In some sense, this comes from the fact that Cacciopoli sets produce a double cover of $\mathcal{Z}_n(M,\Z_2)$, i.e. $\partial \Omega = \partial \Omega^c \in \mathcal{Z}_n(M;\Z_2)$, and the perimeter functional is invariant under taking complements (see \cite[\S 5]{MarquesNevesMultiplicity} for details). Contrast this with equation \ref{PMCFunctionalDef}, as in general $\AA^h(\Omega) \neq \AA^h(\Omega^c)$ for all $\Omega \in \CC(M)$. \newline 
\indent Despite this roadbloack, partial progress has been made towards conjecture \ref{PMCconjecture} for the case of $h \equiv c > 0$. In \cite{DeyCMCs}, Dey used the p-widths to construct many c-CMC surfaces for $c$ small. Let $W_0$ denote the \textit{one parameter Almgren-Pitts width} of a manifold (see \cite[\S 2.3]{DeyCMCs} and \S \ref{MinMaxBackground} below). 
\begin{theorem}[Thm 1.1, \cite{DeyCMCs}] \label{DeyCMCThm}
Let $(M^{n+1}, g)$ be a closed Riemannian manifold, $n+1 \geq 3$. Suppose $p \in \mathbb{N}$ such that $\omega_p(M,g) < \omega_{p+1}(M,g)$, $c \in \mathbb{R}^+$ such that $c \cdot \text{Vol}_g(M) < \omega_{p+1}(M,g) - \omega_p(M,g)$ and $\eta \in \mathbb{R}$ is arbitrary. Then there exists an $\Omega \in \CC(M)$ such that $\p \Omega$ is a closed c-CMC (with respect to the inward unit normal) hypersurface with optimal regularity and $\omega_p(M,g) < \mathcal{A}^c(\Omega) < \omega_p(M,g) + W_0 + \eta$.
\end{theorem}
\noindent Dey further shows that because of the Weyl Law, asymptotically $\omega_{p+1} - \omega_p > 0$. However, due to the \textit{sublinear} growth (which was important in \cite{MarquesNevesPositive, song2018existence, irie2018density}), 
\begin{equation} \label{widthDifference}
\lim_{p \to \infty} \omega_{p+1} - \omega_p = 0
\end{equation}
and so for fixed $c > 0$, Dey's construction only produces finitely many $c$-CMC surfaces. However, he does show \cite[Theorem 1.2]{DeyCMCs} that given $c > 0$, there exists at least $\gamma_0 c^{-\frac{1}{n+1}}$ many closed c-CMC surfaces. We also mention his related work \cite{DeyWidths} showing a subadditivity law for the $p$-widths. \newline 
\indent Dey's construction is part of a larger effort to construct arbitrarily many $c$-CMC surfaces for any fixed $c > 0$. We mention the work of Zhou--Zhu \cite{zhou2019min}, which was a foundational in Dey's proof and showed the existence of at least one nontrivial smooth, closed, almost embedded $c$-CMC surface in an arbitrary closed manifold. We also mention the work of Pacard--Xu \cite{pacard2009constant}, who showed that for $\Lambda_M > 0$, the Lusternik-Shnirelman category of $M$, there are at least $\Lambda_M$, $c$-CMC surfaces for all $c$ sufficiently large. \nl 
\indent However, it is currently unknown if there exists even $2$ distinct c-CMC surfaces (or even constant curvature geodesics) for any $c > 0$ - this is known as the ``Twin Bubble Conjecture," attributed to Zhou \cite{zhou2022mean}, as a higher dimensional analogue of Arnold's conjecture \cite{arnold2004arnold}. In a different direction, Mazurowski--Zhou \cite{mazurowski2023half, mazurowski2024infinitely, mazurowski2024alternative} have shown that generically, in low dimensions, there exist infinitely many constant mean curvature hypersurfaces, which bound a Caccioppoli set with volume equal to $\frac{1}{2} \text{Vol}_g(M)$. We note however, that the value of $c$ is not uniform across all of these sets. Very cleverly, Mazurowski--Zhou exploit the $\Z_2$ invariance of half-volume Caccioppoli sets (i.e. if $\text{Vol}(\Omega) = \frac{1}{2} \text{Vol}(M)$, then $\text{Vol}(\Omega^c) = \frac{1}{2} \text{Vol}(M)$), to define a half-volume spectrum, $\{\tilde{\omega}_p\}$ and recreate an analogous min-max program to produce infinitely many half-volume Caccioppoli sets with boundary having constant mean curvature. \newline 
\indent In the line of producing PMCs, Zhou--Zhu \cite{ZhouZhu} had foundational work for producing PMCs via 1-parameter min-max in closed manifolds. The existence and regularity of PMCs via min-max on relative homotopy classes was vital to Zhou's resolution of the ``Multiplicity-One" conjecture of Marques-Neves \cite{MN2014Willmore, MarquesNevesMultiplicity} for generic metrics \cite[Thm A]{ZhouMultiplicity}. Sun--Wang--Zhou \cite{SunWangZhou} then extend many of these ideas to the setting of manifolds with boundary and free boundary minimal surfaces - see \S \ref{PMCMinMaxBackground} for more details. In terms of regularity, Sarnataro--Stryker \cite{sarnataro2023optimal} proved optimal $C^{1,1}$ regularity for PMCs over isotopy classes, producing PMC $2$-spheres in the round $3$-sphere, for an open dense set of prescribing functions of sufficiently small $L^\infty$ norm. In the non-compact setting, Mazurowski \cite{mazurowski2022prescribed} shows the existence of PMC hypersurfaces in $\mathbb{R}^n$ with prescribing function asymptotic to a positive constant near infinity. Very recently, Stryker \cite{stryker2024min} shows the existence of PMCs with sign constraints outside a compact set via min-max with the obstacle problem. \medskip \newline 
\noindent In this paper, we expand Dey's work by asking:
\begin{enumerate}
\item \label{LinearGrowthQ} What if $\omega_{p+1}(M,g) - \omega_p(M,g) \geq K > 0$ for all sufficiently large $p$?
\item \label{PMCoverCMC} What if we consider prescribed mean curvature (PMC) surfaces instead of constant mean curvature (CMC) surfaces?
\end{enumerate}
Question \ref{LinearGrowthQ} can be resolved when $(M^{n+1}, g)$ is a non-compact manifold with finitely many cylindrical ends (\cite[Theorem 9]{song2018existence}), and such manifolds were essential in Song's resolution of Yau's conjecture. Similarly, question \ref{PMCoverCMC} (in the compact setting) is essentially contained in Dey's work \cite{DeyCMCs} if one considers a prescribing function $h \in C^{\infty}(M)$ and  replaces the condition of $c \text{Vol}(M) \in [0, \omega_{p+1} - \omega_p]$ with $\|h\|_{L^1(M)} \in [0, \omega_{p+1} - \omega_p]$. Then, the regularity theory of Zhou--Zhu \cite{ZhouZhu} and Sun--Wang--Zhou \cite{SunWangZhou} demonstrates the existence of a PMC with optimal regularity for every $p$ such that $\|h\|_{L^1(M)} \in [0, \omega_{p+1} - \omega_p]$. To synthesize the two ideas, we will construct a sequence of compact manifolds which are converging to a manifold with a cylindrical end, while applying Dey's construction on each manifold in the sequence. Via novel diameter estimates and barrier arguments, we will show that the construction produces closed PMCs (as opposed to free boundary PMCs) and converge nicely in the limit.
%
\subsection{Main Results} \label{MainResultsSection}
Inspired by Song's cylindrical Weyl law and Dey's construction of $c$-CMC surfaces, we show the existence of infinitely many PMC surfaces on closed manifolds, when it is possible to attach a cylindrical end a la Song's construction and when the prescribing function $h$ satisfies certain restrictions. The construction begins by working on manifolds with boundary. \newline 
%
\indent We will consider a manifold $(M^{n+1}, g)$, such that $\Sigma = \partial M$, is a separating, strictly stable minimal surface. Such a hypersurface admits a ``contracting neighborhood" to one side, i.e. there exists a neighborhood $U \supseteq \Sigma$ and a smooth embedding $\Phi: \Sigma \times [0, \tilde{t}] \to U$ such that $H_t = H_{\Phi(\cdot, t)}$ has mean curvature vector pointing towards $\Phi(\cdot,0)(\Sigma)=\Sigma$. See \S \ref{construction} for more details. Let $\Sigma = \sqcup_{i=1}^m \Sigma_i$ be the decomposition of $\Sigma$ into its connected components, with $A(\Sigma_1) \geq A(\Sigma_i)$ for all $i$, where we denote by $A(\cdot)$ the area functional. \nl 
\indent We will also consider prescribing functions which are ``good" in the sense of Zhou--Zhu \cite{zhou2019min}, which we recall here. Let $h: M^{n+1} \to \R$ be a smooth function on a Riemannian manifold $(M^{n+1}, g)$. We say that $h \in \mathcal{S}^*(g)$ if at least one of the following hold
\begin{enumerate}
    \item\label{morseCondition} $h$ is morse, $\Sigma_0:= \{h = 0\}$ is a smooth closed hypersurface and the mean curvature of $\Sigma_0$ vanishes to at most finite order
    \item The zero set of $h$ is contained in a countable union of connected, smoothly embedded $(n-1)$-dimensional submanifolds.
\end{enumerate}
While condition \ref{morseCondition} is central to the work of Zhou \cite{ZhouMultiplicity} and Sun--Wang--Zhou \cite{SunWangZhou}, whose results we make use of, we note that most of their theorems extend to any of the above conditions listed in the definition of $\mathcal{S}^*(g)$ (see \S \ref{MinMaxBackground}). We also note that while the definition of $\mathcal{S}^*(g)$ is stated for closed manifolds, it also makes sense for manifolds with boundary or open manifolds. \nl 
\indent In addition, we say that a metric $g$ is bumpy if no minimal surface in $M$ admits a Jacobi field. By the bumpy metrics theorem of White \cite{white2017bumpy, white1991space}, this condition is Baire generic. Furthermore, we let $W_0$ denote the $1$-parameter Almgren--Pitts width of $(M,g)$ (see \S \ref{MinMaxBackground}). \medskip 
Our main result will consider prescribing functions which satisfy the following conditions: 
\begin{assumption} \label{hZeroBoundary} 
Assume that $h: M \to \R$ satisfies
\begin{enumerate} 
\item $\|h\|_{L^1(M,g)} < A(\Sigma_1)/2$ 
\item\label{vanishingBoundary}  $h \in C^{\infty}(M)$, $ h\Big|_{\Sigma} = \partial_{\nu} h \Big|_{\Sigma} = 0$, and $\int_{M} h > 0$ 
%
\item\label{goodnessAssumption2} $h \Big|_{M \setminus \Sigma} \in \mathcal{S}^*(g)$ 
%
\end{enumerate}
\end{assumption}
\begin{remark}
We compare condition \ref{goodnessAssumption2} to the following generic conditions of Sun--Wang--Zhou for a manifold with boundary $(M^{n+1}, \partial M, g)$: $h$ is a Morse function and $\Sigma_0 = \{h = 0\}$ is a compact, smoothly embedded hypersurface so that
\begin{enumerate}
\item[(a)] $\Sigma_0$ is transverse to $\partial M$ and the mean curvature of $\Sigma_0$ vanishes to at most finite order
\item[(b)] $\{x \in \partial M \; : \; H_{\partial M}(x) = h(x) \text{ or } H_{\partial M}(x) = -h(x)\}$ is contained in an $(n-1)$-dimensional submanifold of $\partial M$. 
\end{enumerate}
In our setting, it does not make sense to enforce conditions (a) or (b), as $h\Big|_{\Sigma} = 0$ so that neither are true. Sun--Wang--Zhou choose these conditions to prevent min-max PMCs from accumulating or sticking to the boundary, but we will show separately that our constructed min-max PMCs will be closed and never intersect the boundary (see \S \ref{NoSigmaSection} or \S \ref{quantMPSection}). 
\end{remark}
\noindent With these assumptions, we prove:
\begin{theorem} \label{ZeroBoundaryTheorem}
Suppose that $(M^{n+1}, g)$, $3 \leq n+1 \leq 7$, is a manifold with boundary so that $\Sigma = \partial M$ is an embedded strictly stable minimal surface. For any $h$ satisfying Assumptions \ref{hZeroBoundary}, there exist infinitely many distinct, almost embedded with optimal regularity, \emph{multiplicity one} hypersurfaces, $\{Y_{h,p}\}_{p=1}^{\infty}$, with mean curvature $h$ and 
\begin{equation} \label{AreaBounds}
\forall p, \qquad (p+1) \cdot A(\Sigma_1) - 2 \|h\|_{L^1(M)} \leq A(Y_{h,p}) \leq p \cdot A(\Sigma_1) + W_0 + A(\Sigma) + C \cdot p^{1/(n+1)} + 2 \|h\|_{L^1(M)}
\end{equation}
for some $C = C(M, g) > 0$. Moreover, each $Y_{h,p}$ is disjoint from $\Sigma$ and $\mathrm{Index}(\mathcal{R}(Y_{h,p})) \leq p + 1$.
\end{theorem}
We recall the definition of being ``almost embedded" with multiplicity one in \S \ref{MinMaxBackground}, \ref{almostEmbeddingDef}. In general, an almost embedded PMC, $Y_h$, may have \textbf{density $2$} at certain points, and in particular, along an open (codimension $0$) subset $U \subseteq Y_h$. For such a surface, $\mathcal{S}(Y_h)$ denotes the ``touching set" and $\mathcal{R}(Y) = Y \backslash \mathcal{S}(Y)$ denotes the ``regular set" where $Y$ has density $1$. See definitions \ref{almostEmbeddingDef} and figure \ref{fig:dense2}.
\begin{figure}[h!]
    \centering
    \includegraphics[scale=0.35]{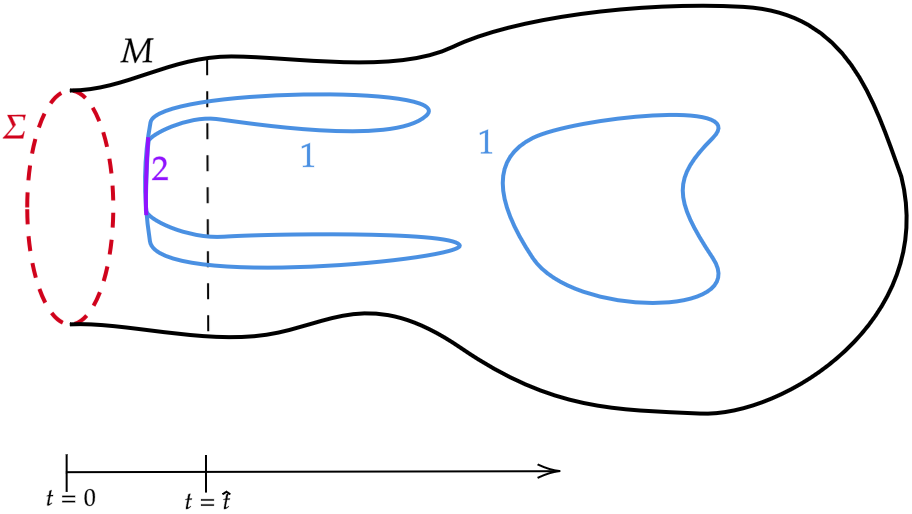}
    \caption{Example of a PMC with large (codimension $0$) touching set}
    \label{fig:dense2}
\end{figure}
We emphasize that figure \ref{fig:dense2} \textit{does not} occur in the context of theorem \ref{ZeroBoundaryTheorem}, as the phrase ``almost embedded with optimal regularity" means that each $Y_{h,p}$ is density one away from a small set of \textit{$(n-1)$-Hausdorff dimension} (codimension $1$, see definition \ref{almostEmbedOptimal}). In particular, there will be no open set with density $2$ as pictured in figure \ref{fig:dense2}, but a smaller set of points of density $2$ may occur for the PMCs produced by theorem \ref{ZeroBoundaryTheorem}. \nl 
\indent While the restrictions of assumption \ref{hZeroBoundary} appear stringent, we note the following corollary:
\begin{corollary}
Suppose $h: M \to \R$ with $h\Big|_{\Sigma} = \partial_{\nu} h \Big|_{\Sigma} = 0$, $h \Big|_{M \backslash \Sigma} > 0$, and $||h||_{L^1(M)} < A(\Sigma_1)/2$. Then there are infinitely many $h$-PMCs with the optimal regularity described in theorem \ref{ZeroBoundaryTheorem}.
\end{corollary}
\indent In order to prove theorem \ref{ZeroBoundaryTheorem}, we approximate functions satisfying assumptions \ref{hZeroBoundary} with prescribing functions which are compactly supported away from $\Sigma$ and are ``good" in the complement of the foliating neighborhood:
\begin{assumption} \label{hCompactSupport} 
Assume that $h \colon M \to \R$ satisfies

\begin{enumerate}
    \item \label{LOneBound} $\|h\|_{L^1(M,g)} < A(\Sigma_1)/2$
    \item \label{LessThanSlice} For some $0 < \hat{t} \leq \tilde{t}$, we have that $|h| \Big|_{\Sigma \times r} < H_r$ for all $r \in (0, \hat{t}]$
    \item \label{compactPositiveIntegral} $h \in C_c^{\infty}(M)$, $\int_{M} h > 0$
    \item \label{goodnessAssumption} For $\hat{t}$ as above, define $M_{\hat{t}} = U^c \cup \{p \; | \; t(p) \geq \hat{t}\}$. Then $h \in S^*(M_{\hat{t}}, g\Big|_{M_{\hat{t}}})$
\end{enumerate}
\end{assumption}
\noindent With these assumptions, we prove the following:
\begin{theorem} \label{mainTheorem}
Suppose that $(M^{n+1}, g)$, $3 \leq n+1 \leq 7$, is a manifold with boundary so that $\Sigma = \partial M$ is an embedded, strictly stable minimal surface. For any $h \in C^{\infty}(M)$ satisfying conditions \ref{hCompactSupport}, there exist infinitely many distinct, almost embedded, \emph{multiplicity one} hypersurfaces, $\{Y_{h,p}\}_{p=1}^{\infty}$, with mean curvature $h$ and
the same area bounds as theorem \ref{ZeroBoundaryTheorem}. Moreover, each $Y_{h,p}$ is disjoint from $\Sigma$ and $\mathrm{Index}(\mathcal{R}(Y_{h,p})) \leq p + 1$.
\end{theorem}
\noindent  In contrast to theorem \ref{ZeroBoundaryTheorem}, the compact support of $h$ in theorem \ref{mainTheorem} allows for a large (potentially codimension $0$) touching set. This is pictured in figure \ref{fig:dense2}, but by condition \ref{goodnessAssumption}, a codimension $0$ touching set can only occur in $Y \cap U$. \nl  
\indent In the closed setting, we can apply theorems \ref{ZeroBoundaryTheorem} \ref{mainTheorem} as long as the manifold admits an embedded, strictly stable minimal surface. To do this, we look at the metric completion of $M \backslash \Sigma$, which we recall as follows: Given $\Sigma \subseteq M$ an embedded hypersurface, let $\text{Comp}(M \backslash \Sigma)$ denote the metric completion of $M \backslash \Sigma$ (see figure \ref{fig:metriccompletion}). When $\Sigma$ is two-sided separating, $\comp(M \backslash \Sigma) = M^+ \sqcup M^-$ has two disjoint components with $\partial M^{\pm} \cong \Sigma$. When $\Sigma$ is two-sided non-separating, $\comp(M \backslash \Sigma)$ is connected with $\partial \comp(M \backslash \Sigma)= \Sigma \sqcup \Sigma$. When $\Sigma$ is one-sided, $\comp(M \backslash \Sigma)$ is connected with $\partial \comp(M \backslash \Sigma) = \tilde{\Sigma}$, the double cover of $\Sigma$. See figure \ref{fig:metriccompletion} and note that any $h \in C^{\infty}(M)$ lifts to a function $C^{\infty}(\comp(M \backslash \Sigma))$. 
\begin{figure}[h!]
\centering
\includegraphics[scale=0.3]{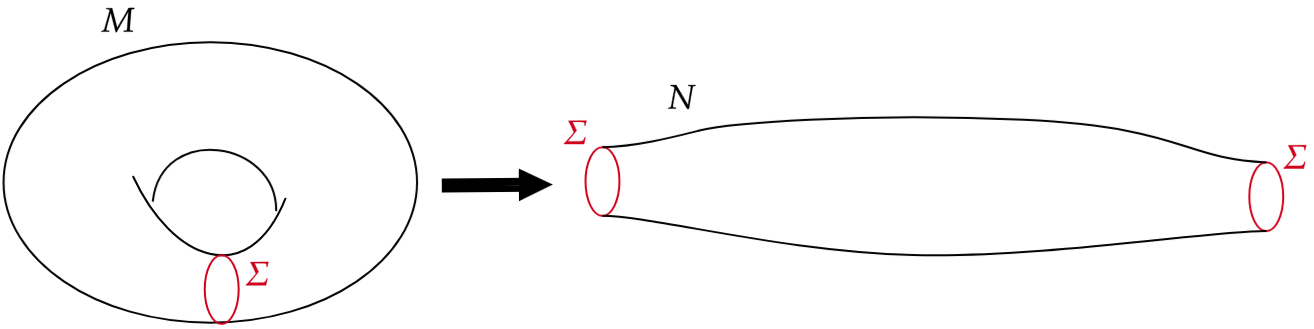}
\caption{Visualization of $N$, the metric completion of $M \backslash \Sigma$}
\label{fig:metriccompletion}
\end{figure}
\nl \indent  Given that $h: M \to \R$ lifts to $\comp(M \backslash \Sigma)$, we will impose the following conditions on $h$ when we work in the closed setting. By abuse of notation, let $\tilde{\Sigma} = \partial \comp(M \backslash \Sigma)$ and $\tilde{\Sigma}_1$ its largest connected component.
\begin{theorem} \label{mainTheoremClosed}
Suppose that $M^{n+1}$, $3 \leq n+1 \leq 7$, is a closed manifold with a closed, embedded, strictly stable minimal surface, $\Sigma$. For $h \in C^{\infty}(M)$, suppose its lift to $\comp(M \backslash \Sigma)$, $\tilde{h}$, satisfies assumptions \ref{hZeroBoundary} (respectively \ref{hCompactSupport}). Then there exist infinitely many distinct, almost-embedded, smooth PMCs with the same multiplicity, area, and index bounds as in Theorem \ref{ZeroBoundaryTheorem} (respectively,  \ref{mainTheorem}). Moreover, each such PMC is disjoint from $\Sigma$. 
\end{theorem}
\noindent We note that the existence of a strictly stable embedded minimal surface is guaranteed when $H_n(M, \Z_2) \neq 0$ and $g$ is bumpy, or when $g$ is bumpy and does not satisfy the Frankel property. Thus, we have the following corollaries:
\begin{corollary} \label{mainCorollaryHomology}
Suppose that $M^{n+1}$, $3 \leq n+1 \leq 7$, is a closed manifold with $H_n(M, \Z_2) \neq 0$ and $g$ bumpy. Then there exists a strictly stable, embedded, minimal surface, $\Sigma$. Moreover, there exist infinitely many distinct, almost embedded PMCs, with the same assumptions on $h \in C^{\infty}(M)$ and conclusions as in theorems \ref{mainTheoremClosed}. 
\end{corollary}
\noindent See figures \ref{fig:metriccompletion}, \ref{fig:hnpmc} for a visualization of the construction in this case.
\begin{figure}[h!]
\centering
\includegraphics[scale=0.4]{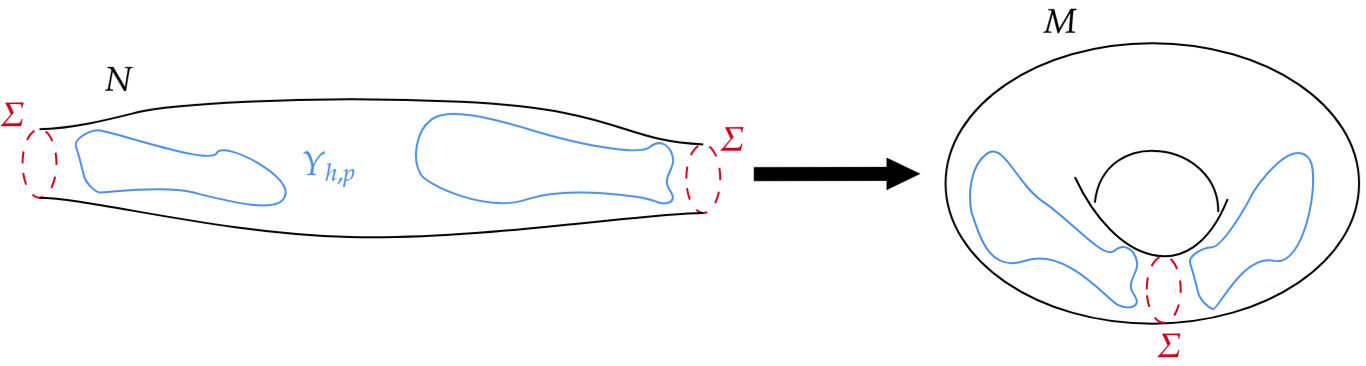}
\caption{Visualization of the constructed PMC, after gluing $\comp(M\backslash \Sigma)$ back to $M$.}
\label{fig:hnpmc}
\end{figure}
\nl Even if $H_n(M, \Z_2) = 0$, we may still be able to apply theorem \ref{mainTheorem} in the closed setting when $g$ does not satisfy the Frankel property:
\begin{corollary} \label{mainCorollaryFrankel}
Suppose $(M^{n+1}, g)$, $3 \leq n+1 \leq 7$ \textit{does not} satisfy the Frankel property and is bumpy. Then there exists a strictly stable, embedded, minimal surface, $\Sigma$, and infinitely many distinct, almost embedded PMCs, with the same assumptions on $h \in C^{\infty}(M)$ and conclusions as in theorem \ref{mainTheoremClosed}.
\end{corollary}
\noindent See figure \ref{fig:dumbbell} for a visualization of the non-Frankel case.
\begin{figure}[h!]
\centering
\includegraphics[scale=0.3]{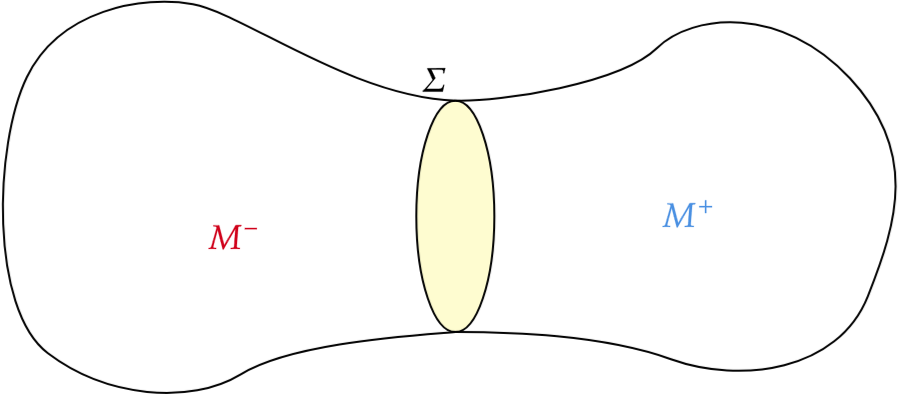}
\caption{Visualization of a dumbell metric which is not Frankel and for which corollary \ref{mainCorollaryFrankel} applies.}
\label{fig:dumbbell}
\end{figure}
\nl \indent To the authors' knowledge, this is the first construction of \textbf{infinitely many PMC surfaces} on any type of manifold or for any type of (non-zero) prescribing function. As a result, this can be viewed as partial progress towards the resolution of conjecture \ref{PMCconjecture}. 
\begin{remark} \label{contractingRemark}
As will be clear in the proof of theorems \ref{mainTheorem} \ref{ZeroBoundaryTheorem}, many of the results will hold if $\Sigma = \partial M$ is simply a minimal surface with a \textit{contracting neighborhood} (see \S \ref{construction}). Such a neighborhood exists automatically when $\Sigma$ is strictly stable (see e.g. \cite[Lemma 12]{song2018existence}), however, it is possible to obtain infinitely many PMCs when $g$ is not bumpy and $\Sigma$ admits a contracting neighborhood but is not strictly stable.  
\end{remark}

\subsection{Main Ideas}
As mentioned in \S \ref{Introduction}, Theorems \ref{ZeroBoundaryTheorem} and \ref{mainTheorem} are heavily inspired by Dey's \cite[Theorem 1.1]{DeyCMCs}. Adopting Dey's suspension construction of a $p+1$ sweepout to construct a c-CMC (see \S \ref{DeySuspension}), we construct infinitely many PMCs following Song's construction of infinitely many minimal surfaces on a manifold with a cylindrical end. In the case of a manifold with boundary, $\partial M = \Sigma = \sqcup_k \Sigma_k$, we proceed as follows: 
\begin{itemize}
\item Construct (following Song) a sequence of compact manifolds with boundary $(U_{\eps}, g_{\eps})$, such that $\omega_p(U_{\eps}, g_{\eps})$ approaches $\omega_p(\Cyl(M))$, where $\Cyl(M)$ denotes the Lipschitz manifold obtained by attaching a cylindrical end to $\partial M$. For any $p_0 \in \Z^+$, we can choose $\eps$ sufficiently small  so that
\[
\forall p \leq p_0, \qquad \omega_{p+1}(U_{\eps}, g_{\eps}) - \omega_p(U_{\eps}, g_{\eps}) \geq \frac{3}{4} A(\Sigma_1)
\]

\item Apply Dey's topological suspension construction and the free boundary PMC min-max of Sun--Wang--Zhou \cite{SunWangZhou} to conclude the existence of at least $p$ distinct (potentially free boundary) PMCs, $\{Y_{\eps,i}\}_{i = 1}^p$ for some $h_{\eps}$, a good prescribing function with $h_{\eps} \rightarrow h$ satisfying assumptions \ref{hCompactSupport}.

\item Via novel diameter estimates of Chambers and the second-named author \cite{chambers2024}, show that each $\{Y_{\eps,i}\}_{i = 1}^p$ does not intersect $\partial U_{\eps}$ and is in fact ``tethered to the core" of $M$, within $\Cyl(M)$. In particular, this shows that $Y_{\eps,i}$ are \textit{closed} and not free boundary PMCs. 

\item Take a limit as $\eps \to 0$ to conclude the existence of a varifold, $V_{p,h} = \lim_{\eps \to 0} Y_{\eps, p}$ with bounded first variation. Applying the regularity theory of Sun--Wang--Zhou \cite{SunWangZhou} (see also \cite{ZhouMultiplicity, ZhouZhu}) and White/Solomon-White's Maximum principle for stationary varifolds, we show that if $h$ has compact support, then $V_{h,p}$ is induced by an almost embedded PMC, and no connected component of it is equal to any component of $\partial M$.

\item We extend and improve the existence result for good functions $h \in C^{\infty}(M)$ with vanishing boundary conditions (see Assumption \ref{hZeroBoundary}) by taking an approximating sequence $h_k \in C^{\infty}_c(M)$ of prescribing functions, constructing the corresponding PMCs, $\{Y_{h_k,p}\}$, as limits of free-boundary min-max PMCs in $(U_\eps,g_\eps)$ as $\epsilon \downarrow 0$, and sending $k \to \infty$ through a diagonal argument. The maximum principle and a monotonicity (pinching) argument shows that no copies of boundary components can arise in the limit. 
%
%

\end{itemize}
In order for such a construction to work, Song requires a minimal surface with a ``contracting neighborhood", which in our context is verified if $\partial M$ is strictly stable. See \S \ref{construction} for details.\nl  
\indent We note that stable (perturbatively, strictly stable) minimal surfaces can occur via min-max, as described in work by Mazet--Rosenberg \cite{mazet2017minimal}. We also note that strictly stable manifolds arise naturally in hyperbolic geometry, where totally geodesic submanifolds are natural objects of study and always strictly stable (see e.g. \cite{filip2024finiteness}). The homology condition of corollary \ref{mainCorollaryHomology} is also relevant, as the resolution of the virtual Haken Conjecture by Agol \cite{agol2013virtual} demonstrates that every hyperbolic 3-manifold admits a finite cover with $H^2(\tilde{M}) \neq 0$. Applying a generic perturbation (which may not preserve $K = -1$) would yield a strictly stable minimal surface. Moreover, strictly stable minimal surfaces appear as the fiber of quasi-fuchsian manifolds \cite{uhlenbeck1983closed}, though these manifolds are non-compact (see also \cite{lowe2021deformations, guaraco2021mean}). Nonetheless, it would be interesting to extend our construction to this type of manifold.  \nl 
\indent We remark that our current construction \textit{does not construct c-CMC surfaces} (see condition \ref{compactPositiveIntegral}), though we hope to address this in the future. \nl 
\indent We now explain the motivation behind the assumptions on $h$ \ref{hCompactSupport} in the compact support setting:
\begin{itemize}
\item Condition \ref{LOneBound} is chosen so that $\omega_{p+1} - \omega_p > 2 \|h\|_{L^1}$, which allows one to apply a mountain pass construction of PMCs. It is motivated by the conditions of theorem \ref{DeyCMCThm} (see also \S \ref{YepsConstruction}, equation \eqref{SweepDiff}). 
\item Condition \ref{LessThanSlice} is needed to prevent free boundary PMC's in our min-max construction.
\item Condition \ref{compactPositiveIntegral} is needed to apply a maximum principle of White/Solomon--White \cite{white2009maximum, solomon1989strong} so as to prevent our p-sweepouts from producing copies of any of the $\{\Sigma_k\}_{k=1}^m$ (see \S \ref{MPsection}).

\item Condition \ref{goodnessAssumption} is needed to apply theorem \ref{thm:compactness with changing metrics} below and conclude that our min-max PMC is a smooth, multiplicity one, almost embedded surface. 
\end{itemize}
We emphasize the necessity of condition \ref{goodnessAssumption}, i.e. our compactly supported prescribing functions need to be ``good" on most of the manifold, as without any constraints, one could sweepout multiple copies of minimal surfaces in the complement of our contracting neighborhood around $\Sigma$. This is indicated in the compactness theorems of Zhou--Zhu (see e.g. \cite[Thm 3.19]{ZhouZhu}), and in general it is easy to imagine a prescribing function, $h$, along with a sequence of $h$-PMCs converging to a (nondegenerate!) minimal surface with multiplicity as follows: consider $\Sigma^n \subseteq M^{n+1}$ minimal, nondegenerate, and let $\Sigma_t$ denote the hypersurface a signed distance $t$ from $\Sigma$, along with $(s,t)$ fermi coordinates for a neighborhood of $\Sigma$. By choosing $h: M \to \R$ to be a prescribing function which is locally $h(s,t) = H_{\Sigma_t}(s)$, we can consider $\Sigma_{1/k} \cup \Sigma_{-1/k}$ to be a sequence of multiplicity one PMCs which converge to a non-degenerate minimal surface with multiplicity. Condition \ref{goodnessAssumption} prevents such multiplicity from occuring in $U^c$, and as described in section \S \ref{noMinimalSection}, the maximum principle prevents any minimal components contained in the contracting neighborhood. \nl
\indent We also explain assumptions \ref{hZeroBoundary} for the prescribing functions from theorem \ref{ZeroBoundaryTheorem}:
\begin{itemize}
\item Condition \ref{LOneBound} is identical to the previous setting and needed to apply the mountain pass construction of Dey
\item Condition \ref{vanishingBoundary} is needed as we will approximate our prescribing functions in theorem \ref{ZeroBoundaryTheorem} via compactly supported functions converging in \underline{$C^1$}. The need for $C^1$ is due to the $C^1$-regularity requirements for curvature estimates of PMCs (see remark \ref{C1 alpha convergence}), and hence compactness theorems of PMCs.

\item Condition \ref{goodnessAssumption2} is again imposed to apply theorem \ref{thm:compactness with changing metrics} and prevent multiplicity in our min-max hypersurfaces.
\end{itemize}
\subsection{Paper Organization}
The paper is organized as follows:
\begin{enumerate}
\item In \S \ref{background}, we establish essential background including Song's construction of a cylindrical end \S \ref{construction}, Almgren--Pitts' general min-max theory \S \ref{MinMaxBackground}, the PMC min-max theory of Zhou--Zhu \cite{ZhouZhu} and Sun--Wang--Zhou \cite{SunWangZhou} \S \ref{PMCMinMaxBackground}, and Dey's suspension construction \S \ref{DeySuspension}. We show the existence of a one parameter Almgren--Pitts sweepout with respect to \textit{relative cycles}, $\Znr(M, \partial M; \Z_2)$ (\S \ref{APWidthBoundarySection}), and compute bounds on the corresponding width for manifolds with cylindrical ends (\S \ref{APWidthBoundSection}). To the authors' knowledge, this had not been done before.

\item In \S \ref{YepsConstruction}, we will begin the proof of theorem \ref{mainTheorem} by applying Dey's suspension construction to produce $p$ PMC surfaces, $\{Y_{\eps, i}\}_{i = 1}^p$ on $(U_{\eps}, g_{\eps})$ for all $\eps$ sufficiently small. The surfaces obey $H = h_{\eps} \Big|_{Y_{\eps, p}}$, where in the limit $h_{\eps}$ vanishes on the cylindrical end and converges to $h \in C^{\infty}_c(M)$, our original prescribing function. We also show that the $\{Y_{\eps, i}\}$ are closed (and hence not free boundary) using diameter estimates from \cite{chambers2024} and a ``tethering" argument.

\item In \S \ref{EpsToZero}, \ref{conclusionSection}, we complete the proof of theorem \ref{mainTheorem} by showing that the surfaces $\{Y_{\eps, i}\}$ converge to distinct PMCs on the original manifold $M$. The main goal is to prevent accumulation of $\{Y_{\eps,i}\}$ on $\Sigma$ and also verify that the PMCs are multiplicity one and density at most $2$. The argument requires an adaptation of a maximal principle for stationary varifolds from Solomon--White \cite{solomon1989strong}, as well as a ``no-pinching" argument to ensure that no part of our limiting object converges to a component of $\Sigma$.

\item In \S \ref{ExtensionSection}, we extend our argument to prescribing functions satisfying assumptions \ref{hZeroBoundary}, proving theorem \ref{ZeroBoundaryTheorem}. This follows by an approximation with compactly supported functions, $\{h_i\} \in C^{\infty}_c(M)$. Here, the regularity of the curvature estimates for surfaces with prescribed mean curvature requires us to have $h$ vanish at the boundary to order $1$.

\item In \S \ref{corollariesSection}, we prove theorem \ref{mainTheoremClosed} as a direct application of theorems \ref{mainTheorem}, \ref{ZeroBoundaryTheorem} on $\comp(M \backslash \Sigma)$. Since the resulting PMCs are always bounded away from $\partial \comp(M \backslash \Sigma)$, they isometrically (almost) embedin the original manifold, $M$. We also prove corollaries \ref{mainCorollaryHomology} \ref{mainCorollaryFrankel} by showing the existence of a stable embedded minimal surface, $\Sigma$, and then applying theorems \ref{mainTheorem} \ref{ZeroBoundaryTheorem} on $\comp(M \backslash \Sigma)$. 
\end{enumerate}
\subsection{Acknowledgements}
The authors are grateful to Otis Chodosh, Christos Mantoulidis, and Antoine Song for inspiring conversations on the work. The authors also thank Xin Zhou, Jonathan Zhu, and Costante Bellettini for answering questions during the process of this work. The second author is grateful to David Fisher and Ben Lowe for their expertise in hyperbolic geometry and the existence of strictly stable minimal surfaces in these spaces. The first author was supported by ANID Fondecyt Iniciaci\'on grant number 11230874. This work began while at the ``Geometric Flows and Relativity" workshop in Punta Del Este, Uruguay of March 2024, and the authors are thankful to the organizers for coordinating the event.
%
%
\section{Background} \label{background}
%

\subsection{Min-max theory} \label{MinMaxBackground}

In this section, we recall some definitions and results from Almgren-Pitts min-max theory, as well as the existence results for free-boundary hypersurfaces of prescribed mean curvature (PMC) in a compact Riemannian manifold $(M^{n+1},g)$ with nonempty boundary $\partial M$. We follow \cite{DeyCMCs,DeyWidths, LiokumovichMarquesNeves,ZhouMultiplicity,SunWangZhou}.

For $d=0,1,\ldots,n+1$. we denote by $\bI_{d}(M;\Z_2)$ the space of $d$-dimensional mod 2 flat chains in $M$. The space $\bI_{n+1}(M,\Z_2)$ can be naturally identified with the space $\mathcal{C}(M)$ of Caccioppoli sets $\Omega\subset M$, by identifying every such set with the current $[\![\Omega]\!]$. Under this identification, the boundary of $\partial \Omega \in \bI_n(M;\Z_2)$ is naturally identified with the current associated with the reduced boundary $\partial^*\Omega$ of $\Omega$. We also consider the spaces of \emph{mod $2$-cycles} 
    \[\ZZ_d(M;\Z_2)=\{T \in \bI_d(M;\Z_2)\colon \partial T=0\}\]
endowed with the \emph{flat metric}
    \[
    \mathcal{F}(T_1,T_2) = \inf\{\mathbf{M}(Q)+\mathbf{M}(R) \colon T_1-T_2 = Q+\partial R\},
    \]
where we denote by $\mathbf{M}(T)$ the \emph{mass} of a cycle. We remark that for $\Omega \in \bI_{n+1}(M;\Z_2)$, the mass $\mathbf{M}(\partial \Omega)$ is the perimeter $P_{(M,g)}(\Omega)=\mathcal{H}^{n}(\partial^* \Omega)$ of $\Omega$, in the sense of Caccioppoli sets.

For $T \in \bI_d(M;\Z_2)$ we will denote by $|T|$ the (multiplicity one) integral $d$-varifold associated with $T$, and by $\|T\|$ the associated Radon measure in $M$. We also consider the $\mathbf{F}$-\emph{metric} in the spaces $\bI_d(M,\Z_2)$ defined by
    \begin{align*}
        \mathbf{F}(\Omega_1,\Omega_2) &= \mathcal{F}(\Omega_1,\Omega_2) + \mathbf{F}(|\partial \Omega_1|,|\partial \Omega_2|), \qquad  \Omega_1,\Omega_2 \in \mathcal{C}(M)\\
        \mathbf{F}(T_1,T_2) &= \mathcal{F}(T_1,T_2) + \mathbf{F}(|T_1|,|T_2|), \qquad T_1,T_2 \in \bI_d(M,\Z_2),
    \end{align*}
where the $\mathbf{F}$-metric for $n$-varifolds is defined by Pitts in \cite[p. 66]{Pitts}. When equipped with these metrics, the spaces of $n$ and $(n+1)$-flat cycles will be indicated by $\bI_n(M,\mathbf{F};\Z_2)=\ZZ_n(M;\mathbf{F};\Z_2)$ and $\bI_{n+1}(M,\mathbf{F};\Z_2)=(\mathcal{C}(M),\mathbf{F})$. We remark that (see \cite[p. 21]{MN2014Willmore})
\[
\forall S, T \in \bI_k(M), \qquad \mathcal{F}(S - T) \leq \bF(S, T) \leq 2 \bM(S - T).
\]
\noindent We will also consider the spaces of \emph{relative $n$-cycles} in $M$ (see also \cite[Definition 1.20]{Almgren}). Let
    \[\ZZ_n(M,\partial M;\Z_2)=\{T \in \bI_n(M;\Z_2)\colon \mathrm{support}(\partial T)\subset \partial M\}\]
and define $T,S \in \ZZ_n(M,\partial M;\Z_2)$ to be equivalent if $T-S \in \bI_n(\partial M;\Z_2)$. We let $\Znr(M,\partial M;\Z_2)$ be the quotient space, and denote the equivalence class of a $T \in \ZZ_{n}(M,\partial M;\Z_2)$ by $P(T)$, or simply by $T$ when clear from context. The space of relative cycles inherit a \emph{flat} (semi-)norm and a \emph{mass norm} given by
    \[
    \mathcal{F}(P(T)) = \inf_{S \in P(T)} \mathcal{F}(S) = \inf\{\mathcal{F}(T+R) \colon R \in \bI_n(\partial M;\Z_2)\}
    \]
and
    \[
    \mathbf{M}(P(T))= \inf_{S \in P(T)}\mathbf{M}(S) = \inf\{\mathbf{M}(T+R) \colon R \in \bI_n(\partial M;\Z_2)\},
    \]
respectively. The spaces $\ZZ_n(M,\partial M;\Z_2)$ and $\Znr(M,\partial M;\Z_2)$ will be endowed with the topology induced by $\mathcal{F}$. We also note that, by \cite{LiZhou} and \cite[\S 3]{guang2021min}, each class $\tau \in \Znr(M,\partial M;\Z_2)$ has a unique \emph{canonical representative} $T \in \tau \subset \ZZ_n(M,\partial M;\Z_2)$ such that $T\lfloor \partial M = 0$, and it satisfies $\mathbf{M}(T)=\mathbf{M}(\tau)$.

In the case $\partial M = \emptyset$, Marques and Neves \cite{MarquesNevesMultiplicity} provided a short proof of the fact that $\mathcal{C}(M)$ is contractible (with respect to the flat metric), and the boundary map $\partial\colon \mathcal{C}(M) \to \ZZ_n(M;\Z_2)$ is a double cover. The analogous result when $\partial M \neq 0$ is used in \cite[Thm 4.7, Step II]{SunWangZhou}, but we include the proof from \cite{MarquesNevesMultiplicity} for completeness, adapting it to the case of relative cycles (see also \cite[Section 3.2]{LiZhou}). 
%
\begin{lemma} \label{topology_znr}
The boundary operator
    \[
    P\circ \partial \colon (\mathcal{C}(M),\mathcal{F}) \to \Znr(M,\partial M;\Z_2)
    \]
is a 2-sheeted covering map with contractible domain. Moreover, $\Znr(M,\partial M;\Z_2)$ is weakly homotopically equivalent to $\R \mathbb{P}^\infty$.
\end{lemma}

\begin{proof}
The contractibility of $(\mathcal{C}(M),\mathcal{F})\simeq \bI_{n+1}(M;\Z_2)$ is demonstrated as in \cite[Claim 5.3]{MarquesNevesMultiplicity}. We can choose a Morse function $f\colon M \to [0,1]$ so that each level set $f^{-1}(t)$ intersects $\partial M$ transversely, by choosing a collar neighborhood of $\partial M$ and using the parametric transversality Theorem to ensure this boundary condition.
This implies that the continuous map $\Phi\colon \R\mathbb{P}^\infty \to \ZZ_n(M,\partial M;\Z_2)$ defined by
    \[\Phi([a_0:a_1:\ldots:a_k:0\ldots]) = \partial\{x\in M \colon a_0+a_1f(x) + \ldots + a_kf(x)^k \leq 0\}\]
has the following property: for each $a \in \R\mathbb{P}^\infty$, the $n$-integral current $\Phi(a)$ is the canonical representative in its equivalence class in $\Znr(M,\partial M;\Z_2)$.

Given a continuous map $\Psi \colon I^p \to \Znr(M,\partial M;\Z_2)$ and $U_0 \in \mathcal{C}(M)$ such that $P(\partial U_0)=\Psi(0)$, we can construct a unique continuous lift $U \colon I^p \to \mathcal{C}(M)$ with $U(0)=U_0$ as in \cite[Claim 5.2]{MarquesNevesMultiplicity}. First, note that, by using the Constancy Theorem for mod 2 flat chains in both $\partial M$ and $M$, we see that if $P(\partial \Omega) = P(\partial \Omega')$ for $\Omega,\Omega'\in \mathcal{C}(M)$, then $\Omega - \Omega' \in \{0,M\}$. This observation implies that such lift (if exists) is unique.\medskip

In order to show the existence of a lift, we observe:\medskip

\noindent \textbf{Claim.} There exists $\epsilon_0>0$ and $\nu_0>0$ such that for any $\tau \in \Znr(M,\partial M;\mathbb{Z}_2)$ with $\mathcal{F}(\tau)<\epsilon_0$ there exists a unique $\Omega \in \mathcal{C}(M)$ such that $P(\partial \Omega)=\tau$ and $\mathbf{M}(\Omega) \leq \nu_0 \mathcal{F}(\tau)$. \smallskip

By the results of \cite{Almgren} adapted to the context of mod 2 cycles, for $X \in \{M,\partial M\}$, there are positive constants $\epsilon_X>0$ and $\nu_X>0$ such that for any $B_0 \in \mathcal{Z}_{\dim(X) - 1}(X;\mathbb{Z}_2)$ with $\mathcal{F}_X(B_0) < \epsilon_X$, there exists $W_0 \in \mathbf{I}_{\dim(X)}(X;\mathbb{Z}_2)$ with $\partial W_0 = B_0$ and $\mathbf{M}(W_0) \leq \nu_X \mathcal{F}_X(B_0)$.  We claim that $\epsilon_0 = \min\{\frac{\epsilon_M}{(1+\nu_{\partial M})}, \epsilon_{\partial M}\}$ and $\nu_0 =\nu_M(1 + \nu_{\partial M})$ have the desired property.

To see this, let $T_0 \in \mathcal{Z}_n(M,\partial M;\mathbb{Z}_2)$ be the canonical representative of $\tau$, so that 
    \[
    \mathcal{F}(T_0) = \mathcal{F}_M(T_0)=\mathcal{F}(\tau)<\epsilon_0< \epsilon_M.\]
We also have
    \[\mathcal{F}_{\partial M}(\partial T_0) \leq \mathcal{F}_M(\partial T_0) \leq \mathcal{F}_M(T_0),\]
where in the last inequality we used that if $T_0 = Q + \partial R$, then $\partial T_0 = \partial Q = Q_0 + \partial R_0$ for $Q_0=0$ and $R_0 = Q$, hence $\mathcal{F}_M(\partial T_0)\leq \mathbf{M}(Q_0) + \mathbf{M}(R_0) = \mathbf{M}(Q)\leq \mathbf{M}(Q) + \mathbf{M}(R)$. Thus $\mathcal{F}_{\partial M}(\partial T_0) < \epsilon_0 \leq \epsilon_{\partial M}$ and there exists $S_0 \in \mathbf{I}_{n}(\partial M;\mathbb{Z}_2)$ such that 
    \[
    \partial S_0 = \partial T_0 \quad  \text{and} \quad \mathbf{M}(S_0) \leq \nu_{\partial M} \mathcal{F}_{\partial M}(\partial T_0) < \nu_{\partial M} \epsilon_0.
    \] 
Using $\mathcal{F}_M(S_0-T_0) \leq \mathbf{M}(S_0) + \mathcal{F}_M(T_0) <(1+\nu_{\partial M}) \epsilon_0\leq \epsilon_M$, we conclude that there exists $\Omega \in \mathbf{I}_{n+1}(M;\mathbb{Z}_2)=\mathcal{C}(M)$ such that 
    \[
    \partial \Omega = S_0 - T_0 \quad \text{and} \quad \mathbf{M}(\Omega) \leq \nu_M \mathcal{F}(T_0-S_0)\leq \nu_M(1+\nu_{\partial M})\mathcal{F}_M(T_0)=\nu_0\mathcal{F}(\tau).
    \]
The uniqueness of such $\Omega$ follows as above, by the Constancy Theorem. Since $P(\partial \Omega) = P(S_0-T_0) = P(T_0) = \tau$ this proves the claim.

Now the existence of the lift follows as in \cite{MarquesNevesMultiplicity}, first in the case $p=1$ by using the claim in each interval of a sufficiently fine partition of $I$, and then for $p>1$ by simply connectedness.
\end{proof}

As also computed by Almgren in \cite{Almgren}, the previous Lemma shows that $\pi_1(\Znr(M,\partial M;\Z_2))=\Z_2$ and $\pi_i(\Znr(M;\Z_2))=0$, for $i\geq 2$. As a consequence, its $\Z_2$-cohomology ring $H^*(\Znr(M,\partial M;\Z_2),\Z_2)$ is isomorphic to $\Z_2[\bar \lambda]$, for the unique nonzero $\bar\lambda \in H^1(\Znr(M,\partial M;\Z_2),\Z_2)$.

We will follow the presentation of \cite{DeyWidths} for the \emph{volume spectrum} of $(M^{n+1},g)$ and work with topologically nontrivial maps $X \to \ZZ_n(M,\partial M;\Z_2)$ defined on \emph{simplicial complexes} $X$, instead of the standard construction where $X$ is assumed to be a cubical complex, namely a subcomplex of $[0,1]^N$. To follow \cite{SunWangZhou}, we also introduce the notation for cubical complex below. However, we observe that the results from \cite[Chapter 4]{BP} imply that it would be equivalent to require the domains of $p$-sweepouts $X \in \mathcal{P}_p$ to be cubical complexes. See \cite[\S 2.3]{DeyWidths}.

We say that a $\mathcal{F}$-continuous map $\Phi\colon X \to \Znr(M,\partial M;\Z_2)$ defined on a finite simplicial complex is a $p$-\emph{sweepout} if $\Phi^*(\bar\lambda^p) \neq 0 \in H^p(X,\Z_2)$. We will denote by $\mathcal{P}_p$ the space of all $p$-sweepouts which satisfy the \emph{no concentration of mass} condition:
    \[
    \lim_{r \to 0^+} \sup\{\|\Phi(x)\|(B_r(q)) \colon x \in X, q \in M\} = 0.
    \]
As noted in \cite[Lemma 2.7]{LiokumovichMarquesNeves}, any map into $\Znr(M,\partial M;\Z_2)$ which is continuous with respect to the mass norm has no concentration of mass.

The $p$\emph{-width} of $(M^{n+1},g)$ is the min-max number for the area (mass of $n$-cycles) associated to $\mathcal{P}_p$, that is
    \[
    \omega_p(M,g) = {\adjustlimits \inf_{\Phi \in \mathcal{P}_p} \sup_{x \in \mathrm{dom}(\Phi)}} \mathbf{M}(\Phi(x)).
    \]
The sequence $\{\omega_p(M,g)\}_p$ is called the \emph{volume spectrum} of $(M,g)$. As described in the introduction, the volume spectrum was first described in the work of Gromov \cite{Gromov} as a nonlinear analogue of the spectrum of the Laplacian, and its asymptotic growth played a crucial role in recent developments of the variational theory for minimal as well as constant and prescribed mean curvature surface. Closely related to the $p$-widths is $W_0 = W_0(M, g)$, the \emph{one parameter Almgren--Pitts width} \cite{Almgren,Pitts}, whose definition is recalled in \S \ref{APWidthBoundarySection}. 
\subsection{Applications of Min-Max to PMCs} \label{PMCMinMaxBackground}
We outline the essentials of the min-max construction for free boundary PMC's (labelled FBPMC's) as in Sun--Wang--Zhou, \cite[\S 3]{SunWangZhou}. All the setups here are the same as those in Zhou \cite[\S 1.1]{ZhouMultiplicity}. Let $(N^{n+1}, g)$ be a compact, oriented smooth manifold with boundary, $\partial N$, and $3 \leq n+1 \leq 7$.
%
%
\begin{definition}[\cite{SunWangZhou}, Defn 3.4] \label{HomotopyClassDef}
Let $X^k$ be a cubical complex of dimension $k\in \mathbb{N}$ in some $I(m,j)$ and $Z\subset X$ be a cubical subcomplex. 

Let $\Phi_0:X\rightarrow (\CC(N),\F)$ be a continuous map. Let $\Pi$ denote the collection of all sequences of continuous maps $\{\Phi_i:X\rightarrow\CC(N)\}_{i\in\mb N}$ such that
\begin{enumerate}
\item each $\Phi_i$ is homotopic to $\Phi_0$ in the $\mathbf{F}$-topology on $\CC(N)$;
\item there exist homotopy maps $\{\Psi_i:[0,1]\times X\rightarrow\CC(N)\}_{i\in\mb N}$ which are continuous in the $\mathbf{F}$ topology, $\Psi_i(0,\cdot)=\Phi_i$, $\Psi_i(1,\cdot)=\Phi_0$, and satisfy
\[\limsup_{i\rightarrow\infty}\sup\{\F(\Psi_i(t,x),\Phi_0(x)):t\in[0,1],x\in Z\}=0.\]
%
\end{enumerate}
\begin{remark}
The original homotopy maps $\Psi_i$ from \cite{SunWangZhou} are only defined to be continuous in the flat topology. However, as noted in \cite[Definition 3.4]{DeyCMCs}, the homotopies can actually be taken to be continuous in the $\mathbf{F}$-metric via \cite[Prop 1.14,1.15]{ZhouMultiplicity} and \cite[\S 3]{MarquesNevesMultiplicity}. 
\end{remark}
Given a pair $(X,Z)$ and $\Phi_0$ as above, $\{\Phi_i\}_{i\in \mb N}$ is called a {\em $(X,Z)$-homotopy sequence of mappings into $\CC(N)$}, and $\Pi$ is called the {\em $(X,Z)$-homotopy class of $\Phi_0$}. 
\end{definition}
\begin{definition}
We define the $h$-width by
\[\mf L^h=\mf L^h(\Pi):=\inf_{\{\Phi_i\}\in \Pi}\limsup_{i\rightarrow\infty}\sup_{x\in X}\{\AA^h(\Phi_i(x))\}.\]
\end{definition}
\noindent A sequence $\{\Phi_i\}_{i\in\mb N}\in \Pi$ is called a {\em minimizing sequence} if $\mf L^h(\{\Phi_i\})=\mf L^h(\Pi)$, where 
\[\mf L^h(\{\Phi_i\}):=\limsup_{i\rightarrow\infty}\sup_{x\in X}\{\AA^h(\Phi_i(x))\}.\]
\noindent Given $\Phi_0$ and $\Pi$, by the same argument as \cite[Lemma 1.5]{ZhouMultiplicity}, there exists a minimizing sequence.
\begin{definition}\label{def:critical set}
If $\{\Phi_i\}_{i\in\mb N}$ is a minimizing sequence in $\Pi$, the {\em critical set} of $\{\Phi_i\}$ is defined by 
\[\mf C(\{\Phi_i\})=\{V=\lim_{j\rightarrow\infty}|\partial \Phi_{i_j}(x_j)| \text{\ as varifolds : with\ }\lim_{j\rightarrow\infty}\AA^h(\Phi_{i_j}(x_j))=\mf L^h(\Pi)\}.\]
\end{definition}
%
\noindent We also recall Sun--Wang--Zhou's notion of  good prescribing functions on a manifold $(N^{n+1}, \partial N, g)$. Let $\mc S =\mc S(g)$ be the collection of all Morse functions $h$ such that the zero set $\Sigma_0 = \{h = 0\}$ is either empty, or is a compact, smoothly embedded hypersurface so that 
\begin{itemize}
\item $\Sigma_0$ is transverse to $\partial M$ and the mean curvature of $\Sigma_0$ vanishes to at most finite order;
\item $\{x\in\partial M:H_{\partial M}(x)=h(x)\text{ or } H_{\partial M}(x)=-h(x)\}$ is contained in an $(n-1)$-dimensional submanifold of $\partial M$. 
\end{itemize}
Note that $\mathcal{S}(g)$ differs from the larger class of functions, $\mathcal{S}^*(g)$, as defined in \S \ref{MainResultsSection}. 
\begin{lemma}[Lemma 2.3, \cite{SunWangZhou}]\label{lem:Sg contains open dense}
    $\mc S(g)$ contains an open and dense subset in $C^\infty(M)$.
\end{lemma}
\noindent We denote by $\mc P^h$ the collection of free boundary $h$-hypersurfaces such that $\llbracket\Sigma\rrbracket=\partial \Omega$ for some open set $\Omega\subset N$. We record the main min-max theorem of \S 3 in Sun--Wang--Zhou.
\begin{theorem}[Theorem 3.11,\cite{SunWangZhou}]\label{thm:index bound for all g}
Let $(N^{n+1},\partial N,g)$ be a compact Riemannian manifold of dimension $3\leq (n+1)\leq 7$, and $h\in \mc S(g)$ which satisfies $\int_N h\geq 0$. Given a $k$-dimensional cubical complex $X$ and a subcomplex $Z\subset X$, let $\Phi_0:X\rightarrow\CC(N)$ be a map continuous in the $\mf F$-topology, and $\Pi$ be the associated $(X,Z)$-homotopy class of $\Phi_0$. Suppose 
\begin{equation}\label{eq:relative constraint}
\mf L^h(\Pi)>\max\big\{\max_{x\in Z}\AA^h(\Phi_0(x)),0\big\}.
\end{equation}
Then there exists a nontrivial, smooth, compact, almost embedded hypersurface with free boundary $(\Sigma^n,\partial\Sigma)\subset (N,\partial N)$ , such that
\begin{itemize}
\item $\llbracket\Sigma\rrbracket =\partial \Omega$ for some $\Omega\in\C(N)$, where the mean curvature of $\Sigma$ with respect to the unit outer normal of $\Omega$ is $h$, i.e.
\[H|_\Sigma = h|_\Sigma ;\]
\item $\AA^h(\Omega) =\mf L^h(\Pi)$;
\item $\mathrm{index}_w(\Sigma) \leq k$.
\end{itemize}
\end{theorem}
\noindent Here, we recall that by ``almost-embedded" we mean 
\begin{definition} \label{almostEmbeddingDef}
Let $U \subseteq M^{n+1}$ an open subset and $\Sigma^n$ a smooth $n$-dimensional manifold. A smooth immersion $\phi: \Sigma \to U$ is an \textit{almost embedding} if at any point $p \in \phi(\Sigma)$ where $\Sigma$ fails to be embedded, then there exists a small neighborhood $W \subseteq U$ or $p$ such that 
\begin{itemize}
    \item $\Sigma \cap \phi^{-1}(W)$ is a disjoint union of connected components, $\cup_{i = 1}^{\ell} \Sigma_i$ 
    \item $\phi(\Sigma_i)$ is an embedding for each $i = 1, \dots, \ell$
    \item For each $i$, any other component $\phi(\Sigma_j)$, $j \neq i$ lies on one side of $\phi(\Sigma_i)$ in $W$
\end{itemize}
\end{definition}
\noindent See figure \ref{fig:mult2touch} for a visualization of almost embeddedness in the context of PMCs.
\begin{figure}[h!]
    \centering
    \includegraphics[scale=0.4]{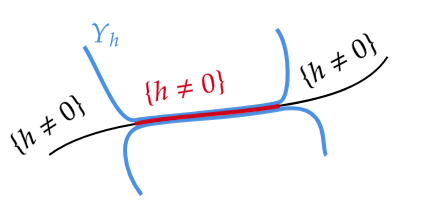}
    \caption{Almost embeddedness can occur on PMCs on large sets where $h \equiv 0$. See \cite[Remark 2.7]{zhou2019min}}
    \label{fig:mult2touch}
\end{figure}
\noindent We will often identify $\phi(\Sigma)$ with $\Sigma$ and $\phi(\Sigma_i)$ with $\Sigma_i$ when clear from context. The points in $\Sigma$ where $\Sigma$ fails to be embedded is the \textit{touching set}, $\mathcal{S}(\Sigma)$, and $\mathcal{R}(\Sigma) = \Sigma \backslash \mathcal{S}(\Sigma)$ is the regular set. In accordance with \cite[Prop 3.17]{ZhouZhu}, we also define a notion of ``almost embedded with optimal regularity"
\begin{definition} \label{almostEmbedOptimal}
We say that $\Sigma$ is ''almost embedded with optimal regularity" if $\mathcal{S}(\Sigma)$ is a contained in a countable union of connected, embedded $(n-1)$-dimensional submanifolds.
\end{definition}
\noindent Because a touching set occurs for a generic set of prescribing functions (\cite{ZhouZhu, ZhouMultiplicity}), we see that any PMC with a non-empty touching set will necessarily have points of density $2$. Despite this, we refer to a PMC, $Y \subseteq M$, as \textbf{\textit{multiplicity 1}} if each connected component of $Y$ is the image of a smooth immersion, $\phi: \Sigma \to M$, for $\Sigma$ a connected smooth n-dimensional manifold, such that $\phi$ is also an almost embedding. This allows for a multiplicity 1 PMC to have density $2$ points, in particular on open subsets of the hypersurface, see figure \ref{fig:dense2}. The notion of ``almost embedded with optimal regularity" is introduced to distinguish the case of when the set of points with density $2$ is ``small", e.g. $(n-1)$-dimensional. \nl 
\indent We also recall that the second variation of an $h$-PMC hypersurface $\Sigma=\partial \Omega$ along a normal vector field $\varphi \nu$ is given by 
%
\begin{equation} \label{SecondVariation}
\delta^2\mathcal{A}^h|_\Omega(X,X) = \int_{\Sigma} [|\n \varphi|^2 - (\Ric(\nu, \nu) + |A|^2 + \p_{\nu} h) \varphi^2] - \int_{\partial \Sigma} A^{\partial N}(\nu, \nu) \varphi^2
\end{equation}
where $A^{\partial N}$ is the second fundamental form of $\partial N$ and $A$ is the second fundamental form of $\Sigma$. However, the index is defined with respect to k-unstable deformations as in \cite[Definition 2.5]{SunWangZhou} - this is because free boundary $h$-hypersurfaces are a priori almost embedded. The notion of k-instability was introduced by Zhou \cite[Definition 2.1, 2.3]{ZhouMultiplicity} for closed $h$-hypersurfaces based on Marques-Neves \cite[Definition 4.1]{marques2015morse}.
	
	\begin{definition}\label{def:k-unstable}
		Given $\Sigma\in \mc P^h$ with $\Sigma=\partial \Omega$, $k\in\mb N$ and $\epsilon\geq 0$, we say that $\Sigma$ is {\em$k$-unstable in an $\epsilon$-neighborhood} if there exist $0<c_0<1$ and a smooth family $\{F_v\}_{v\in\oB^k}\subset \mathrm{Diff}(M)$ with $F_0=\Id, F_{-v}=F_v^{-1}$ for all $v\in\oB^k$ (the standard $k$-dimensional unit ball in $\mb R^k$) such that, for any $\Omega'\in\B^\F_{2\epsilon}(\Omega)$, the smooth function:
		\[\AA_{\Omega'}^h:\oB^k\rightarrow[0,+\infty),\ \ \ \ \AA^h_{\Omega'}(v)=\AA^h(F_v(\Omega'))\]
		satisfies
		\begin{itemize}
			\item $\AA^h_{\Omega'}$ has a unique maximum at $m(\Omega')\in B_{c_0/\sqrt {10}}^k(0)$;
			\item $-\frac{1}{c_0}\Id\leq D^2\AA^h_{\Omega'}(u)\leq -c_0\Id$ for all $u\in\oB^k$.
		\end{itemize} 
	\end{definition}
\noindent Since $\Sigma$ is a critical point of $\AA^h$, necessarily $m(\Omega)=0$.
	%
%
	\begin{definition}\label{def:Morse index}
		Assume that $\Sigma\in\mc P^h$. Given $k\in\mb N$, we say that {\em the weak Morse index of $\Sigma$ is bounded (from above) by $k$}, denoted as 
		\[\mathrm{index}_w{(\Sigma)}\leq k,\]
		if $\Sigma$ is not $j$-unstable in 0-neighborhood for any $j\geq k+1$. $\Sigma$ is said to be {\em weakly stable} if $\mathrm{index}_w(\Sigma)=0$.
	\end{definition}
\noindent We also recall the notion of the index of the regular set following Guang--Wang--Zhou \cite[\S 2.1]{guang2021compactness}. Given $\Sigma$ a two-sided, almost embedded hypersurface, let 
\[
C_c^{\infty}(\mathcal{R}(\Sigma)) = \{f \in C^{\infty}(\Sigma) \; : \; f \text{ vanishes in a neighborhood of } \mathcal{S}(\Sigma)\}
\]
We define $\mathrm{Index}(\mathcal{R}(\Sigma))$ to be the maximal dimension of a linear subspace of $C_c^{\infty}(\mathcal{R}(\Sigma))$ such that $\delta^2 \AA^h$ restricted to that subspace is negative definite. Clearly, if $\Sigma$ is almost embedded then 
\begin{equation} \label{weakIndexUpper}
\mathrm{Index}(\mathcal{R}(\Sigma)) \leq \mathrm{index}_w(\Sigma)    
\end{equation}
\noindent Given $\Lambda>0$ and $I\in\mb N$, let 
\begin{equation}\label{eq:def of Ph}
\mc P^h(\Lambda, I):=\{\Sigma\in\mc P^h: A(\Sigma)\leq \Lambda,\mathrm{index}_w(\Sigma)\leq I\}.
\end{equation}
We recall the following theorem of Sun--Wang--Zhou 
%
%
\begin{theorem}[Thm 2.9, \cite{SunWangZhou}]\label{thm:compactness for FPMC}
Let $(N^{n+1},\partial N,g)$ be a compact Riemannian manifold with boundary of dimension $3\leq (n+1)\leq 7$. Assume that $\{h_k\}_{k\in\mb N}$ is a sequence of smooth functions in $\mc S(g)$ such that $\lim_{k\rightarrow\infty}h_k=h_\infty$ in the smooth topology. Let $\{\Sigma_k\}_{k\in\mb N}$ be a sequence of hypersurfaces such that $\Sigma_k\in\mc P^{h_k}(\Lambda,I)$ for some fixed $\Lambda>0$ and $I\in \mb N$. Then,
\begin{enumerate}
\item \label{compactness thm:smooth limit} up to a subsequence, there exists a smooth, compact, almost embedded free boundary $h_\infty$-hypersurface $\Sigma_{\infty}$ such that $\Sigma_k\rightarrow\Sigma_\infty$ (possibly with integer multiplicity) in the varifold sense, and hence also in the Hausdorff distance by monotonicity formula;
\item \label{compactness thm:locally smoothly convergence} there exists a finite set of points $\mc Y\subset \Sigma_\infty$ with $\sharp(\mc Y)\leq I$, such that the convergence of $\Sigma_k\rightarrow\Sigma_\infty$ is locally smooth and graphical on $\Sigma_\infty\setminus \mc Y$;
\item \label{compactness thm:generic multiplicity one convergence} if $h_\infty\in\mc S(g)$, then the multiplicity of $\Sigma_\infty$ is 1, and $\Sigma_\infty\in\mc P^{h_\infty}(\Lambda,I)$;
\item \label{compactness thm:smooth and proper} assuming $\Sigma_k\neq \Sigma_\infty$ eventually and $h_k=h_\infty=h\in\mc S(g)$ for all $k$ and $\Sigma_k$ smoothly converges to $\Sigma_\infty$, then $\mc Y=\emptyset$, and $\Sigma_\infty$ has a non-trivial Jacobi field; 
\item \label{compactness thm:index decreasing}  if $h_k=h_\infty=h\in\mc S(g)$ and the convergence is not smooth, then $\mc Y$ is not empty and $\Sigma_\infty$ has strictly smaller weak Morse index than $\Sigma_k$ for all sufficiently large $k$;
\item \label{compactness thm:limit index bound} if $h_\infty\equiv0$ and $\Sigma_\infty$ is properly embedded, then the classical Morse index of $\Sigma_\infty$ satisfies $\mathrm{index}(\Sigma_\infty)\leq I$ (without counting multiplicity)
\item \label{compactness thm:weak index bound} if $\Sigma_{\infty}$ is multiplicity $1$, then $\Sigma_{\infty} \in \mc P^{h_{\infty}}(\Lambda, I)$.
\end{enumerate}
\end{theorem}
\begin{remark} 
The above theorem has been modified so that points \ref{compactness thm:smooth limit} and \ref{compactness thm:locally smoothly convergence} hold for any smooth $h_{\infty}$, as opposed to the originally stated $h_{\infty} \in \mathcal{S}(g)$ or $h_{\infty} \equiv 0$. Note that no extra work is needed for this extension, as it is a direct adaptation of the original proof of \ref{compactness thm:smooth limit} and \ref{compactness thm:locally smoothly convergence} from \cite{ZhouZhu} to the boundary case. See a related remark in \cite[Remark 2.7]{ZhouMultiplicity}.
\end{remark}
\begin{remark}
Point \ref{compactness thm:weak index bound} is not stated in the original theorem  \ref{thm:compactness for FPMC}, however it follows immediately as weak Morse index and area bounds are preserved under limits when the convergence happens with multiplicity $1$. See e.g. \cite[Thm 2.6, Part 2]{ZhouMultiplicity}. 
\end{remark}
\begin{remark} \label{C1 alpha convergence}
If $\partial N = \emptyset$ or, more generally, if $\Sigma_k$ are eventually contained in a compact set $K \subset N\setminus \partial N$, then the requirement that $h_k \to h_\infty$ in the smooth topology can be weakened to $\|h_k - h_\infty\|_{C^{1,\alpha}(K)}\to 0$. This is a consequence of the deep regularity and compactness results for prescribed mean curvature varifolds by Bellettini-Wickramasekera \cite{BWStable,BWPrescribed}. More precisely, by the regularity theory for PMCs with prescribing functions in $\mathcal{S}^*(g)$ \cite[Theorem 7.1]{ZhouZhu}, we see that the Caccioppoli sets $\Omega_k$ with $\partial\Omega_k = \llbracket\Sigma_k\rrbracket$ satisfy assumptions (a1) and (a2) in \cite[Theorem 1.5]{BWPrescribed}. Hence, \cite[Corollary 1.3]{BWPrescribed} applies (see also its proof in Section 8 of \cite{BWPrescribed}), proving parts \ref{compactness thm:smooth limit} and \ref{compactness thm:locally smoothly convergence} in Theorem \ref{thm:compactness for FPMC} above. Then part \ref{compactness thm:generic multiplicity one convergence} is proved as in \cite{ZhouMultiplicity}, namely using Theorem 3.19 in \cite{ZhouZhu} and convergence as Caccioppoli sets, and part \ref{compactness thm:weak index bound} follows as in the previous remark.
\end{remark}

\noindent We also record their theorem with changing ambient metrics on $(N,\partial N)$.
\begin{theorem}[Thm 2.10, \cite{SunWangZhou}]\label{thm:compactness with changing metrics}
Let $(N^{n+1},\partial N)$ be a closed manifold of dimension $3\leq (n + 1)\leq 7$, and $\{g_k \}_{k\in\mb N}$ be a sequence of metrics on $(N,\partial N)$ that converges smoothly to some limit metric $g$. Let $\{h_k\}_{k\in\mb N}$ be a sequence
of smooth functions with $h_k\in\mc S(g_k)$ that converges smoothly to some limit $h_\infty \in C^\infty (N)$. Let $\{\Sigma_k\}_{k\in\mb N}$ be a sequence of hypersurfaces with $\Sigma_k\in\mc P^{h_k}(\Lambda,I;g_k)$ for some fixed $\lambda > 0$ and $I\in\mb N$. Then there exists a smooth, compact, almost embedded free boundary $h_\infty$-hypersurface $\Sigma_\infty$, such that points \ref{compactness thm:smooth limit}\ref{compactness thm:locally smoothly convergence}\ref{compactness thm:generic multiplicity one convergence} \ref{compactness thm:weak index bound} from theorem \ref{thm:compactness for FPMC} are satisfied.
\end{theorem}
\noindent Again, we make the same change in theorem \ref{thm:compactness with changing metrics} to include all $h \in C^{\infty}(N)$, regardless of if $h \in \mathcal{S}(g)$ or $h \equiv 0$.
In the proof of Theorem \ref{mainTheorem}, we will construct almost embedded $h$-PMCs as limits of prescribed mean curvature boundaries produced via min-max. In order to obtain infinitely many such PMCs, we need to rule out multiple copies of the same hypersurfaces as potential limits, and this requires us to study the density of $h$-PMCs for functions $h$ which are not \emph{globally} good prescribing functions, but satisfy the Assumptions \ref{hCompactSupport} above. This can be achieved provided the limit contains no minimal components, as explained in the proof of the compactness result \cite[Theorem 3.19]{ZhouZhu}. We record this below, and sketch its proof for convenience.
\begin{theorem}[Thm 3.19, \cite{ZhouZhu}] \label{densityTheorem}
Under the same hypotheses from Theorem \ref{thm:compactness with changing metrics}, suppose that $\Sigma_k=\partial \Omega_k$ are closed boundaries and that the closed, almost embedded limit $h$-PMC, $\Sigma_\infty$, does not contain any minimal connected component. If $p \in\Sigma_\infty\setminus \mathcal{Y}$ and $h(p)\neq 0$, then the density of the varifold induced by the immersion $\Sigma_\infty \to M$ at $p$ is either $1$, if $p$ is a regular point, or $2$, if $p$ is a touching point. In general, the touching set has density $2$ and the regular set has density $1$.
\end{theorem}

\begin{proof}
By conclusions \ref{compactness thm:smooth limit} and \ref{compactness thm:locally smoothly convergence} in Theorems \ref{thm:compactness for FPMC} and \ref{thm:compactness with changing metrics} (for closed PMCs), we can choose a neighborhood $B_p$ of $p$ in which $\Sigma_\infty$ has a graphical decomposition $\bigcup_{i=1}^l \Sigma_\infty^i$, where each sheet is a $h_\infty$-PMC with respect to the normal vector $\nu_\infty^i$, $p \in \bigcap_{i=1}^l\Sigma_\infty^i$ and $l$ is the density at $p$.

Suppose $p$ is a touching point with $l> 2$. We will follow the proof of \cite[Theorem 3.19]{ZhouZhu} to conclude $h(p)=0$ and reach a contradiction. Since all sheets contain $p$, there are two of them with the same orientation, namely $\nu_\infty^i(p) = \nu_\infty^j(p)$ for some $1\leq i <j\leq l$. But then there must be a third sheet $\Sigma_\infty^m$ with $i<m<j$ and $\nu_\infty^m(p) = -\nu_\infty^j(p)$. This can be shown by using the Constancy Theorem \cite[Theorem 26.27]{simon1983lectures}, the graphical convergence of $\Sigma_k$, and the assumption $\Sigma_k = \partial \Omega_k$. On the other hand, since all sheets touch at $p$ and $\Sigma_\infty^i$ and $\Sigma_\infty^j$ have the same orientation, by the maximum principle (see \cite[Lemma 3.11]{ZhouZhu}) $\Sigma_\infty^i=\Sigma_\infty^j$. By the graphical decomposition, these coincide with $\Sigma_\infty^m$ as well, and shows that $\Sigma_\infty \cap B_p$ has mean curvature $h$ with respect to both orientations, so $\Sigma_\infty \cap B_p$ is minimal, and in particular $h(p)=0$.
\end{proof}
\subsubsection{The One Parameter Almgren--Pitts Width of a manifold with boundary} \label{APWidthBoundarySection}
As mentioned in \S \ref{MinMaxBackground}, we can define $W_0 = W_0(M, g)$, the \emph{one parameter Almgren--Pitts width} \cite{Almgren,Pitts}. Let $\mathcal{R}$ denote the set of all continuous maps $\sigma: [0,1] \to (\CC(M), \mathcal{F})$ such that $\sigma(0) = M$, $\sigma(1) = \emptyset$ and $\partial \circ \sigma: [0,1] \to \Znr(M, \partial M; \mathcal{F}; \Z_2)$ has no concentration of mass. Then 
\[
W_0(M, g) := \text{inf}_{\sigma \in \mathcal{R}} \sup_{t \in [0,1]} \mathbf{M}(\partial \sigma(t))
\]
$W_0$ is a min-max value associated to the fundamental class of $M$ under the isomorphism $H_{n+1}(M,\partial M;\Z_2) \simeq \pi_1(\Znr(M,\partial M;\Z_2))$. Concretely, if $\sigma \in \mathcal{R}$, then $\partial \circ \mathcal{R}$ is a closed path in $\Znr(M,\partial M;\Z_2)$ based at the zero cycle, and such a path lifts to an element of $\mathcal{R}$ if and only if it has no concentration of mass and it is homotopically nontrivial. \nl 
%
\indent To see that such a nontrivial path in $\Znr(M, \p M; \mathcal{F}; \Z_2)$ exists (and hence show that $\mathcal{R}$ is non-empty), we proceed with a similar idea as in lemma \ref{topology_znr}. Choose a morse function $f: M \to [0,1]$ so that the level sets $f^{-1}(t)$ intersects $\partial M$ transversely. Then the map
\begin{align*}
\gamma &: [0,1] \to \Znr(M, \partial M; \mathcal{F}; \Z_2) \\
\gamma(t) &= f^{-1}(t)
\end{align*}
is well defined and has no concentration of mass by \cite[Lemma 5.2]{MarquesNevesPositive}.
Moreover by lemma \ref{topology_znr}, $\gamma(t)$ lifts to a map 
\begin{align*}
\sigma&: [0,1] \to (\CC(M), \mathcal{F}) \\
\sigma(t) &= \{x \in M \; | \; f(x) < t\} \\
\partial \circ \sigma(t) &= \gamma(t)
\end{align*}
%
Thus $R \neq \emptyset$, and from the definition, $\omega_1(M,g) \leq W_0(M,g)$. 
\subsection{Gluing a cylindrical end} \label{construction}
In this section, we recall Song's construction of a cylindrical end \cite[\S 2.2]{song2018existence}, as the reader may not be familiar with the construction, and the details are essential to our work. \nl
\indent We provide a sketch as follows: suppose that $\Sigma = \partial M$ is minimal and each component of $\Sigma = \sqcup_{i = 1}^m \Sigma_i$ admits a \textit{contracting neighborhood}. This means that for each $i$, there exists a one-sided local foliation, $\{\Sigma_{i,t}\}$, such that the mean curvature vectors of $\Sigma_{i,t}$ point toward $\Sigma_i$.
%
See figure \ref{fig:contractingngbd} for a visualization. 
\begin{figure}[h!]
	\centering
	\includegraphics[scale=0.4]{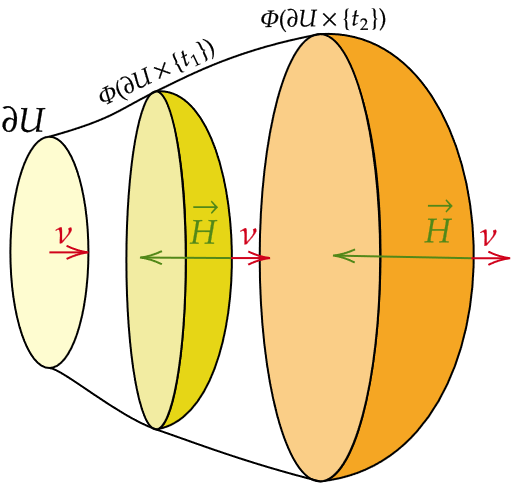}
	\caption{Visualization of the contracting neighborhood near our strictly stable minimal surface $\Sigma = \partial U$}
	\label{fig:contractingngbd}
\end{figure}
We will be interested in the non-compact manifold
\[
\Cyl(M) = M \sqcup_{\{\Sigma_i\}_{i=1}^m} (\Sigma_i \times \R^{\geq 0})
\]
where the metric on $\Sigma_i \times \R^{\geq 0}$ is given by $h_{i} + dt^2$ where $h_i = g \Big|_{T \Sigma_i}$ (see figure \ref{fig:cylindricalend}). We will often suppress the $i$ index, noting that the construction can be done component by component. In this decomposition, we may refer to $M = \text{Core}(M) \subseteq \Cyl(M)$ as the ``core" of $\Cyl(M)$.
\begin{figure}[h!]
	\centering
	\includegraphics[scale=0.3]{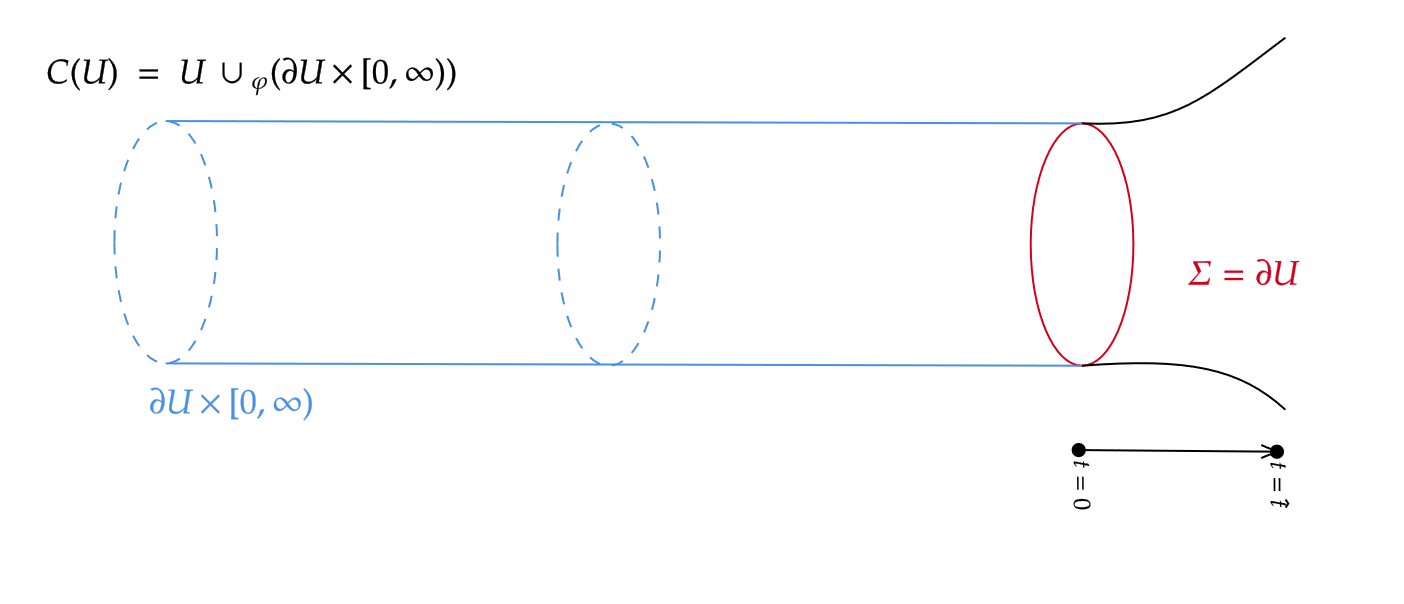}
	\caption{Visualization of limit metric and manifold, $\Cyl(U)$}
	\label{fig:cylindricalend}
\end{figure}
Roughly speaking, Song constructs a manifold, $(U_{\eps}, g_{\eps})$, such that $(U_{\eps}, g_{\eps}) \xrightarrow{\eps \to 0} \Cyl(M)$ in $C^{\infty}_{loc}$ outside a small ``folding region" (see \cite[Lemma 6]{song2018existence}). Despite the low regularity of the convergence, many geometric quantities (such as the p-widths) still converge. We present the formal details below (with new visualizations), which are taken from Song \cite[\S 2.2]{song2018existence} with only minor notational changes: \nl
%
\indent Let $(U,g)$ be a smooth, connected compact Riemannian manifold with boundary. Suppose that $\partial U$ is a minimal and a neighborhood of $\partial U$ in $U$ is smoothly foliated by closed leaves whose mean curvature vectors are pointing towards $\partial U$, i.e. there exists a diffeomorphism
\[
\Phi : \partial U \times [0,\hat{t}] \to U
\]
where $\Phi(\partial U \times \{0\}) = \partial U$ is minimal, and for all $t\in(0,\hat{t}]$, the leaf $\Phi(\partial U \times \{t\})$ has mean curvature vector pointing towards $\partial U$. By convention, $\Phi(\partial U \times \{t\})$ has positive mean curvature with respect to the normal vector in the direction given by $\frac{\partial}{\partial t}$. We record the following well known lemma which provides a quantitative description of the mean curvature of the leaves when $\partial U$ is strictly stable:
\begin{lemma} \label{quantMCLemma}
For $\partial U = \Sigma$ a strictly stable minimal surface, there exists a diffeomorphism  $\Phi: \Sigma \times [0, \hat{t}] \to U$ and $\hat{t}$ sufficiently small such that $H_{\Phi(\Sigma, t)} \geq C t$ for all $t \in [0,t]$
\end{lemma}
\noindent The lemma is implicit from \cite[Lemma 11]{song2018existence} and follows by letting $\Phi(s,t) = \exp_{\Sigma}(\phi_0(s) t)$ for $t$ sufficiently small, where $\phi_0(s)$ is the first eigenfunction of the Jacobi operator (and hence positive everywhere).  \nl 
\indent Let $\varphi: \partial U \times \{0\} \to \partial U$ be the canonical identification. Define the following non-compact manifold with cylindrical ends:
$$\Cyl(U):= U \cup_\varphi (\partial U \times [0,\infty)).$$
We endow it with the metric $h\Big|_U = g$ and $h=g\llcorner{\partial U} \oplus ds^2$ on $\p U \times [0, \infty)$ and denote this metric as $g_{\Cyl}$. Here $g\llcorner{\partial U}$ is the restriction of $g$ to the tangent bundle of the boundary $\partial U$ and $g\llcorner{\partial U} \oplus ds^2$ is the product metric on $\partial U \times [0,\infty)$. Note that the metric $h$ may only be Lipschitz continuous.

Next, we define for any small $\epsilon>0$ a compact Riemannian manifold with boundary, $(U_\epsilon,g_\epsilon)$, diffeomorphic to $U$ and converging to $(\Cyl(U),g_{\Cyl})$ as $\epsilon\to 0$ in a sense to be defined later. For $0<\epsilon<\hat{t}$, define  
\[
\tilde{U}_\epsilon = U \backslash \Phi(\partial U \times [0,\epsilon))
\]
with boundary $\partial \tilde{U}_\epsilon= \Phi(\partial U \times \{\epsilon\})$. For a small positive number $\delta_{\epsilon}>0$, the map
\begin{align*}
\tilde{\gamma}_\epsilon &: \partial \tilde{U}_\epsilon \times [-\delta_{\epsilon},0] \to M \\
\tilde{\gamma}_\epsilon(x,t) &= \exp({x}, t \nu)
\end{align*}
is well-defined and gives Fermi coordinates on one side of $\partial \tilde{U}_\epsilon$. We take $\delta_{\epsilon}>0$ small so that 
\begin{itemize}
\item $\lim_{\epsilon\to 0} \delta_\epsilon=0,$
\item $\tilde{\gamma}_\epsilon(\partial \tilde{U}_\epsilon \times [-\delta_{\epsilon},0]) \subset \Phi(\partial U \times [0,\epsilon])$, 
\item and for all $t\in [-\delta_{\epsilon},0]$, the hypersurface $\tilde{\gamma}_\epsilon( \partial \tilde{U}_\epsilon \times\{t\})$ has positive mean curvature. (The normal vector is in the direction of $\nu$.) 
\end{itemize}
For $0 < \epsilon<\hat{t}$, we define
$${U}_\epsilon = \tilde{U}_\epsilon \cup \tilde{\gamma}_\epsilon (\partial \tilde{U}_\epsilon \times [-\delta_{\epsilon},0])$$
where by convention $U_0 = U$. \nl
\indent For a positive number $\hat{d}>0$ small enough and fixed independently of $\epsilon\in[0,\hat{t})$, consider the Fermi coordinates on a $\hat{d}$-neighborhood of $\partial {U}_\epsilon$: 
$$\gamma_\epsilon : \partial {U}_\epsilon \times [0,\hat{d}] \to U$$
$$\gamma_\epsilon(x,t) = \exp({x}, t \nu)$$
where $\exp$ is the exponential map with respect to $g$, $\nu$ is the inward unit normal of $\partial {U}_\epsilon$.
We remark that for all $s\in [0,\delta_\epsilon]$, 
$$\gamma_\epsilon (\partial {U}_\epsilon\times \{s\}) = \tilde\gamma_\epsilon (\partial \tilde{U}_\epsilon\times \{-\delta_\epsilon+s\}).$$ For each small positive $\epsilon$, choose a smooth function $\vartheta_\epsilon:[0,\delta_\epsilon] \to \mathbb{R}$ and $z_\epsilon\in(0,\delta_\epsilon)$ with the following properties:
\begin{enumerate}
\item \label{enum:vepsProps} $1\leq \vartheta_\epsilon$ and $\frac{\partial}{\partial t}\vartheta_\epsilon \leq 0$,
\item $\vartheta_\epsilon\equiv1$ in a neighborhood of $\delta_\epsilon$,
\item \label{enum:constzeps} $\vartheta_\epsilon$ is constant on $[0,z_\epsilon]$,
\item \label{enum:bigInt} $\lim_{\epsilon\to 0}\int_{[0,\delta_\epsilon]}\vartheta_\epsilon =\infty$,
\item \label{enum:foldingRegionSmall} $\lim_{\epsilon\to 0}\int_{[z_\epsilon,\delta_\epsilon]}\vartheta_\epsilon =0$.
\end{enumerate}
This function naturally induces a function on $U_\epsilon$ still called $\vartheta_\epsilon$, defined by
$$\vartheta_\epsilon(\gamma_\epsilon(x,t)) = \vartheta_\epsilon(t) \text{ for all } (x,t)\in \partial U_\epsilon \times [0,\delta_\epsilon]$$
and extended continuously by $1$.
The original metric $g$ can be written in the Fermi coordinates $\gamma_\epsilon$ as $g_t\oplus dt^2$. Now, by abuse of notation, define the following smooth metric $g_\epsilon$ on $U_\epsilon$:
\begin{equation} \label{modifiedMetric}
g_\epsilon(q)= \Big\{ \begin{array}{rcl}
g_t(q) \oplus (\vartheta_\epsilon (q) dt)^2 & \text{ for $q\in \gamma_\epsilon(\partial U_\epsilon \times [0,\delta_\epsilon])$} \\
g(q) & \text{ for $q\in U_\epsilon\backslash \gamma_\epsilon(\partial U_\epsilon \times [0,\delta_\epsilon])$} 
\end{array}
\end{equation}
This gives a compact manifold with boundary $(U_\epsilon, g_\epsilon)$ - see figure \ref{fig:piecewisemetric} below for a  visualization.
\begin{figure}[h!]
\centering
\includegraphics[scale=0.4]{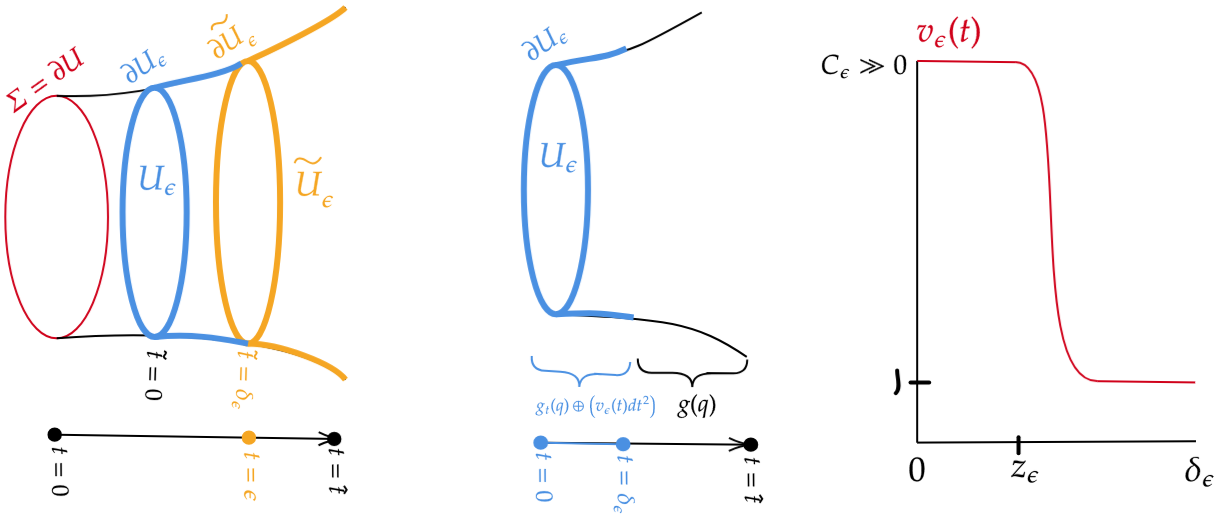}
\caption{Visualization of the metrics $g_{\eps}$ in terms of scaling function $v_{\eps}(t)$}
\label{fig:piecewisemetric}
\end{figure}
We recall the following relevant lemmas from Song:
\begin{lemme}[Lemma 4, \cite{song2018existence}] \label{meancurv}
Let $(U_\epsilon,g_\epsilon)$ be defined as above. By abuse of notations, consider $\partial U_\epsilon \times [0,\delta_{\epsilon}]$ as a subset of $U_\epsilon$ via $\gamma_\epsilon$. Then for $t\in [0,\delta_{\epsilon}]$, the slices $\partial U_\epsilon \times \{t\}$ satisfy the following with respect to the new metric $h_\epsilon$:
\begin{enumerate}
\item they have non-zero mean curvature vector pointing in the direction of $-\frac{\partial}{\partial t}$,
\item their mean curvature goes uniformly to $0$ as $\epsilon$ converges to $0$,
\item their second fundamental form is bounded by a constant $C$ independent of $\epsilon$.
\end{enumerate}
\end{lemme}
%
%
%
\noindent In figure \ref{fig:blowup}, we visualize the change of coordinates and one choice of $v_{\eps}(t)$.
\begin{figure}[h!]
	\centering
	\includegraphics[scale=0.3]{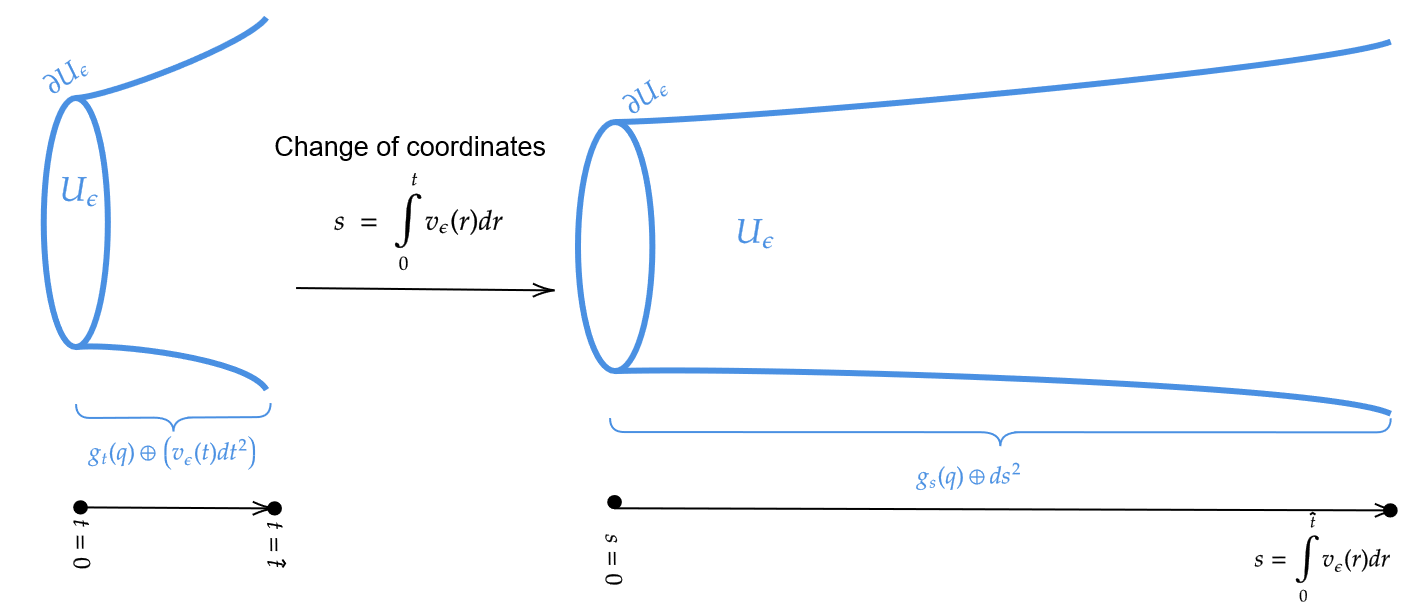}
	\caption{Change of coordinates $t \to s$ so that the contracting neighborhood approximates a cylindrical end}
	\label{fig:blowup}
\end{figure}
In figure \ref{fig:endvisualized}, we visualize a graph of the mean curvature with respect to the $s$ coordinate.
\begin{figure}[h!]
	\centering
	\includegraphics[scale=0.4]{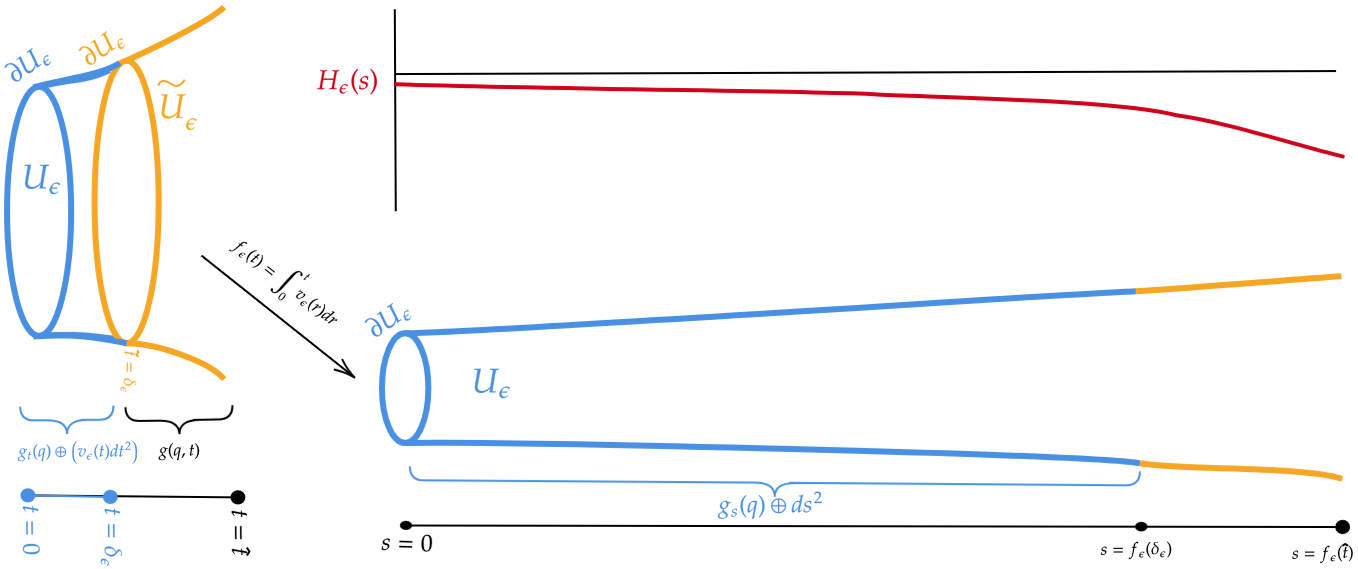}
	\caption{Visualization of the end with a graph of the mean curvature of the slices}
	\label{fig:endvisualized}
\end{figure}
%
%
\noindent We recall \cite[Lemma 6]{song2018existence} which shows that $(U_\epsilon,g_\epsilon)$ converges to the non-compact manifold with cylindrical ends, $\Cyl(U)$, and the convergence is smooth outside of a folding region where the curvature can be unbounded. 
%
%
%
\begin{lemme}[Lemma 6, \cite{song2018existence}] \label{geomconv}
The sequence of Riemannian manifolds $(U_\epsilon,g_\epsilon,q)$ converges geometrically to the non-compact manifold $(\Cyl(U),g_{\Cyl},q)$ in the $C^0$ topology. Moreover the geometric convergence is smooth outside of $\partial U\subset \Cyl(U)$ in the following sense.
\begin{enumerate}
\item Let $q\in U\backslash \partial U$. For small $\epsilon$, we have $q\in U_\epsilon\backslash\gamma_\epsilon(\partial U_\epsilon\times [0,\delta_\epsilon))$. Then as $\epsilon\to 0$,
$$(U_\epsilon\backslash\gamma_\epsilon(\partial U_\epsilon\times [0,\delta_\epsilon]), g_\epsilon,q)$$ converges geometrically to $(U\backslash \partial U,g,q)$ in the $C^\infty$ topology.
\item Fix any connected component $C$ of $\partial U$; for $\epsilon$ small we can choose a component $C_\epsilon$ of $\partial U_\epsilon$ so that $C_\epsilon$ converges to $C$ as $\epsilon\to 0$ and the following holds. Let $\epsilon_k>0$ be a sequence converging to $0$. For all $k$ let $q_k \in \gamma_\epsilon(C_{\epsilon_k}\times [0,\delta_{\epsilon_k}))$ be a point at fixed distance $d'>0$ from $\gamma_\epsilon(C_{\epsilon_k}\times \{\delta_{\epsilon_k}\})$ for the metric $g_{\epsilon_k}$, $d'$ being independent of $k$. Then 
$$\big(\gamma_{\epsilon_k}(C_{\epsilon_k}\times [0,\delta_{\epsilon_k})),g_{\epsilon_k},q_k\big)$$
subsequently converges geometrically to $(C\times (-\infty,0),g_{prod}, q_\infty)$ in the $C^\infty$ topology, where $g_{prod}$ is the product of the restriction of $g$ to $C$ and the standard metric on $(-\infty,0)$, and $q_\infty$ is a point of $C\times (-\infty,0)$ at distance $d'$ from $C\times \{0\}$.
\end{enumerate}
\end{lemme}
%
%
We also recall \cite[Lemma 7]{song2018existence}, which  describes how the folding region of $(U_\epsilon,g_\epsilon)$ is controlled in the $C^1$ topology by $g$. In particular, we emphasize the existence of charts near the folding region of $U_{\eps}$, such that the metrics, $\{g_{\eps}\}$, remain bounded in the $C^1$ topology \textit{despite} converging to a metric which is only lipschitz. We use $\llcorner$ to denote the restriction of a metric $g$ to a submanifold. 
\begin{lemme}[Lemma 7, \cite{song2018existence}] \label{norm}
For any $\frak{d}_1\in(0,\hat{d})$, there exists $\eta>0$ 
such that for all $\epsilon>0$ small, there is an embbeding $\theta_\epsilon: \partial U\times [-\frak{d}_1,\frak{d}_1]\to U_\epsilon$ satisfying the following properties: 
\begin{enumerate}
\item $\theta_\epsilon(\partial  U\times \{0\}) = \gamma_\epsilon(\partial U_\epsilon\times\{\delta_\epsilon\})$ and 
$$\theta_\epsilon(\partial U\times [-\frak{d}_1,\frak{d}_1]) = \{q\in U_\epsilon; \quad d_{h_\epsilon}\big(q,\gamma_\epsilon(\partial U_\epsilon\times\{\delta_\epsilon\})\big)\leq \frak{d}_1\},$$
\item  $\|\theta_\epsilon^*g_\epsilon \|_{C^1(\partial U\times [-\frak{d}_1,\frak{d}_1])} \leq \eta$ where $\|.\|_{C^1(\partial U\times [-\frak{d}_1,\frak{d}_1])}$ is computed with the product metric $h':=g\llcorner \partial U  \oplus ds^2$,
\item the metrics $\theta_\epsilon^*g_\epsilon$ converge in the $C^0$ topology to a Lipschitz continuous metric which is smooth outside of $\partial U\times\{0\}\subset\partial U\times [-\frak{d}_1,\frak{d}_1]$ and 
$$\lim_{\epsilon \to 0}\|\theta_\epsilon^*g_\epsilon\llcorner (\partial U\times [0,\frak{d}_1]) - \gamma_0^*g\llcorner (\partial U\times [0,\frak{d}_1]) \|_{C^0({\partial U\times [0,\frak{d}_1]})}=0$$
where $\|.\|_{C^0(\partial U\times [0,\frak{d}_1])}$ is computed with $h'$.
\end{enumerate}
\end{lemme}
\noindent Having constructed the cylindrical end, we recall the linear growth of $\omega_p(\Cyl(U), h)$ and the convergence of $\omega_p(U_{\eps}, h_{\eps}) \to \omega_p(\Cyl(U), h)$. In the below, we decompose $\Sigma$ into its connected components, i.e. $\Sigma  = \sqcup_{i=1}^m \Sigma_i$,
%
\begin{theorem}[Thm 9 \cite{song2018existence} ] \label{CylindricalWeylThm}
Let $(C, h)$ be an $n+1$-dimensional connected non-compact manifold with a finite number of cylindrical ends, each of which is isometric to a product metric, $(\Sigma_i \times [0, \infty), h_i \oplus dt^2)$. Suppose that $\Sigma_1$ has the largest $n$-volume among $\{\Sigma_i\}_{i=1}^m$, then %
\begin{enumerate}
    \item $\forall p \in \N^+$,
    \[
    \omega_{p+1}(C) - \omega_p(C) \geq A(\Sigma_1)
    \]
    \item There exists $\hat{C} > 0$ such that for all $p$,
    \[
    p \cdot A(\Sigma_1) \leq \omega_p(C) \leq p \cdot A(\Sigma_1) + \hat{C} p^{1/(n+1)}
    \]
\end{enumerate}
\end{theorem}
\begin{remark} \label{ConvergenceOfWidths}
For each $p$, $\omega_p(U_{\eps}, g_{\eps}) \xrightarrow{\eps \to 0} \omega_p(C, g_{\Cyl} = h \oplus dt^2)$ (see \cite[Thm 10, Step 2]{song2018existence}). 
\end{remark}
%
%
\subsubsection{Bounds on $W_0(U_{\eps}, g_{\eps})$} \label{APWidthBoundSection}
%
Let $\eta>0$ arbitrary and choose $t_1 \in (0,\hat t)$ such that the manifold with smooth boundary $M_1 = M \setminus \Phi(\Sigma \times [0,t_1))$ satisfies    

\[
\sup_{t \in [0,t_1]}|A(\Sigma) - A(\Phi(\Sigma\times\{t\}))|<\eta \qquad \text{and}\qquad |W_0(M,g) - W_0(M_1,g)|<\eta.
\]
Note that the latter condition follows from the continuity of $W_0$ with respect to the metric and the fact that $M$ and $M_1$ are diffeomorphic via a map $\Phi_{t_1}: M \to M_1$ such that $||\Phi_{t_1}^*(g) - g||_{C^1}$ can be made arbitrarily small as $t_1 \to 0$. Recall \cite[Remark 5]{song2018existence} that $\partial U_\epsilon$ has a neighborhood in $U_\epsilon$ that is foliated by the hypersurfaces
\[
\{\Phi(\partial U\times \{t\})\}_{t \in [\epsilon, \hat t]} \cup \{\gamma_\epsilon(\partial U_\epsilon\times\{t\})\}_{t\in [0,\delta_\epsilon]}
\]
and that these satisfy $\Phi(\partial U \times \{\epsilon\} ) = \gamma_\epsilon (\partial U_\epsilon\times \{\delta_\epsilon\})$, and  $U_\epsilon \setminus \left(\,\gamma_\epsilon(\partial U_\epsilon \times [0,\delta_\epsilon) \cup \Phi(\partial U \times [\epsilon,t_1))\,\right) = M_1$. 

Let $\Xi \colon [0,1] \to \mathcal{C}(M,g)$ be a flat-continuous map with $\Xi(0) = M$, $\Xi(1)=\emptyset$, and such that $\partial \circ \Xi \colon [0,1] \to \mathcal{Z}_{n,rel}(M,\partial M;\Z_2)$ has no concentration of mass and

\[
\sup_{t \in [0,1]} \mathbf{M}(\partial\Xi(t)) \leq W_0(M,g) +\eta,
\]
where the mass is computed with respect to the original metric $g$. We can use $\Xi$ to construct a sweepout of $(U_\epsilon,g_\epsilon)$ by restricting $\Xi(t)$ to $M_1$ and then using the foliation of $\partial U_\eps$ to complete the sweepout of $U_{\eps}$. More precisely, we define $\tilde \Xi \colon [\epsilon-t_1-\delta_\epsilon,1] \to \mathcal{C}(U_\epsilon,g_\epsilon)$ by

\[
\tilde\Xi(t) = \left\{\begin{array}{rl} U_\epsilon\setminus (\, \gamma_\epsilon(\partial U_\epsilon \times [0,t-\epsilon+t_1+\delta_\epsilon)\,)\,) , & t \in [\epsilon-t_1-\delta_\epsilon,\epsilon -t_1], \\[5pt] U_\epsilon\setminus (\, \gamma_\epsilon(\partial U_\epsilon \times [0,\delta_\epsilon)\cup\Phi(\partial U\times[\epsilon,t+t_1))\,)\,),& t\in[\epsilon-t_1,0],\\[5pt] \Xi(t) \cap M_1, & t\in[0,1].  \end{array} \right.
\]
Since $\Xi(1)=M \supset M_1$ and the domains above depend continuously on $t$ in the flat metric, the map $\tilde \Xi$ is continuous. By a direct computation,
\[
\mathbf{M}(\partial \tilde\Xi(t)) = \left\{\begin{array}{rl} \mathbf{M}(\,\gamma_\epsilon(\partial U_\epsilon \times \{t-\epsilon+t_1+\delta_\epsilon\})\,), & t \in [\epsilon-t_1-\delta_\epsilon,\epsilon -t_1], \\[5pt] \mathbf{M}(\,\Phi(\Sigma \times \{t+t_1\})\,),& t\in[\epsilon-t_1,0],\\[5pt] P_{(U_\epsilon,g_\epsilon)}( \Xi(t) \cap  M_1), & t\in[0,1],  \end{array} \right.
\]
where $P_{(U_\eps,g_\eps)}$ denotes the perimeter functional in $(U_\eps,g_\eps)$. Recalling that $g_\epsilon=g$ on $U_\epsilon \setminus \gamma_\epsilon(\partial U_\epsilon \times [0,\delta_\epsilon]) \supset (\partial M_1) \cup \mathrm{supp}(\partial \Xi(t))$  and \cite[Proposition 4.1]{DeyCMCs}, we see that $\partial\circ \tilde \Xi$ has no concentration of mass, and estimate the mass of $\partial \tilde\Xi(t)$ with respect to $g_\epsilon$ by:
\[
\sup_{t \in [\epsilon-t_1-\delta_\epsilon,1]}\mathbf{M}(\partial \tilde\Xi(t)) \leq \max \left\{\sup_{s \in [0,\delta_\epsilon]}A_{g_\epsilon}(\gamma_\epsilon(\partial U_\epsilon \times\{s\})),\sup_{s \in [\epsilon,t_1]}A_{g}(\Phi(\Sigma\times\{s\})), W_0(M,g)+ A(\Sigma) + 2\eta \right\}
\]
The second supremum is bounded above by $A(\Sigma) + \eta$, by the choice of $t_1$. Moreover, using the geometric convergence of $(U_\epsilon,g_\epsilon)$ to $\mathrm{Cyl}(M)$ (see Lemma \ref{geomconv}, part 2), the first supremum can be bounded from above by $A(\Sigma) + o_\epsilon(1)$, where $A(\Sigma)=\mathrm{Area}_g(\Sigma)$. Therefore,
\[
\sup_{t \in [\epsilon-t_1-\delta_\epsilon,1]}\mathbf{M}(\partial \tilde\Xi(t))  \leq W_0(M,g) + A(\Sigma) + 2\eta + o_\epsilon(1).
\]
Since $\tilde \Xi(\epsilon-t_1-\delta_\epsilon) =U_\epsilon$ and $\tilde \Xi(1) = \Xi(1)\cap M_1=\emptyset$, we conclude that

\[
W_0(U_\epsilon,g_\epsilon) \leq \sup_t\mathbf{M}(\partial\tilde\Xi(t))\leq W_0(M,g) + A(\Sigma) + 2\eta + o_\epsilon(1)
\]
for sufficiently small $\epsilon$, and thus
\begin{equation} \label{W0epsUpperBound}
\limsup_{\epsilon \to 0^+} W_0(U_\epsilon,g_\epsilon) \leq  W_0(M,g) + A(\Sigma).
\end{equation}
%
\subsection{Approximating sequences $h_{\eps} \to h$}
Having constructed the approximating manifolds $(U_{\eps}, g_{\eps})$ to $\Cyl(M)$, we consider a sequence of approximating functions $h_{\eps}: U_{\eps} \to \R$. Recall that we assume that $h$ satisfies Assumption \ref{hCompactSupport}.

%
For $N = U_{\eps}$ in the above, there exists a sequence of functions of approximating functions, $\{h_{\eps}\}$, such that 
\begin{enumerate}
\item \label{hEpsSMall} $\|h_{\eps}\|_{L^1(U_{\eps})} < A(\Sigma_1)/2$, $\|h_{\eps}\|_{L^{\infty}(U_{\eps})} \leq 2 \|h\|_{L^{\infty}(M)}$
\item \label{hEpsGood} $h_{\eps} \in \mc S(U_{\eps}, g_{\eps})$ 
\item \label{endBehavior} $h_{\eps} \to 0$ in $C^{1}(\gamma_{\eps}(\p U_{\eps} \times [0, \delta_{\eps}]))$ and $h_{\eps} \rightarrow h$ in $C^1_{loc}(M\setminus \Sigma)$ as $\eps \to 0$ for all $\eps$ sufficiently small. 
\item \label{integralCondition} $\int_{U_{\eps}} h_{\eps} \geq 0$
\end{enumerate}
%
%
Note that such an $h_{\eps}$ always exists, as $h \in C^1_c(M)$ implies that $\gamma_{\eps}(\partial U_{\eps} \times [0, \delta_{\eps}]) \subseteq \{h = 0\}$ for all $\eps$ sufficiently small. Thus conditions \ref{hEpsGood} and \ref{endBehavior} follow from the density of good functions, $\mathcal{S}(U_\eps,g_{\eps})$, in order to take a sufficiently small perturbation of $h$. Note that conditions \ref{hEpsSMall} and \ref{integralCondition} are also possible due to the density of good functions, as well as the original conditions on $h$, \ref{LOneBound} and \ref{compactPositiveIntegral}.

\subsection{Dey's Suspension Construction}  \label{DeySuspension}
In this section, we briefly recall the topological suspension construction of Dey \cite[2.4]{DeyCMCs}. \nl 
\indent Let $A$ be a topological space and define the \textit{cone over $A$} as
\[
CA = \frac{A \times [0,1]}{\sim}
\]
where `$\sim$' collapes $A \times \{1\}$ to a point, i.e.
\[
(x,1) \sim (y,1) \qquad \forall x,y \in A
\]
The \textit{suspension of $A$} is given by 
\[
SA = \frac{A \times [-1,1]}{\sim}
\]
where $\sim$ collapses $A \times \{1\}$ to a point and $A \times \{-1\}$ to another point, so that $SA = C_+ A \cup C_- A$. We note that if $A$ is path-connected, then $SA$ is simply connected, and hence if $A$ is a cubical complex then so are $CA$ and $SA$. \nl  
\indent Let $h_{\eps} \in \mathcal{S}(g_{\eps})$ as imposed by condition \ref{hEpsGood}. We want to apply theorem \ref{thm:index bound for all g} to the homotopy class associated to the following map constructed in \cite[equation 4]{DeyCMCs}. We recall this construction as follows, reproducing \cite[\S 4]{DeyCMCs}, Parts 1--4 with brevity, in the setting of manifolds with boundary.\nl 
\indent Fix $\delta > 0$. There exists a $p$-sweepout $\Phi: X \to \Znr(U_\eps, \partial U_\eps; \Z_2)$ with no concentration of mass and $X$ connected so that 
\begin{equation} \label{initialSweepout}
\sup_{x \in X} \{\bM(\Phi(x))\} \leq \omega_p + \delta
\end{equation}
by nature of being a sweepout, we have that 
\[
(\Phi)_*: \pi_1(X) \rightarrow \pi_1(\Znr(U_\eps, \partial U_\eps; \Z_2)) = \Z_2
\]
is onto. Thus, we can lift $\Phi$ to a map $\tPhi: \tX \to \CC(U_\eps)$ with $\rho: \tX \to X$ the double cover projection map, and $\p \circ \tPhi = \Phi \circ \rho$, where we omit the quotient map onto relative cycles. Let $T: \tX \to \tX$ the corresponding deck transformation with 
\begin{equation} \label{tPhiSym}
\tPhi(T(x)) = U_\eps - \tPhi(x) \qquad \forall x \in \tX
\end{equation}
Finally, let $\Lambda: [0,1] \to \CC(U_\eps)$ be a $1$-sweepout with $\Lambda(0) = U_\eps$, $\Lambda(1) = \emptyset$ and $\partial \circ \Lambda: [0,1] \to \ZZ_n(U_\eps, \partial U_\eps; \Z_2)$ having no concentration mass and 
\[
\sup_{t \in [0,1]} (\bM(\p \Lambda(t))) \leq W_0(U_\eps,g_\eps) + \delta,
\]
Define the subset 
\[
\tilde{Z} = \{(x,t) \in S \tX \; : \; |t| \leq 1/2\} 
\]
As in \cite[ \S 4, part 2]{DeyCMCs}, we can consider the following map on $S \tX$ 
\begin{align*} 
\tPsi &: S \tX \rightarrow \CC(U_\eps) \\
\tilde{\Psi}(x, t) &=\begin{cases}
\tilde{\Phi}(x) & (x,t) \in \tilde{Z} \\
(\tilde{\Phi}(x) \cap \Lambda(2t-1)) & (x,t) \in C_+ \backslash \tilde{Z} \\
U_\eps - (\tilde{\Phi}(x) \cap \Lambda(-1-2t)) & (x,t) \in C_- \backslash \tilde{Z}
\end{cases}
\end{align*}
We then have the following directly from \cite{DeyCMCs} adapted to the setting of manifolds with boundary 
%
\begin{lemma}[Claim 4.2, \cite{DeyCMCs}]
The map $\tilde{\Psi}: S \tX \to \CC(U_\eps)$ is continuous in the flat topology. 
\end{lemma}
%
%
\noindent Noting the $\Z_2$ action on $S \tX$ given by $(x,t) \mapsto (T(x), -t)$ along with the fact that $\partial \tPsi(x,t) = \p \tPsi(T(x), -t)$, we can define 
\[
Y = S \tX / ((x,t) \sim (T(x), -t)
\]
and note that $\tPsi$ descends to a map $\Psi: Y \to \Znr(U_\eps, \p U_\eps; \Z_2)$. By \cite[Claim 4.3]{DeyCMCs}, $\Psi$ has no concentration of mass, so the interpolation theorems of \cite[Theorem 13.1, 14.1]{MN2014Willmore} and \cite[Theorem 3.9, 3.10]{MarquesNevesPositive}, tells us there exists a map $\Psi': Y \to \Znr(U_{\eps}, \p U_{\eps}; \Z_2)$ which is continuous and the $\mathbf{F}$ topology and 
\[
\bM(\Psi'(y)) \leq \bM(\Psi(y)) + \delta
\]
We also note the following, which again follows from Dey's original work with minor adaptations to the setting of manifolds with boundary: 
%
\begin{lemma}[Claim 4.3 \cite{DeyCMCs}] \label{lem:PsiPrimeSweep}
$\Psi': Y \to \Znr(U_\eps, \partial U_\eps; \Z_2)$ is a $p+1$ sweepout with no concentration of mass.
\end{lemma}
\noindent We can also lift $\Psi'$ to a map on Caccioppoli sets, $\tPsi': S \tilde{X} \to \CC(U_{\eps})$. We conclude this section by noting that the $(p+1)$-sweepout constructed here satisfies the following uniform mass upper bound:
\begin{align} \nonumber
\sup_{y \in Y} \mathbf{M}(\Psi'(y)) & \leq \ \sup_{y \in Y} \mathbf{M}(\Psi(y)) + \delta \\ \nonumber
&= \sup_{(x,t) \in S\tilde X} \mathbf{M}(\partial \tilde\Psi(x,t)) + \delta\\ \nonumber
& \leq \ \sup_{x \in X} \mathbf{M}(\Phi(x)) + \sup_{t \in [0,1]}\mathbf{M}(\partial \Lambda(t)) +\delta \\ \nonumber
& \leq \omega_p(U_\eps,g_\eps) + W_0(U_\eps,g_\eps) + 2\delta \\[4pt] \label{PsiPrimeUpper}
& \leq \omega_p(\mathrm{Cyl}(M)) + W_0(M,g) + A(\Sigma) + 2\delta + o_\eps(1) 
\end{align}
having used \cite[Prop 4.1]{DeyCMCs} in the third line to bound $\mathbf{M}(\partial \tilde{\Psi}(x,t))$ and equation \eqref{W0epsUpperBound} in the fourth line to bound $W_0(U_{\eps}, g_{\eps})$.
\section{Proof of Theorem \ref{mainTheorem}}
We begin the proof of theorem \ref{mainTheorem} by constructing our approximate PMCs, $Y_{\eps,p}$, on the cylindrical approximation, $(U_{\eps}, g_{\eps})$.
\subsection{Construction of $Y_{\eps,p}$} \label{YepsConstruction}
With Dey's sweepout construction, we can now apply the min-max theorem \ref{thm:index bound for all g} to the manifold $(U_{\eps}, g_{\eps})$ with prescribing function $h_{\eps}$. In the below, we abbreviate $\omega_{p,\eps} = \omega_{p}(U_{\eps}, g_{\eps})$. Consider the homotopy class of $\Pi = \Pi(S \tX, \tilde{Z})$ of $\tPsi'$. For our choice of $\tZ$, we have 
\begin{align*}
\sup_{z \in \tZ} \AA^{h_\eps}(\tPsi'(z)) &= \sup_{x \in X} \sup_{t \in [0,1/2)} \AA^{h_\eps}(\tPsi'(x,y,t)) \\
& \leq \sup_{x \in X} \sup_{t \in [0, 1/2)} \mathbf{M}(\p \tPsi'(x,y,t)) + \int_M |h_{\eps}| \\
& \leq \sup_{x \in X} \sup_{t \in [0, 1/2)} \mathbf{M}(\p \tPsi(x,y,t)) + \int_M |h_{\eps}| + \delta \\
& = \sup_{x \in X} \AA^{h_\eps}(\tPhi(x)) + \int_M |h_{\eps}| + \delta \\
& \leq \omega_{p,\eps} + \int_M |{h_\eps}| + \delta 
\end{align*}
We now aim to show that $\mathbf{L}^{h_\eps}(\Pi) > \sup_{z \in \tZ} \AA^{h_\eps}(\tPsi'(z))$. Recalling the definition \ref{HomotopyClassDef}, suppose we have a sequence $\{\tPsi_i: S\tX  \to (\CC(U_{\eps}), \bF)\}$ along with homotopies $G_i: S\tX  \times [0,1] \to (\CC(U_{\eps}), \bF)$ such that $G_i(\cdot, 0) = \tPsi', \; G_i(\cdot, 1) = \tPsi_i$, and 
\[
\limsup_{i \to \infty} \sup \{\bF(G_i(z, s), \tPsi'(z)) \; : \; z \in Z, \;\; s \in [0,1]\} = 0
\]
Now define the map 
\begin{equation} \label{modifiedSeq}
\tilde{\Psi}_i^*(x,t) = \begin{cases}
\tPsi'(x) & t =0 \\
G_i((x,t), 2t) & t \in [0, 1/2] \\
\tilde{\Psi}_i(x,t) & t \in [1/2,1] \\
U_{\eps} - G_i((T(x),-t), -2t) & t \in [-1/2,0] \\
U_{\eps} - \tPsi_i(T(x), -t) & t \in [-1,-1/2]
\end{cases}
\end{equation}
Note the continuity at $t = 0$ due the symmetry from equation \eqref{tPhiSym}, thus $\tilde{\Psi}_i^*$ descends to a map, $\Psi_i^*: Y \to \Znr(U_{\eps},\partial U_\eps; \Z_2)$. Moreover, $\tPsi_i$ is $\bF$-continuous as $\tPsi', G_i, \Lambda$ are all individually $\bF$-continuous. 
Finally, $\Psi^*_i$ is homotopic to $\Psi$ in the $\bF$ topology by considering the map 
\[
\tilde{H}_i^*((x,t),s) = \begin{cases}
\tilde{\Psi}'(x) & t =0 \\
G_i((x,t), 2ts) & t \in [0, 1/2] \\
G_i((x,t), s) & t \in [1/2,1] \\
U_{\eps} - G_i((T(x),-2t), -2ts) & t \in [-1/2,0] \\
U_{\eps} - G_i((T(x), -t), s) & t \in [-1,-1/2]
\end{cases} 
\]
which homotopies $\tilde{\Psi}'$ to $\tilde{\Psi}_i^*$. Since $\Psi:Y \to \Znr(U_{\eps}; \Z_2)$ (the corresponding quotiented map from $\tilde{\Psi}$) is homotopic to $\Psi'$ (the corresponding quotiented map from $\tilde{\Psi}'$) from \S \ref{DeySuspension}, we see that the induced map $\Psi_i^*: Y \to \Znr(U_{\eps}; \Z_2)$ (coming from quotienting $\tilde{\Psi}_i^*$) is also a $(p+1)$-sweepout. Thus, we have 
%
\begin{align} \label{bulkComp}
\sup_{(x,t) \in S\tX} \AA^{h_\eps}(\tPsi_i^*(x,t)) &= \sup_{(x,t) \in S \tilde{X}} \bM(\p \tPsi_i^*(x,t)) - \int_{\tPsi_i^*(x,t)} h_{\eps} d \mathcal{H}^{n+1} \\ \nonumber
&= \sup_{(x,t) \in Y} \bM(\Psi_i^*(x,t)) -  \int_{\tPsi_i^*(x,t)} h_{\eps} d \mathcal{H}^{n+1} \\ \nonumber
& \geq \omega_{p+1,\eps} - \int_{U_{\eps}} |h_{\eps}|
\end{align}
Note that the definition of $G_i$ and $\Z_2$ action of $\tPsi'(x,t) = U_{\eps} - \tPsi'(T(x), -t)$  
gives
\begin{align*}
\limsup_i \sup \{\bF(G_i((x,t), 2t), \tPsi'(x,t)) \; : \; t \in [0,1/2]\} &= 0 \\
\limsup_i \sup \{\bF(U_{\eps} - G_i((T(x),-t), -2t), \tPsi'(x,t)) \; : \; t \in [-1/2,0]\} &= 0
\end{align*}
which implies that for $i$ sufficiently large
\begin{align} \label{ZComp}
\sup_{\substack{(x,t) \in S\tX \\ \st t \in [-1/2,1/2]}} \bM(\p \tPsi_i^*(x,t)) &\leq \sup_{(x,t) \in Z} \bM(\p\tPsi'(x,t)) + \delta \\ \nonumber
& \leq \sup_{(x,t) \in Z} \bM(\p\tPsi(x,t)) + 2\delta \\ \nonumber
& = \sup_{x \in X} \bM(\Phi(x)) + 2\delta \\ \nonumber
& \leq \omega_{p,\eps} + 3\delta \\ \nonumber 
\implies \sup_{\substack{(x,t) \in S\tX \\ \st t \in [-1/2,1/2]}} \AA^{h_\eps}(\tPsi_i^*(x,t)) & \leq \omega_{p,\eps} + 3\delta + \int_{U_{\eps}} |h_{\eps}|
\end{align}
By theorem \ref{CylindricalWeylThm}, the convergence of $\omega_p(U_{\eps}, g_{\eps})$ to $\omega_p(C, h)$, and for our choice of prescribing function, $h_{\eps}$, (see condition \ref{hEpsSMall}) we get that for $p$ fixed and $\eps$ sufficiently small, we have that
\begin{align*}
\omega_{p+1,\eps} - \|h_{\eps}\|_{L^1(U_{\eps}, g_{\eps})} & > \omega_{p,\eps} + 3\delta + \|h_{\eps}\|_{L^1(U_{\eps}, g_{\eps})} 
\end{align*}
Thus, in order for \eqref{bulkComp} to hold, i.e. 
\[
\sup_{(x,t) \in S\tX} \AA^{h_\eps}(\tPsi_i^*(x,t)) \geq \omega_{p+1,\eps} - \|h_{\eps}\|_{L^1}
\]
it's clear that 
\[
\sup_{\substack{(x,t) \in S\tX \\\st |t| \geq 1/2 }} \AA^{h_\eps}(\tPsi_i^*(x,t)) \geq \omega_{p+1,\eps} - \|h_{\eps}\|_{L^1}
\]
must hold. Now since $\tPsi_i^*$ equals $\tPsi_i$ on $t \in [-1,-1/2] \cup [1/2,1] $, the computations in \eqref{bulkComp} \eqref{ZComp} imply that for $i$ sufficiently large
\begin{align} \label{SweepDiff}
\sup_{(x,t) \in S\tX} \AA^{h_\eps}(\tPsi_i) & \geq \sup_{\substack{(x,t) \in S\tX \\ |t| \geq 1/2}} \AA^{h_\eps}(\tPsi_i(x,t)) \\ \nonumber
& \geq \omega_{p+1,\eps} - \|h_{\eps}\|_{L^1} \\ \nonumber
& > \omega_{p,\eps} + \|h_{\eps}\|_{L^1} + C_0 + 3\delta\\ \nonumber
& \geq \sup_{z \in Z} \AA^{h_\eps}(\tPsi'(z)) + C_0
\end{align}
where $C_0 > 0$ is independent of $\delta$ and $\eps$ for all $\eps$ sufficiently small. This follows by condition \ref{LOneBound}, choosing $\delta$ sufficiently small, the convergence of $\|h_{\eps}\|_{L^1(U_{\eps},g_{\eps})} \to \|h\|_{L^1(M^+, g)}$ (which follows from condition \ref{endBehavior}), the convergence of the widths, $\omega_{p, \eps} \to \omega_p(\Cyl(M))$ via remark \ref{ConvergenceOfWidths}, and the linear growth of the widths from theorem \ref{CylindricalWeylThm}. Since
this holds for an arbitrary sequence of $\{\Psi_i: S\tX \to (\CC(U_{\eps}), \bF)\}$, we see that 
\[
\mathbf{L}^{h_\eps}(\Pi(S\tX , Z, \tPsi')) > \sup_{z \in Z} \AA^{h_\eps}(\tPsi'(z))
\]
Applying theorem \ref{thm:index bound for all g}, we conclude the existence of an embedded, multiplicity $1$, $h_{\eps}$-Free Boundary PMC on $(U_{\eps}, g_{\eps})$ with optimal regularity and $\mathrm{index}_w \leq p+1$. We label the PMC surface $Y_{\eps,p}$, the underlying Cacciopoli set, $\Omega_{\eps, p}$, and record that by theorem \ref{thm:index bound for all g}, 
\begin{align} \nonumber
\mathcal{A}^{h_\eps}(\Omega_{\eps, p}) &= \mathbf{L}^{h_\eps}(\Pi(S\tX , Z, \tPsi')) \\ \nonumber 
& \geq \omega_{p+1,\eps} - \|h_{\eps}\|_{L^1(U_{\eps}, g_{\eps})} \\ \label{AreaLowerBound}
\implies A(Y_{\eps, p}) & \geq \omega_{p+1, \eps} - 2 \|h_{\eps}\|_{L^1(U_{\eps}, g_{\eps})},
\end{align}
and by equation \eqref{PsiPrimeUpper}
\begin{align} \nonumber
\mathcal{A}^{h_\eps}(\Omega_{\eps, p}) & = \mathbf{L}^{h_\eps}(\Pi(S\tX , Z, \tPsi)) \leq \omega_{p,\eps} + W_0(M,g) + A(\Sigma) + \|h_\eps\|_{L^1(U_\eps,g_\eps)} + 2\delta + o_\eps(1) \\ \label{YepsAreaBounds}
 \implies A(Y_{\eps,p}) & \leq \omega_{p,\eps} + W_0(M,g) + A(\Sigma) + 2\|h_\eps\|_{L^1(U_\eps,g_\eps)} + 2\delta + o_\eps(1)
\end{align}
\subsection{Diameter estimates for $Y_{\eps,p}$}
In this section, we recall the recent work of Chambers and the second author \cite{chambers2024} on the diameter of submanifolds. Their main theorem is 
\begin{theorem}[Thm 2, \cite{chambers2024}] \label{CJ:diameterEstimate}
Suppose $Y^m \subseteq N^{n+1}$ with $\partial Y$ connected and the following hold
\begin{enumerate}
    \item The sectional curvature of $N$ is bounded above globally, i.e. there exists $k_0 > 0$ such that $K_N \leq k_0$.
    \item The injectivity radius of $N$ is bounded below, i.e. there exists $r_0 > 0$ such that $\text{Inj Rad}(N) \geq r_0$.
\end{enumerate}
Then the intrinsic diameter of $Y$ is bounded by 
\[
d_{int}(M) \leq C(k_0, r_0, m) \left( \int_Y |H|^{m-1} + \text{max}(\text{Vol}(Y), \text{Vol}(M)^{1/m}) \right) + \text{d}_{int}(\partial M)
\]
\end{theorem}
We would like to apply theorem \ref{CJ:diameterEstimate} to show that $d_{int}(Y_{\eps,p})$ is bounded independently of $\eps$. However, there are two issues with applying this theorem. First, our ambient manifold $(U_{\eps}, g_{\eps})$ does \textit{not} have sectional curvature bounds uniform in $\eps$, as the sectional curvature is becoming unbounded in norm near the ``folding region." Second, $Y_{\eps,p}$ has boundary and a priori, we have no control on the diameter of the boundary for min-max PMCs. To avoid this problem, we will apply \cite[Lemma 8]{chambers2024} which bounds the distance of an interior point to the boundary:
\begin{lemma}[Lemma 8, \cite{chambers2024}] \label{lem:distToBoundary}
Let $P^m$ a smooth submanifold with boundary of $(N^{n+1}, g)$ a smooth manifold (with or without boundary) such that the sectional curvature is bounded above, $K_N \leq k_0$, and the injectivity radius is bounded below, $\text{Inj Rad}(N) \geq r_0 > 0$. Let $x \in \mathring{P}$ and consider a distance minimizing curve in $P$ from $x$ to $\partial P$. Then 
\[
\ell(\gamma) \leq C(m, k_0, r_0) \left(\int_P |H|^{m-1} + \max (\text{Vol}(P), \text{Vol}(P)^{1/m}) \right)
\]
\end{lemma}
\noindent We will now apply lemma \ref{lem:distToBoundary} to a part of $Y_{\eps,p}$. For $P \subseteq M$, let $d_{ext}(P)$ denote the \textit{extrinsic diameter}. We aim to prove the following:
%
%
\begin{proposition} \label{prop:extDiamBound}
Suppose $P^m \subseteq (U_{\eps}, g_{\eps})$ is a connected submanifold with $\text{Vol}(P) \leq K$ and $|H| \leq C$ for some $C, K > 0$ and $\partial P^m \subseteq \partial U_{\eps}$ (potentially empty). Then there exists $\alpha = \alpha(C,K)$, independent of $\eps$ such that $\text{d}_{ext}(P) \leq \alpha$.
\end{proposition}
\begin{proof}
It suffices to bound the external diameter of each connected component of $P \cap \{t \leq z_{\eps}\}$ and $P \cap \{t > z_{\eps}\}$. Note that by nature of condition \ref{enum:foldingRegionSmall}, the region $U_{\eps} \cap \{t > z_{\eps}\}$ automatically has finite external diameter. In particular 
\begin{align} \nonumber
\text{d}_{ext}(U_{\eps} \cap \{t > z_{\eps}\}, h_{\eps}) & \leq \sum_{i=1}^p\text{d}_{ext}([U_{\eps} \cap \{z_{\eps} \leq t \leq \delta_{\eps}\}]_i, h_{\eps}) + \text{d}_{ext}(U_{\eps} \cap \{t > \delta_{\eps}\}, h_{\eps})\\ \nonumber
& \leq \sum_{i=1}^p \text{d}_{ext}([U_{\eps} \cap \{z_{\eps} \leq t \leq \delta_{\eps}\}]_i, h_{\eps}) + \text{diam}_g(M) \\ \label{boundaryBound}
&\leq \sum_{i=1}^p \sup_{t \in [0, \hat{t}]} \text{diam}([\partial U_{\eps}]_i, g_t) + \int_{z_{\eps}}^{\delta_{\eps}}v_{\eps}(r) dr + \text{diam}_g(M) \\ \label{niceRegionBound}
& \leq K + o_{\eps}(1) + \text{diam}_g(M)
\end{align}
Here, 
\[
U_{\eps} \cap \{z_{\eps} \leq t \leq \delta_{\eps}\} = \sqcup_{i = 1}^P [U_{\eps} \cap \{z_{\eps} \leq t \leq \delta_{\eps}\}]_i
\]
is the manifolds decomposition into connected components and in the first line, it suffices to bound the external diameter of each of these components, and the external diameter of $U_{\eps} \cap \{t > \delta_{\eps}\}$. Similarly, $\partial U_{\eps} = \sqcup_{i = 1}^p [\partial U_{\eps}]_i$. In line \eqref{boundaryBound}, we bound the diameter of 
\[
[U_{\eps} \cap \{z_{\eps} \leq t \leq \delta_{\eps}\}]_i = [\partial U_{\eps}]_i \times [z_{\eps}, \delta_{\eps}] 
\]
with the metric 
\[
h_{\eps} = g_t \oplus (v_{\eps}(t) dt)^2
\]
by the sum of the diameters of the $[\partial U_{\eps}]_i$ fibre (uniformly bounded independent of $\eps$) and the length of geodesics in the $t$ direction (tending to $0$ by \ref{enum:foldingRegionSmall}). The final bound in equation \eqref{niceRegionBound} is clearly uniform in $\eps$. Thus, the same uniform bound holds for $P \cap \{t > z_{\eps}\}$. \nl 
\indent Let $P_{\eps} := P \cap \{t \leq z_{\eps}\} = \sqcup_{i=1}^{p'} [P_{\eps}]_i$ (again, decomposition into connected components), and WLOG, we assume it is non-empty. We consider $P_{\eps}$ as a submanifold of $N_{\eps} := U_{\eps} \cap \{t \leq z_{\eps}\}$ and note that while $N_{\eps}$ has boundary, the metric has uniform sectional curvature bounds in the coordinates $(x, s)$ where $s(t) = \int_0^t v_{\eps}(r) dr$. This follows from condition \ref{enum:constzeps} on the metric $h_{\eps}$, as
\begin{align*}
h_{\eps} \Big|_{\gamma_{\eps}(\p U_{\eps} \times [0, z_{\eps}])} &= g_{t(s)} \oplus ds^2 \\
t \in [0, z_{\eps}] & \implies s(t) = t \cdot C_{\eps}
\end{align*}
for some constant $C_{\eps}$ tending to infinity as $\eps \to 0$. This follows from the defining properties of $v_{\eps}(r)$ \ref{enum:vepsProps} \ref{enum:constzeps}\ref{enum:bigInt} and see figure \ref{fig:piecewisemetric}. We compute
\begin{align*}
\p_s g_{t(s)} &= t'(s) (\p_t g_t) \Big|_{t = t(s)} \\
\p_s^2 g_{t(s)} &= t''(s) (\p_t g_t) \Big|_{t = t(s)} + t'(s) (\p_t^2 g_t) \Big|_{t = t(s)}
\end{align*}
Note that 
\begin{align*}
t'(s) &= \frac{1}{s'(t)} = \frac{1}{v_{\eps}(t)} \\
\implies |t'(s)| & \leq 1
\end{align*}
By condition \ref{enum:vepsProps}. Similarly
\begin{align*}
t''(s) &= - \frac{s''(t)}{s'(t)^2} \\
t \in [0, z_{\eps}] & \implies v_{\eps}'(t) = 0 \\
\implies s''(t) &= v_{\eps}'(t) = 0 \\
\implies t''(s) &= 0
\end{align*}
So we conclude that 
\[
\p_s^2 g_{t(s)} \Big|_{t \in [0, z_{\eps}]} = t'(s) (p_t^2 g_t) 
\]
Since $g_t$ is the restriction of the original metric on $M$ to the level set $\partial U_{\eps} \times [0, \eps]$, it has bounded $t$ derivatives and bounded $\p_{x_i}$ derivatives as well. Thus $h_{\eps}$ has bounded $C^2$ estimates \textit{independent of $\eps$}, and so $N_{\eps}$ has bounded sectional curvature, i.e. $K_{N_{\eps}} \leq k_0$, for $k_0$ independent of $\eps$. Note that this in turn gives us a uniform lower bound on the injectivity radius of $N_{\eps}$, i.e. $(\text{Inj Rad})(N_{\eps}) \geq r_0(k_0) > 0$. If $P_{\eps}$ has no boundary, then we can apply theorem \ref{CJ:diameterEstimate} and conclude 
\begin{align*}
d_{ext}(P_{\eps}) &\leq d_{int}(P_{\eps}) \\ 
&\leq C(k_0, m) \left[ \int_{P_{\eps}} |H|^{m-1} + \max \left( \mathcal{H}^m(P_{\eps}), \mathcal{H}^m(P_{\eps})^{1/m}\right) \right] \\
& \leq C(k_0,m) \cdot [C^{m-1} + 1] \sup(K, K^{1/m}) = \alpha 
\end{align*}
If $\partial P_{\eps} \neq \emptyset$, we bound each component $[P_{\eps}]_i$ individually. Since 
\[
U_{\eps} \cap \{0 \leq t \leq z_{\eps}\} \cong \partial U_{\eps} \times [0, z_{\eps}] \cong \Sigma \times [0, z_{\eps}] = \sqcup_{i = 1}^m \Sigma_i \times [0, z_{\eps}]
\]
we know that for $\eps$ sufficiently small, each $[P_{\eps}]_i$ will lie inside some connected component of $U_{\eps} \cong \Sigma_j \times [0, z_{\eps}]$. In the following argument, we will suppress the ``$i$" and ``$j$" subindices. We apply lemma \ref{lem:distToBoundary} as follows: Let $x,y \in P_{\eps}$ such that 
\[
d_{ext}(P_{\eps}) = \text{dist}_{g_{\eps}}(x,y) 
\]
Note that $\partial P_{\eps} \subseteq (\p U_{\eps} \times \{0\} \sqcup (\p U_{\eps} \times \{z_{\eps}\})$ by assumption on $P$. Let $C_1 = \partial P_{\eps} \cap \{t = 0\}$, $C_2 = \partial P_{\eps} \cap \{t = z_{\eps}\}$ and note that even though $P_{\eps}$ is connected, $C_1$ and $C_2$ may have multiple comoponents. For some $i_x,i_y \in \{1,2\}$, we have $\dist_{P_{\eps}}(x, \partial P_{\eps}) = \dist_{P_{\eps}}(x, C_{i_x})$ and $\dist_{P_{\eps}}(y, \partial P_{\eps}) = \dist_{P_{\eps}}(y, C_{i_y})$. Choose a point $w \in P_{\eps}$ such that 
\[
L_w = \dist_{P_{\eps}}(w, C_1) = \dist_{P_{\eps}}(w, C_2)
\]
And note by lemma \ref{lem:distToBoundary}, we have 
\begin{align*}
L_w &\leq C(k_0, m) \left[ \int_{P_{\eps}} |H|^{m-1} + \max \left( \mathcal{H}^m(P_{\eps}), \mathcal{H}^m(P_{\eps})^{1/m}\right) \right] \\
& \leq C(k_0,m) \cdot [C^{m-1} + 1] \sup(K, K^{1/m}) = \alpha 
\end{align*}
independent of $\eps$. We now construct a competitor curve to the distance minimizing geodesic (in $U_{\eps}$) from $x \to y$ based on whether $i_x = i_y$ or $i_x \neq i_y$:
\newline 
\textbf{Case 1: $i_x = i_y$} \newline 
Suppose $x,y$ are closer to $C_1$ than $C_2$. Let $\tilde{x}, \tilde{y} \in C_1$ such that 
\begin{align*}
\dist_{P_{\eps}}(x, \tilde{x}) &= \dist_{P_{\eps}}(x, C_1) = \dist_{P_{\eps}}(x, \partial P_{\eps}) \\
\dist_{P_{\eps}}(y, \tilde{y}) &= \dist_{P_{\eps}}(y, C_1) = \dist_{P_{\eps}}(y, \partial P_{\eps})
\end{align*}
Consider the competitor curve to the distance minimizing geodesic (in $U_{\eps}$) from $x \to y$: $x \to \tilde{x}$ (minimizing inside of $P_{\eps}$), $\tilde{x} \to \tilde{y}$ (minimizing inside of $\partial U_{\eps} \times \{0\}$, assumed to be connected), $\tilde{y} \to y$ (minimizing inside of $P_{\eps}$). Then 
\begin{align} \nonumber
\text{dist}_{g_{\eps}}(x,y) &\leq \text{dist}_{P_{\eps}}(x, \tilde{x}) + \text{dist}_{\partial U_{\eps} \times \{0\}}(\tilde{x}, \tilde{y}) + \text{dist}_{P_{\eps}}(\tilde{y}, y) \\ \nonumber
& = \dist_{P_{\eps}}(x, \partial P_{\eps}) + \text{dist}_{\partial U_{\eps} \times \{0\}}(\tilde{x}, \tilde{y}) + \dist_{P_{\eps}}(y, \partial P_{\eps}) \\ \nonumber
& \leq 2 \cdot  C(k_0, m) \left[ \int_{P_{\eps}} |H|^{m-1} + \max \left( \mathcal{H}^m(P_{\eps}), \mathcal{H}^m(P_{\eps})^{1/m}\right) \right] + \text{diam}_{g_{\eps}}(\partial U_{\eps} \times \{0\}) \\ \label{distBound1}
& \leq 2 \alpha + A(\Sigma) + o_{\eps}(1)
\end{align}
having used that the slices $\left(\p U_{\eps} \times \{t\}, g_{\eps} \Big|_{\p U_{\eps} \times \{t\}} \right)$ for $t \in [0, \eps]$ are uniformly geometrically converging to $(\partial U, g \Big|_{\partial U}) = (\Sigma, g \Big|_{\Sigma})$ as $\eps \to 0$. This follows by the form of the metric $g_{\eps}$ (see equation \eqref{modifiedMetric}). If $x,y$ are closer to $C_2$, then one replaces $\text{diam}_{g_{\eps}}(\partial U_{\eps} \times \{0\})$ with $\text{diam}_{g_{\eps}}(\partial U_{\eps} \times \{z_{\eps}\})$ but the same bound holds. \nl 
\begin{figure}[h!]
    \centering
    \includegraphics[scale=0.4]{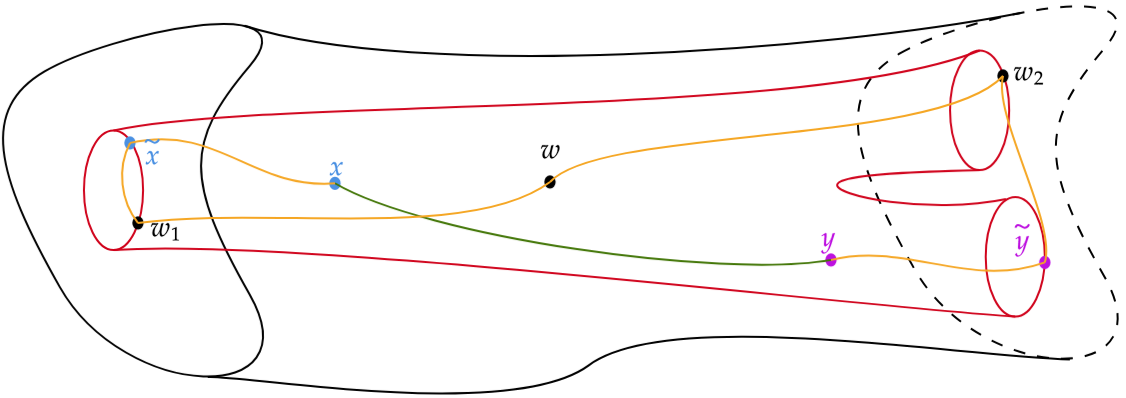}
    \caption{Case 2 visualized. Green denotes distance minimizing geodesic, orange gives the competitor curve.}
    \label{fig:diamcompete}
\end{figure}
\nl \noindent \textbf{Case 2: $i_x \neq i_y$} \newline 
WLOG $i_x = 1$, $i_y = 2$ and suppose there exists $\tilde{x} \in C_1, \tilde{y} \in C_2$ such that 
\begin{align*}
\dist_{P_{\eps}}(x, \tilde{x}) &= \dist_{P_{\eps}}(x, C_1) = \dist_{P_{\eps}}(x, \partial P_{\eps}) \\
\dist_{P_{\eps}}(y, \tilde{y}) &= \dist_{P_{\eps}}(y, C_2) = \dist_{P_{\eps}}(y, \partial P_{\eps})
\end{align*}
And let $\tilde{w}_1 \in C_1$, $\tilde{w}_2 \in C_2$ such that 
\begin{align*}
\dist_{P_{\eps}}(w, \tilde{w}_1) &= \dist_{P_{\eps}}(w, C_1) \\
\dist_{P_{\eps}}(w, \tilde{w}_2) &= \dist_{P_{\eps}}(w, C_2) 
\end{align*}
and recall that $L_w = \dist_{P_{\eps}}(w, C_1) = \dist_{P_{\eps}}(w, C_2) \leq \alpha $ independent of $\eps$. Now consider the competitor curve to the distance minimizing geodesic (in $U_{\eps}$) from $x \to y$: $x \to \tilde{x}$ (minimizing inside of $P_{\eps}$), $\tilde{x} \to \tilde{w}_1$ (minimizing inside of $\p U_{\eps} \times \{0\}$), $\tilde{w}_1 \to w \to \tilde{w}_2$ (minimizing inside of $P_{\eps}$), $\tilde{w}_2 \to \tilde{y}$ (minimizing inside of $\p U_{\eps} \times \{z_{\eps}\}$),  $\tilde{y} \to y$ (minimizing inside of $P_{\eps}$). From this, we can bound
\begin{align*}
\text{dist}_{g_{\eps}}(x,y) &\leq \text{dist}_{P_{\eps}}(x, \tilde{x}) + \text{dist}_{\partial U_{\eps} \times \{0\}}(\tilde{x}, \tilde{y}) + \dist_{P_{\eps}}(w, \tilde{w}_1)  \\
& \quad + \dist_{P_{\eps}}(w, \tilde{w}_2) + \text{dist}_{\partial U_{\eps} \times \{z_{\eps}\}}(\tilde{w}_2, \tilde{y}) + \text{dist}_{P_{\eps}}(\tilde{y}, y) \\
& \leq \alpha + \text{diam}_{g_{\eps}}(\partial U_{\eps} \times \{0\}) + \alpha + \alpha + \text{diam}_{g_{\eps}}(\partial U_{\eps} \times \{z_{\eps}\}) + \alpha \\ 
& \leq 4 \alpha + A(\Sigma) + o_{\eps}(1)
\end{align*}
having bounded $\text{diam}_{g_{\eps}}(\p U_{\eps} \times \{t\})$ via the same argument as in case $1$. See image \ref{fig:diamcompete} for a visualization. \newline 
\indent We have bounded the extrinsic diameter of $P \cap \{t > z_{\eps}\}$, as well as the extrinsic diameter of each component of $P_{\eps} = P \cap \{t \leq z_{\eps}\}$ uniformily in $\eps$. We now argue that $P$ (which is itself connected) has uniformily bounded diameter independent of $\eps$. Let $x,y \in P_{\eps}$ and denote their components by $P_{\eps}^x, P_{\eps}^y$. The extrinsic distance between $x,y$ is less than or equal to the extrinsic distance from $x \to (P \cap \{t = z_{\eps}\})$, plus the diameter of $P \cap \{t \geq z_{\eps}\}$, plus the extrinsic distance from $(P \cap \{t = z_{\eps}\}) \to y$. The first and third terms are bounded by the extrinsic diameter of $P_{\eps}^x$ and $P_{\eps}^y$ respectively, and thus we have a bound of the form $3 \alpha$ for $\alpha$ as in equation \eqref{distBound1}. If $x,y \in P \cap \{t \geq z_{\eps}\}$, we bound the extrinsic distance between them by the extrinsic diameter of $U_{\eps} \cap \{t \geq z_{\eps}\}$, which is bounded uniformily in $\eps$. If $x \in P_{\eps}$ and $y \in P \cap \{t \geq z_{\eps}\}$, we can bound the extrinsic distance between the points by the diameter of $U_{\eps} \cap \{t \geq z_{\eps}\}$ plus the extrinsic diameter of $P_{\eps}^x$, which again produces a uniform bound in $\eps$.
\end{proof}
%
\subsection{Tethering to the Core via the Maximum principle} \label{MaxPrinciple}
%
We will now show that $Y_{\eps,p}$ is ``tethered to the core" for all $\eps$ sufficiently small. This means that each component of $Y_{\eps, p}$ will always touch a point in the interior of $M$ which is uniformily bounded distance (with respect to $g$) away from $\Sigma$, and $Y_{\eps, p} \cap \p U_{\eps} = \emptyset$. Recall the foliation in the contracting neighborhood, $U \cong \Sigma \times [0, \hat{t}]$ as in \S \ref{construction}. 
\begin{proposition} \label{tetherProp}
For all $\eps$ sufficiently small, and each component of $P_{\eps} = \cup_k P_{\eps}^k$
\begin{itemize}
\item There exists a point $p_{\eps}^k \in Y_{\eps,p}^k$ such that $t(p_{\eps}^k) \geq \hat{t}$
\item $Y_{\eps,p}^k \cap \partial U_{\eps} = \emptyset$
\end{itemize}
\end{proposition}
\noindent In fact, the second point follows from the first after noting that $\text{diam}_{ext}(Y_{\eps,p}^k) \leq K$ independent of $\eps$ by proposition \ref{prop:extDiamBound} applied to $P = Y_{\eps, p}$. 
\begin{proof}[Proof of Proposition \ref{tetherProp}] We suppress the $k$ superscript and suppose WLOG that $Y_{\eps, p}$ is connected. Suppose that
\[
t_{\eps} = \sup_{p \in Y_{\eps,p}} t(p) < \hat{t}
\]
By compactness $t_{\eps} = t(p_{\eps})$ for some $p_{\eps} \in Y_{\eps,p}$. But recall that our approximation $h_{\eps}$ satisfies constraint \ref{endBehavior} so that
\[
\forall s \in \Sigma, \qquad |h_{\eps}|(s, t_{\eps}) < H_{t_{\eps}}
\]
for all $\eps$ sufficiently small. Thus, we can apply the standard maximum principle to $Y_{\eps,p}$ at $p_{\eps}$, a contradiction (see figure \ref{fig:preventFBPMC}).
\begin{figure}[h!]
\centering
\includegraphics[scale=0.4]{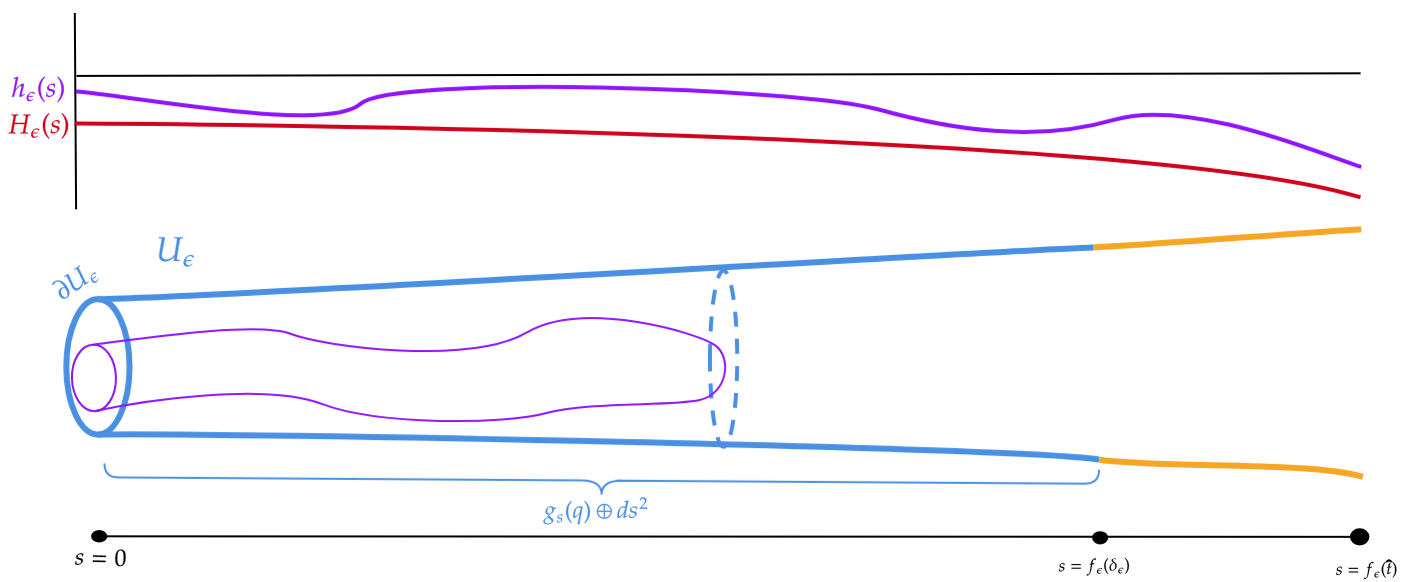}
\caption{Maximum principle argument to prevent the presence of a free boundary prescribed mean curvature surface}
\label{fig:preventFBPMC}
\end{figure}
We also note that $\text{diam}_{ext}(Y_{\eps,p}) \leq K$ independent of $\eps$ from proposition \ref{prop:extDiamBound}, as $Y_{\eps,p}$ has uniformly bounded area (see \eqref{YepsAreaBounds} and remark \ref{ConvergenceOfWidths}) and $H_{Y_{\eps,p}} = h_{\eps}$ where $h_{\eps}$ has uniform $L^{\infty}$ bounds by construction (see point \ref{hEpsSMall}). Thus,
for all $\eps$ sufficiently small
\[
Y_{\eps,p} \cap \partial U_{\eps} = \emptyset
\]
because $\text{dist}(\partial U_{\eps}, \Sigma \times \{\hat{t}\}) \xrightarrow{\eps \to 0} \infty$ as in \S \ref{construction} (in particular because $\int_0^{\delta_{\eps}} v_{\eps}(r) dr \to \infty$).
\end{proof}
\begin{remark}
Proposition \ref{tetherProp} justifies the ``tethered to the core" descriptor and also prevents a free boundary PMC. See figure \ref{fig:tethertocore} for a visualization. 
\end{remark}
\begin{figure}[h!]
\centering
\includegraphics[scale=0.4]{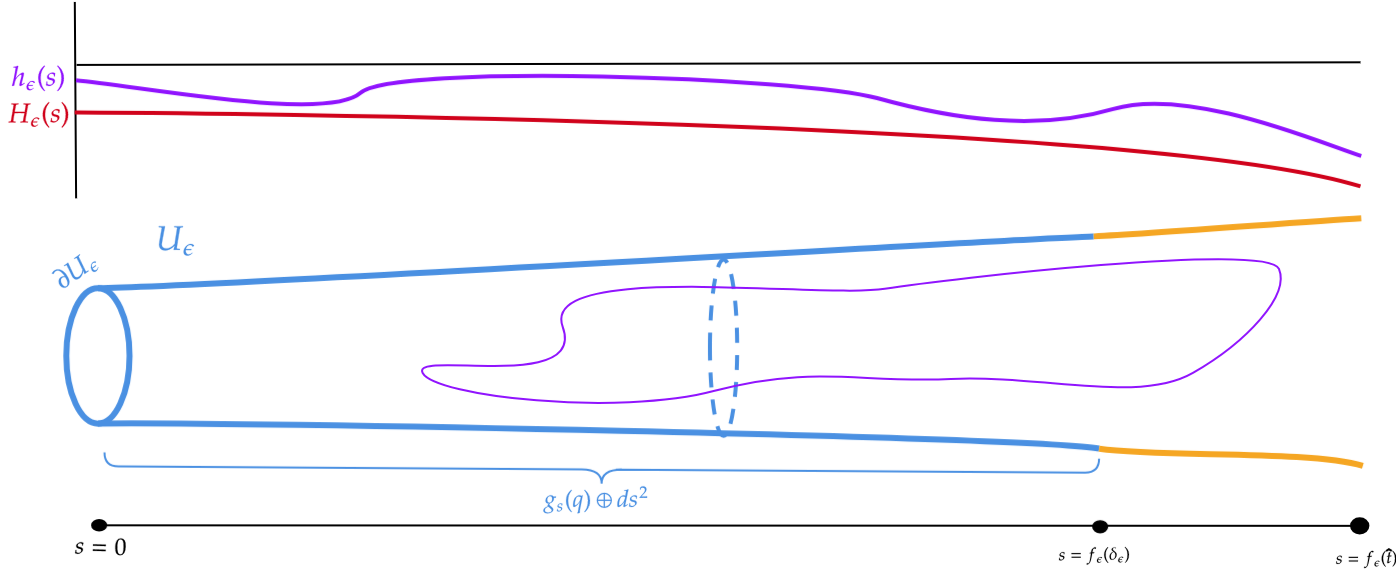}
\caption{Visualization of our PMC touching a part of the core which is a $q \in U$ such that $\text{dist}(q, \partial U) > \hat{t}$}
\label{fig:tethertocore}
\end{figure}
%
%

\begin{remark}
In the above maximum principle argument, we did not need to know the sign of the mean curvature (and hence the underlying Caccioppoli set) since $h$ is small. However, given the sign of $h(p)$, where $p$ is our right most end point, then a choice of normal immediately determines the underlying Caccioppoli set, since the min-max construction of PMCs always yields mean curvature $H = h$, where it is computed with respect to the \textit{outer normal} (see e.g. \cite[\S 2]{SunWangZhou}.
\end{remark}
Though we are considering $h \in C_c^{\infty}(M \backslash \Sigma)$ satisfying assumptions \ref{hCompactSupport}, we note that we can ``shrink" our contracting neighborhood as needed so that proposition \ref{tetherProp} applies to an $h$ satisfying assumptions \ref{hZeroBoundary}.  We will adjust the neighborhood as follows, let: 
\[
\tilde{t} = \sup \{ t \in [0, \hat{t}] \quad \text{s.t.} \quad |h(s,r)| \leq H_r(s), \qquad \forall r \leq t\}
\]
Note that if $h \in C_c^{\infty}(M)$, we clearly have $\tilde{t} > 0$, and if $h$ satisfies assumptions \ref{hZeroBoundary}, then condition \ref{vanishingBoundary} gives $\tilde{t} > 0$ by lemma \ref{quantMCLemma}. We now consider a potentially smaller contracting neighborhood, replacing $\hat{t}$ with $\tilde{t}$, but keeping the $\hat{t}$ notation. With this, proposition \ref{tetherProp} follows verbatim.
\subsection{Construction of $Y_h$: Sending $\eps \to 0$} \label{EpsToZero}
In this section, we will often suppress the $p$ subscript which refers to the homotopy class in the construction of $\{Y_{\eps, p}\}$. We will send $\eps \to 0$ to construct $Y_h \subseteq M \subseteq \Cyl(M)$. Recall that $h \in C_c^{\infty}(M)$ and $h$ satisfies the assumptions of \ref{hCompactSupport}. Our main claim is:
\begin{proposition} \label{PMCConstruction}
Some subsequence $\{Y_{\eps_i}\}$ converges to a multiplicity one, almost embedded PMC, $Y_h \subseteq M$, with mean curvature $H = \pm h$. Moreover, $Y_h$ is closed and $Y_h \cap \Sigma = \emptyset$.
\end{proposition} 
To prove proposition \ref{PMCConstruction} we proceed in a few steps
\begin{enumerate}
\item \label{takeLimit} 
Define the varifold $V_h = \lim_{i \to \infty} V_{Y_{\eps_i}}$. Away from $\Sigma$, the underlying metric $g_{\eps}$ converges smoothly, and so theorem \ref{thm:compactness with changing metrics} gives that $V_h$ is induced by a hypersurface (potentially with multiplicity) in this region.

\item \label{MaxPrincipleDecomposition} Near $\Sigma$, we know that $V_h$ is stationary for the area functional due to the compact support of $h$. Proceeding as in the proof \cite[Step 2, Theorem 10]{song2018existence}, we can apply the maximum principle of Solomon--White/White \cite{solomon1989strong, white2009maximum} and the constancy theorem to conclude that $V_h =  \sum_{i = 1}^m a_i \Sigma_i + W_h$, where $W_h$ is a bounded distance away from $\Sigma$ and $\text{supp}(B_h) \subseteq \Sigma$.

\item \label{NoSigma} We argue that $a_i = 0$ for all $i$ by considering a connected component of $Y_{\eps_i}$ which converges to some connected subcomponents of $\Sigma$, i.e. $\lim_{i \to \infty} Y_{\eps_i}$ has support in $\Sigma$. Such a component of $Y_{\eps_i}$ must also be ``tethered to the core", and by a monotonicity formula argument (see lemma \ref{monotonicityLemma}), some mass of $Y_{\eps_i}$ must be contained in the core, uniformly in $\eps_i$. This prevents varifold convergence to any varifold supported in $\Sigma$ by a ``no-pinching" argument. Hence, our limit PMC is disjoint from $\Sigma$.

\item \label{NoMinimalPart} We further decompose $W_h = Y_{ngbd} + W_{h}^*$, where $Y_{ngbd}$ is connected and contained in $\Sigma \times [0, \hat{t}]$, our contracting neighborhood. We show, by a similar maximum principle + right endpoint argument, that $Y_{ngbd}$ must be empty, and thus $Y_h = W_h^*$ consists of a multiplicity one PMC where $h$ satisfies the goodness conditions of \ref{goodnessAssumption} and hence by theorem \ref{thm:compactness for FPMC}, 
$W_h^*$ must be an almost-embedded,
multiplicity $1$ PMC.

\item We conclude the proof of theorem \ref{mainTheorem} by performing the above construction and analysis for each homotopy class associated to $\omega_p$, forming a collection of surfaces $\{Y_{h,p}\}$. We show that these are distinct by noting that they are multiplicity one, density at most $2$, and their area grows at least linearly via the Weyl Law for manifolds with cylindrical ends. We also show the corresponding index bounds on the regular set, $\mathcal{R}(Y_{h,p})$. 
%
\end{enumerate}
\subsubsection{\ref{takeLimit}: Taking the limit}
Part \ref{takeLimit} follows as in Song \cite[Theorem 10, Step 1]{song2018existence}: By the diameter bounds of proposition \ref{prop:extDiamBound} and the tethering argument of proposition \ref{tetherProp}, we know that there exists $q \in M \backslash \Sigma$ and an $R > 0$, both independent of $\eps$, such that $Y_{\eps} \subseteq B_{g_{\eps}}(q, R)$. Now for a subsequence $\{Y_{\eps_i}\}$, set
\begin{align} \label{LimitVarifold}
V_h &:= \lim_{i \to \infty} Y_{\eps_i} \\ \label{LimitVarifoldMass}
\|V_h\| &= \lim_{i \to \infty} A(Y_{\eps_i}) \in \left[ \omega_{p+1}(\Cyl(M)) - 2 \|h\|_{L^1}, \omega_{p}(\Cyl(M)) +W_0(M,g) + A(\Sigma) + 2\|h\|_{L^1(M,g)}\right]
\end{align}
Note that the interval in \eqref{LimitVarifoldMass} is well defined and non-empty due to lemma \ref{lem:PsiPrimeSweep} and equation \eqref{PsiPrimeUpper} (after choosing $\delta = o_{\eps}(1)$ in \eqref{initialSweepout}). 
In the above notation, $V_h$ is a varifold in $(\Cyl(M), g_{\Cyl})$, and $g_{\Cyl}$ is Lipschitz. Despite the lack of smooth convergence, $g_{\eps} \to g_{\Cyl}$, the uniform $C^1$ bounds on $g_{\eps}$ afforded by lemmas \ref{geomconv} \ref{norm} mean that each $Y_{\eps_i}$ induces a varifold with uniformly bounded first variation in local charts. Thus, the
standard compactness of rectifiable varifolds with bounded first variation (see e.g. \cite[Theorem 5.8]{simon1983lectures}) gives the existence of $V_h$. 
This is the same argument as in \cite[Theorem 10, Step 3]{song2018existence} except we replace stationary varifolds with varifolds of bounded first variation. Via theorem \ref{thm:compactness with changing metrics}, we conclude that away from $\Sigma = \partial U$, $\text{supp}(V_h)$ is actually a smooth hypersurface and the convergence happens smoothly, graphically, away from a finite number of points (potentially with multiplicity). 
\subsubsection{\ref{MaxPrincipleDecomposition}: Initial decomposition} \label{MPsection}
To prove part \ref{MaxPrincipleDecomposition}, we first recall the maximum principle of Solomon--White \cite{solomon1989strong} and White \cite{white2009maximum}:
\begin{theorem}[\cite{solomon1989strong} Theorem, \cite{white2009maximum} Theorem 4] \label{SolWhiteMax}
	Suppose that $N^{n+1}$ is a smooth Riemannian manifold with boundary, $\partial N$ minimal. Suppose that $V$ is an $n$-dimensional varifold that minimizes area to first order in $N$. 
	\begin{enumerate}
		\item \label{maxPrinciple1} If $\text{Spt}(V)$ contains any point of $\partial N$, then it must contain all of $\partial N$.
		\item If $V$ is a stationary integral varifold, then $V$ can be written as $V = W + W'$ where $\text{supp}(W)$ is a connected subset of $\partial N$ and $\text{supp}(W') \cap \partial N = \emptyset$.	
	\end{enumerate}
\end{theorem}
We remark that being minimal implies minimizing to first order. We further note that point \ref{maxPrinciple1} of theorem \ref{SolWhiteMax} actually holds as long as $V$ minimizes area to first order in \textit{an open neighborhood} of $\partial N$. This follows by noting that the proof \ref{SolWhiteMax} is local to the point $p \in \partial N \cap \text{Spt}(V)$. The authors construct a vector field supported in a neighborhood of $\partial N$ to show that when $V \cap \partial N \ni p$ but no neighborhood of $p$ is contained in $\partial N$, then $V$ is not stationary against a vector field which is supported in a neighborhood of $\partial N$. Such a neighborhood of $\partial N$ can be taken to be arbitrarily small and hence, applies in our setting, where $V_h$ is certainly minimal in a neighborhood of $\Sigma$ due to the compact support of $h$ away from $\Sigma$. Applying theorem \ref{SolWhiteMax} to $V = V_h$, we conclude $V_h = \sum_{i = 1}^m a_i \cdot \Sigma_i + W_h$ as desired, where $\text{supp}(W_h) \cap \Sigma = \emptyset$. By the previous section, $\text{supp}(W_h)$ is a smooth hypersurface.
\begin{remark}
The above argument shows that $\text{dist}(\text{supp}(W_h), \Sigma) > 0$, but there is no quantitative lower bound afforded by the proof of theorem \ref{SolWhiteMax}. Moreover, the above argument does not show that $W_h \subseteq \text{supp}(h)$. However, see the theorems in \S \ref{quantMPSection} of the appendix.
\end{remark}
\subsubsection{\ref{NoSigma}: $a_i = 0$} \label{NoSigmaSection}
To prove \ref{NoSigma}, suppose $a_i > 0$ for some $i$. We first recall a monotonicity formula for manifolds with bounded mean curvature (see \S \ref{MonotonicityForm} in the appendix for the proof):
\begin{lemma} \label{monotonicityLemma}
Let $\eta \in (0, 1)$. Consider a metric $\tilde g$ on $B_3 \subset \R^{n+1}$ such that $\|\tilde g - g_{eucl}\|_{C^1(B_3)} \leq \eta$ and let $b \geq 0$. Then there exist nonnegative real numbers $\mathfrak{c} = \mathfrak{c}(\eta,b)$ and $\mathfrak{a} = \mathfrak{a}(\eta,b)$ such that the following holds. 

Let $V$ be a $n$-dimensional varifold in $B_3$ which has $b$-bounded first variation with respect to $\tilde g$, namely:
\[
|\delta^{\tilde g}V(X)| = \left| \int_{B_3 \times G(n,n+1)}\mathrm{div}^{\tilde g}_SX \,dV(x,S) \right| \leq b \int|X|_{\tilde g}\,d\mu_V, \qquad \text{for any $C^1$ vector field $X$ on $B^3$}.
\]
Then for any $\xi \in B_1$ and for any $0 < \sigma \leq \rho < 1$:
\[
e^{\mathfrak{c}\sigma}\cdot\frac{\|V\|(B(\xi,\sigma))}{\sigma^n} \leq (1+\mathfrak{a})e^{c(1+\mathfrak{a})\rho}\cdot\frac{\|V\|(B(\xi,(1+\mathfrak{a})\rho))}{\rho^n}
\]
where $B(\xi, r)$ denotes the Euclidean ball of radius $r$ centered at $\xi$. 
\end{lemma}
\noindent Lemma \ref{monotonicityLemma} implies that our hypersurfaces $Y_{\eps_i,p}$ cannot ``pinch off" and form a copy of $\Sigma_i$ as follows (see figure \ref{fig:nopinching}):
\begin{figure}[h!]
\centering
\includegraphics[scale=0.4]{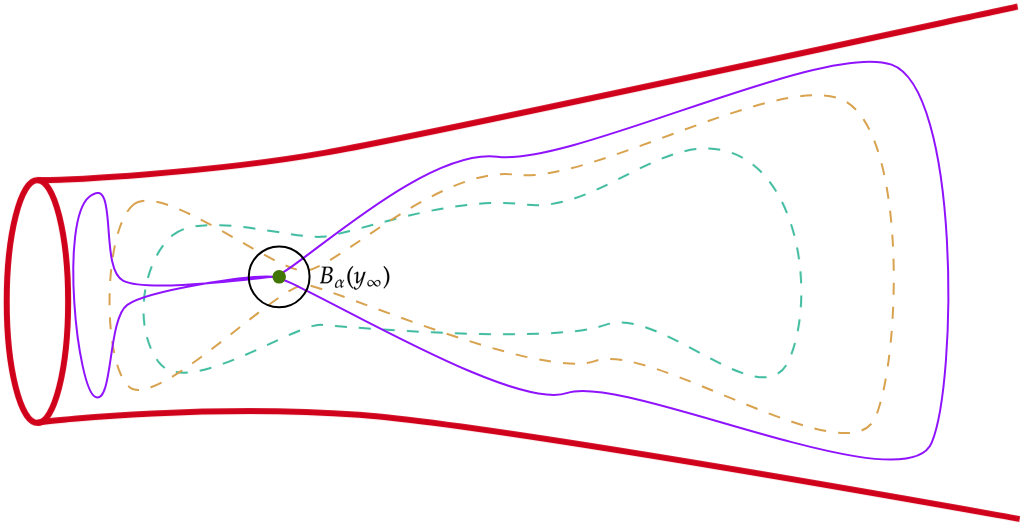}
\caption{A visualization of what can go wrong as $\eps \to 0$. Even though all of the $Y_{\eps,p}$ have uniform diameter bounds, because the metric is modified on $t \in [0, \delta_{\eps}]$ and $\delta_{\eps} \to 0$, the $Y_{\eps,p}$ may accumulate around $\{t = \delta_{\eps}\}$. In the limit, this naively can lead to pinching at a point $y_{\infty}$.}
\label{fig:nopinching}
\end{figure}
\begin{lemma}[No Pinching]\label{NoPinchingLemma}
No component of $Y_{\eps_i,p}$ can converge to a non-trivial varifold supported in the boundary, $\tilde{B}_h$.
\end{lemma}
\begin{proof}
Suppose that a component $Y_{\epsilon_i,p}^*$ of $Y_{\epsilon_i,p}$ converges to 
\begin{equation} \label{componentConverge}
Y_0= \tilde{B}_h +\tilde{W}_h
\end{equation}
where $\text{supp}(\tilde{B}_h) \subseteq \Sigma$ and $\tilde{W}_h$ is a varifold with $\text{supp}(\tilde{W}_h) \cap \Sigma_k = \emptyset$ as in part \ref{MaxPrincipleDecomposition}. This means that we can pick some $\alpha \in (0, \mathrm{dist}(\Sigma,\tilde{W}_h)/4) $ so that the $\alpha$-neighborhoods of $\Sigma$ and of (the support of) $\tilde{W}_h$ with respect to the (Lipschitz) metric $g_\infty$ do not cover $Y_{\epsilon_i,p}^*$. In other words, for each $i$, there is a $y_i \in Y_{\epsilon_i,p}^*\subset U_{\epsilon_i}$ such that
\[ \mathrm{dist}^{g_\infty}(y_i,\Sigma) \geq 2\alpha \qquad \text{and} \qquad  \mathrm{dist}^{g_\infty}(y_i,\tilde{W}_h) \geq 2\alpha.\]
(see figure \ref{fig:nopinching} for the visualization). Passing to a subsequence, if necessary, we may assume  $y_i \to y \in \mathring{M}$ where $y$ also satisfies the inequalities above. In particular, we may assume that the closure of $B^{g_\infty}_\alpha(y)$ does not intersect $\Sigma \cup \tilde{W}_h$, and hence
\[a_k\mathcal{H}_{g_\infty}^{n}(\Sigma \cap B^{g_\infty}_\alpha(y)) + \|\tilde{W}_h\|(B^{g_\infty}_\alpha(y))=0\]
Let $\rho \in (0,\alpha)$ be such that lemma \ref{monotonicityLemma} can be applied to $g_\infty$-balls of radius $\leq \rho$. In particular, for sufficiently large $i$, we get $\|Y_{\epsilon_i,p}^*\|(B_\rho^{g_\infty}(y)) \geq C$ for some $C>0$ independent of $i$. We may also assume that $B_{\rho}^{g_\infty}(y_i)\subset B_\alpha^{g_\infty}(y)$ for sufficiently large $i$. Nevertheless this leads to a contradiction, as: 
\[0 =a_k\mathcal{H}_{g_\infty}^{n}(\Sigma_k \cap B^{g_\infty}_\alpha(y)) + \|\tilde{W}_h\|(B^{g_\infty}_\alpha(y))= \lim_i \|Y_{\epsilon_i,p}^*\|(B^{g_\infty}_\alpha(y))\geq \liminf_i\|Y_{\epsilon_i,p}^*\|(B^{g_\infty}_\rho(y))\geq C > 0\]
\end{proof}
\subsubsection{\ref{NoMinimalPart}: Almost embedded, multiplicity $1$, PMC } \label{noMinimalSection}
To prove part \ref{NoMinimalPart}, we first decompose $Y_\infty=\mathrm{Spt}(V_h)$ as
\[
Y_{\infty} = Y_{ngbd} + W_h^*
\]
where $Y_{ngbd}$ denotes all (if there are any) connected components which are contained in $\Sigma \times [0, \hat{t}]$, i.e. the contracting neighborhood. If we consider 
\[
t_{max} = \sup_{p \in Y_{ngbd}} t(p) = t(p_0) \qquad p_0 \in Y_{ngbd}
\]
then noting that $|h(p_0)| < H_t(p_0)$ in our contracting neighborhood by property $2$ from \S \ref{hCompactSupport}, we get a contradiction via the maximum principle as in the proof of proposition \ref{tetherProp}. \nl
\indent It suffices to show that $W_h^*$ is almost-embedded, multiplicity $1$, and has density at most $2$ everywhere. By condition \ref{goodnessAssumption} on $h$, we have that $W_h^*$ restricted to $M_{s}$ (for $s \leq \hat{t}$) satisfies the generic regularity conclusions of theorem \ref{thm:compactness with changing metrics}, and so $W_h^*$ is realized by an almost-embedded with optimal regularity, multiplicity $1$ PMC (with boundary) in this region - i.e. the touching set restricted to $\{t \geq s\}$ will be at most Hausdorff dimension $(n-1)$. By theorem \ref{densityTheorem}, the density of any non-minimal component of $W_h^*$ is at most $2$. So it suffices to show that any minimal components of  $W_h^*$ are multiplicity $1$. This is true for any such components lying in $\{t \geq s\}$. Moreover, by construction and the tethering argument of proposition \ref{tetherProp}, every component of $W_h^*$ leaves $\Sigma \times [0, \hat{t}]$, so there are no minimal components contained inside $\Sigma \times [0, s]$. Thus the multiplicity of each component of $W_h^*$ (as almost embedded hypersurfaces) is $1$ and the touching set can only have n-Hausdorff dimension in $W_h^* \cap \{t \leq s\}$.
\subsection{Conclusion of proof of \ref{mainTheorem}} \label{conclusionSection}
With this, we conclude the proof of theorem \ref{mainTheorem} as follows:
\begin{itemize}
\item For each $p \in \mathbb{Z}^+$, we construct an almost embedded, multiplicity one PMC surface, $Y_{h,p}$, with $H_{Y_{h,p}} = h$. 
\item We compute its area as
\begin{align*}
A(Y_{h,p}) &= \lim_{i \to \infty} A(Y_{\eps_i, p}) \\
A(Y_{h,p}) & \in \left[ (p+1) \cdot A(\Sigma_1) - 2 \|h\|_{L^1}, p \cdot A(\Sigma_1) + W_0(M,g) + A(\Sigma) + 2\|h\|_{L^1(M,g)} + C p^{1/(n+1)} \right]
\end{align*}
%
where $C = C(M, g)$. Here, we used \eqref{LimitVarifoldMass} in the second line, the convergence of the widths $\omega_{p, \eps} \to \omega_p(\Cyl(U), g_{\Cyl})$ from remark \ref{ConvergenceOfWidths}, and the bounds of $\omega_p(\Cyl(U), g_{\Cyl})$ from theorem \ref{CylindricalWeylThm}.
%
\item $\mathrm{Index}(\mathcal{R}(Y_{h,p})) \leq p+1$ by construction.
\end{itemize}
Note that though $\{Y_{h,p}\} \Big|_{t \leq \hat{t}}$ may have points with multiplicity $2$, each $Y_{h,p}$ decomposes into a collection of almost-embedded, connected components, $Y_{h,p} = \cup_{k} Y_{h,p}^k$, with $Y_{h,p}$ as a whole being almost embedded with density at most $2$. Each $Y_{h,p}^k$ leaves the neighborhood $\Sigma \times [0, \hat{t}]$ again by proposition \ref{tetherProp}, and the restricted components, $Y_{h,p}^{k} \cap \{t \geq \hat{t}\}$ are all pairwise distinct. Thus, each $Y_{h,p}^k$ occurs with multiplicity $1$ in the decomposition of $Y_{h,p}$. Moreover, each $Y_{h,p}^k$ is itself almost embedded. Now by the linear growth of $A(Y_{h,p})$, we see that there must be infinite distinct, almost embedded, $h$-PMCs, which have density at most $2$. \nl 
%
\indent To see the index bound, we present a modification of the argument of \cite[Thm 2.6, Part 2]{ZhouMultiplicity} to localize the variations to the regular set. Suppose that $\mathrm{Index}(\mathcal{R}(Y_{h,p})) = I > p + 1$. Then there exist $\{\phi_{i}\}_{i = 1}^I \in C_{c}^{\infty}( \mathcal{R}(Y_{h,p}))$ such that $\delta^2 \AA^h$ is negative definite on $V = \text{Span}\left(\{\phi_i\}_{i = 1}^I \right)$. Let $K = \cup_{i = 1}^I \text{supp}(\phi_i)$ and note that $K \subset \subset \mathcal{R}(Y_{h,p})$. Let $U$ be a tubular normal neighborhood of $K$ such that $U \supseteq \{p \in M \; | \; \dist(p, K) \leq \delta\}$ for some $\delta > 0$. Via the same construction as in \cite[Prop 4.3]{marques2015morse}, $V$ induces a family of diffeomorphisms $\{F_v \; : \; v \in \overline{B}^I\} \subseteq \text{Diff}(U)$ such that
%
\[
-\frac{1}{c_0} Id \leq D^2 \AA^h(F_v(\Omega_p \cap U)) \leq - c_0 Id
\]
for some $c_0 \in (0,1)$. Since $Y_{\eps_i} \cap U$ converges smoothly to $Y_{h,p} \cap U$ away from finitely many points, the corresponding restrictions of the underlying Cacciopoli sets $\Omega_{\eps_i} \cap U$ converge  to $\Omega_p \cap U$ in the $\mathbf{F}$-norm in $\CC(U)$. Thus the map $v \mapsto \AA^h (F_v(\Omega_{\eps_i} \cap U))$ converges \textit{smoothly} $v \mapsto \AA^h(F_v(\Omega_{\eps_i} \cap U))$ and hence 
\[
-\frac{2}{c_0} \leq D^2 \AA^h(F_v(\Omega_{\eps_i} \cap U)) \leq - \frac{c_0}{2} Id
\]
for all $v \in \overline{B}^I$, contradicting the upper weak index bounds of $Y_{\eps_i}$. 
\section{Proof of theorem \ref{ZeroBoundaryTheorem}: Extension to $h \not \in C_c^{\infty}(M \backslash \Sigma)$} \label{ExtensionSection}
In this section, we prove Theorem \ref{ZeroBoundaryTheorem}, showing that we can extend and improve our construction of PMCs to good functions $h \in C^{\infty}(M)$ with vanishing boundary conditions \ref{hZeroBoundary}.
\begin{theorem*}
Suppose that $(M^{n+1}, g)$, $3 \leq n+1 \leq 7$, is a manifold with boundary so that $\Sigma = \partial M$ is an embedded, strictly stable minimal surface. For any $h$ satisfying Assumptions \ref{hZeroBoundary}, there exist infinitely many distinct, almost embedded with optimal regularity, \emph{multiplicity one} hypersurfaces, $\{Y_{h,p}\}_{p=1}^{\infty}$, with mean curvature $h$ and area bounds \eqref{AreaBounds}. Moreover, each $Y_{h,p}$ is disjoint from $\Sigma$ and $\mathrm{Index}(\mathcal{R}(Y_{h,p}))$.
\end{theorem*}
\begin{proof}
As $\Sigma$ is strictly stable, there exists a contracting neighbood, $U \supseteq \Sigma$, and a corresponding map $\Phi: \Sigma \times [0, \hat{t}] \to U$. As described at the end of \S \ref{MaxPrinciple}, we can make $\hat{t}$ potentially smaller (but non-zero), to enforce that 
\[
|h(s,t)| \leq H_t(s) \qquad \forall (s,t) \in \Sigma \times [0, \hat{t}]
\]
Now, consider a sequence $h_i$ with compact support satisfying Assumptions \ref{hCompactSupport} and such that $\|h_i\|_{C^2}(M)$ is uniformly bounded and $h_i \xrightarrow{C^{1}(M)} h$. 
We can construct such sequence by multiplying by a bump function and using the high vanishing order of $h$ at the boundary. Formally, consider 
\[
h_i(p) = h(p) \varphi_{\delta_i}(\text{dist}(p, \Sigma))
\]
where $\varphi(t)$ is a smooth bump function that is identically $0$ for $t \leq \frac{1}{2}$, goes from $0 \to 1$ on $[\frac{1}{2},1]$, and is identically $1$ for $t \geq 1$. Then let $\varphi_{\delta}(t) = \varphi(t/ \delta)$. Clearly $h_i \to h$ smoothly around all points with $\dist(p, \Sigma) > \delta_0$ for some $\delta_0 > 0$ fixed. For all $t < \delta_0$ and $\delta_i$ sufficiently small, we let $(s,t) \in \Sigma \times [0, \delta_0]$ be local coordinates for $M$. Since $h \Big|_{\Sigma} = \p_{\nu} h \Big|_{\Sigma} = 0$, we have 
\begin{align} \nonumber
h_i(s,t) &= \varphi_{\delta_i}(t) \cdot h \\ \label{hExpansion}
h(s,t)  &= t^2 \cdot h_{t,2}(s) + R(s,t) 
\end{align}
where $R(s,t) = O(t^3)$ is the remainder. Note that for $k \geq 1$,
\begin{align*}
t \in [0, \delta_i] & \implies |\p_t^k \varphi_{\delta_i}(t)| \leq \delta_i^{-k} \cdot C_k \\
t > \delta_i & \implies |\p_t^k \varphi_{\delta_i}(t)| = 0
\end{align*}
where $C_k$ is independent of $\delta_i$. We now bound $h_i - h$ on $t \in [0, \delta_i]$ by noting that
\begin{align*}
\sup_{(s,t) \in \Sigma \times [0,\delta_i]} |D(h_i - h)(s,t)| & \leq K \Big[ |\p_t \varphi_{\delta_i}(t)| \cdot |h(s,t)| + |1- \varphi_{\delta_i}(t)| \cdot |D h(s,t)| \Big] \\
& \leq K [ \delta_i^{-1} \cdot \delta_i^2 + 2 \cdot \delta_i]  \leq K \delta_i
\end{align*}
and
\begin{align*}
\sup_{(s,t) \in \Sigma \times [0,\delta_i]} |D^2(h_i - h)(s,t)| & \leq K \Big[ |\p_t^2 \varphi_{\delta_i}(t)| \cdot |h(s,t)| + |\p_t \varphi_{\delta_i}(t)| \cdot |D h(s,t)|  + |(1 - \varphi_{\delta_i}(t))| \cdot |D^2 h(s,t)| \Big] \\
& \leq K [ \delta_i^{-2} \cdot \delta_i^2  + \delta_i^{-1} \delta_i + \delta_i^{-1} \delta_i^2 + 2 ] \leq K
\end{align*}
having used \eqref{hExpansion} extensively. A similar (easier) bound works for the $0$-th order convergence. Since $h_i \equiv h$ outside of $\Sigma \times [0,\delta_i]$, this finishes the proof of convergence in $C^1$. By the interpolation inequality for $C^2 \subset C^{1,\alpha} \subset C^1$, we obtain $h_i \to h$ in $C^{1,\alpha}(M)$, for any $\alpha \in (0,1)$. \nl 
%
\indent From construction it is clear that the $h_i$ satisfy parts \ref{LOneBound}, \ref{LessThanSlice}, and \ref{compactPositiveIntegral} for all $i$ sufficiently small, since $h$ satisfies each of these properties. To see that condition \ref{goodnessAssumption} is satisfied, let $\tilde{t}_i = 2 \delta_i$. Then for $i$ sufficiently large (and hence $\delta_i$ sufficiently small, we see that 
\begin{align*}
M_{\tilde{t}_i} &= U^c \cup \{p \; | \; t(p) \geq \tilde{t}_i\} \\
h_i \Big|_{M_{\tilde{t}_i}} &= h \Big|_{M_{\tilde{t}_i}}
\end{align*}
So clearly $h_i$ is morse on $M_{\tilde{t}_i}$. By condition \ref{goodnessAssumption2} on $h$, we know that 
\[
\Sigma'(h_i) = \{h_i = 0\} \cap \{t \geq \tilde{t}_i\} = \{h = 0\} \cap \{t \geq \tilde{t}_i\} = \Sigma'(h)
\]
where $\Sigma'(h)$ is the closed, smoothly embedded hypersurface in $M \backslash \Sigma$, which is thus contained in the interior of $M_{\tilde{t}_i}$ for $\tilde{t}_i$ sufficiently small. Thus $\Sigma'(h_i)$ is automatically transverse to $\{t = \tilde{t}_i\}$, and the mean curvature vanishes to finite order by the assumption on $\Sigma'$ from \ref{goodnessAssumption2}. \nl 
\indent For fixed $p$ and each $h_i$, we construct an almost-embedded $h_i$-PMCs, $Y_{i,p}$ from theorem \ref{mainTheorem}. Moreover, in the process of the construction, we actually construct $Y_{i,p,\eps_i}$, where $Y_{i,p,\eps_i}$ is an $h_{\eps_i}$ PMC on $(U_{\eps_i}, g_{\eps_i})$. Moreover, $\eps_i$ is arbitrarily small, $h_{\eps_i}$ is an arbitrarily close approximation of $h_i$, and $Y_{i,p,\eps_i}$ can be taken to be graphical over $Y_{i,p}$ with an arbitrarily small graphical function, outside a set of finitely many points. Via theorem \ref{thm:compactness with changing metrics} and in particular, conclusion \ref{compactness thm:smooth limit}, we send $i \to \infty$ and construct $V_{h,p}$ the limit varifold with $Y_{h,p} = \text{supp}(V_{h,p})$ a smooth $h$-PMC with $Y_{i,p, \eps_i}$ converging to $Y_{h,p}$ graphically away from at most $p$ points - this follows from the compactness for PMCs with mean curvature satisfying uniform $C^1$ bounds (see \ref{thm:compactness for FPMC} and remark \ref{C1 alpha convergence}). Note moreover, that $Y_{h,p}$ is closed, as opposed to free boundary with non-trivial boundary. This follows by  applying Zhou's original compactness theorem for PMC surfaces with bounded index, \cite[Theorem 2.6]{ZhouMultiplicity}: consider each $Y_{i,p}$ as closed PMC surfaces inside $M$ (or a slight extension of $M$ beyond the boundary), and hence each $Y_{i,p, \eps_i}$ as closed PMCs inside of $M$ to conclude. Thus $Y_{h,p} = \lim_{i \to \infty} Y_{i,p,\eps_i}$ is also closed inside $M$. In particular, $Y_{h,p}$ will not be free boundary unless (vacuously) $Y_{h,p} \cap \Sigma = \emptyset$.  \nl 
\indent Now decompose
\[
V_{h,p} = B_h + W_h
\]
where $\text{supp}(B_h) \subseteq \Sigma$ and $\text{supp}(W_{h}) \cap \Sigma = \emptyset$ by the maximum principle (which applies as $Y_{h,p} = \text{supp}(V_{h,p})$ is now a hypersurface). To see that $B_h = 0$, we argue as in proposition \ref{PMCConstruction}, part \ref{NoSigma}: if $Y_{i,p} \xrightarrow{i \to \infty} B_h + W_h$, then some connected component of $Y_{i,p}$, call it $Y_{i,p}^*$, converges to a varifold, $\tilde{B}_h$, supported in $\Sigma$. However, by proposition \ref{tetherProp}, each $Y_{i,p}^*$ contains a point $q \in Y_{i,p}^*$ such that $\dist(q, \Sigma) \geq \hat{t}$. Thus the same ``no-pinching" argument of lemma \ref{NoPinchingLemma} applies, meaning that $Y_{i,p}$ cannot converge to a varifold supported in $\Sigma$ (plus some other components), i.e. $B_h = 0$. \nl 
\indent We thus conclude that $Y_{i,p} \rightarrow Y_{h,p}$, where $Y_{h,p}$ is bounded away from $\Sigma$. Moreover, because $h$ satisfies restrictions \ref{hZeroBoundary} (in particular assumption \ref{goodnessAssumption2}), we argue that $W_h$ is multiplicity $1$ and almost embedded with optimal regularity. To see this, note that for $i$ sufficiently large, the smooth convergence away from finitely many points given by theorem \ref{thm:compactness for FPMC} (and the fact that $Y_{h,p}$ is compact) guarantees that there exists a $t_0 > 0$ such that $\dist(Y_{i,p}, \Sigma) > t_0$ for all $i$ sufficiently large. Taking $i$ potentially even larger, we guarantee that for $M_{t_0} = \{\dist(p) > t_0\}$,
\[
h_i \Big|_{M_{t_0}} = h
\]
But by the goodness assumption of \ref{goodnessAssumption2}, we see that any convergence of PMCs in $M_{t_0}$ then occurs with multiplicity $1$. \nl 
\indent Note that the above argument implicitly shows that each $Y_{i,p}$ is also multiplicity $1$ for $i$ sufficiently large, and hence $Y_{i,p} = \partial \Omega_{i,p}$ for some Caccioppoli set, as $Y_{i,p}$ itself is a limit of boundaries converging with multiplicity $1$ (see also \cite[Thm 6.3]{simon1983lectures}, or \cite[Prop 7.3]{ZhouZhu}). Since we have graphical convergence with multiplicity one away from finitely many points of $Y_{h,i} \xrightarrow{i \to \infty} Y_p$, the underlying Caccioppoli sets also converge, i.e. $\Omega_{i,p} \xrightarrow{i \to \infty} \Omega_p$ in the $\mathbf{F}$ topology, where $Y_{h,p} = \partial \Omega_p$. In particular, the weak index bound is preserved by part \ref{compactness thm:weak index bound} of theorem \ref{thm:compactness for FPMC}, i.e. $\mathrm{index}_w(Y_{p})\leq p+1$, so by equation \eqref{weakIndexUpper}, we have that $\mathrm{Index}(\mathcal{R}(\Sigma)) \leq p+1$ as well. \nl 
\indent To see that there are infinitely many distinct $Y_{h,p}$, we note that the min-max values are preserved, i.e. if $Y_{i,p} = \partial \Omega_i$, then
\begin{align*}
\lim_{i \to \infty} \AA^{h_i}(\Omega_{p,i}) &= \AA^h(\Omega_p) \\
\AA^{h_i}(\Omega_{p,i}) &
\in \left[ \omega_{p+1}(\Cyl(M)) - \|h_i\|_{L^1}, \omega_{p}(\Cyl(M)) +W_0(M,g) + A(\Sigma) +\|h_i\|_{L^1(M,g)}\right] \\ 
\implies \AA^h(\Omega_p) &\in \left[ \omega_{p+1}(\Cyl(M)) - \|h\|_{L^1}, \omega_{p}(\Cyl(M)) +W_0(M,g) + A(\Sigma) + \|h\|_{L^1(M,g)}\right] \\
\implies A(Y_{h,p}) &\in \left[ (p+1) \cdot A(\Sigma_1) - 2 \|h\|_{L^1}, p \cdot A(\Sigma_1) + W_0(M,g) + A(\Sigma) + 2\|h\|_{L^1(M,g)} + C p^{1/(n+1)} \right]
\end{align*}
Again having used the bounds in theorem \ref{CylindricalWeylThm}. Thus infinitely many of these intervals are distinct and we conclude that infinitely many of the $Y_p = \partial \Omega_p$ is also distinct.
\end{proof}
%

\section{Closed Manifolds: Proof of Theorem \ref{mainTheoremClosed} and Corollaries \ref{mainCorollaryHomology} \ref{mainCorollaryFrankel}} \label{corollariesSection}

In this section, we will prove the existence of infinitely many PMCs in the closed setting, as described in theorem \ref{mainTheoremClosed} and corollaries \ref{mainCorollaryHomology}, \ref{mainCorollaryFrankel}.
\begin{proof}[Proof of Theorem \ref{mainTheoremClosed}]
It is assumed that $\Sigma^n \subseteq M^{n+1}$ is strictly stable, so its corresponding lift $\tilde{\Sigma} \subseteq \comp(M \backslash \Sigma)$ is also strictly stable and has a contracting neighborhood. Thus, we can apply theorems \ref{mainTheorem} and \ref{ZeroBoundaryTheorem} for prescribing function $h \in C^{\infty}$ which satisfy conditions \ref{hCompactSupport} or \ref{hZeroBoundary} on $\comp(M \backslash \Sigma)$. \nl 
\indent We note that if $\Sigma$ happens to be separating in $M$, i.e. $M = M^+ \sqcup_{\Sigma} M^-$, then we can apply theorems \ref{mainTheorem} and \ref{ZeroBoundaryTheorem} in the component which contains the contracting neighborhood. This is a weaker condition than $\partial \comp(M \backslash \Sigma)$ having a contracting neighborhood.
\end{proof}
\begin{proof}[Proof of corollary \ref{mainCorollaryHomology}]
For $g$ a bumpy metric on $M$, minimizing area in a non-trivial homology class $\alpha \in H_n(M, \Z_2)$ produces a smooth, embedded stable minimal surface $\Sigma$ (see e.g. \cite[\S 5.1.6]{federer2014geometric}). Let $\tilde{\Sigma} = \partial \comp(M \backslash \Sigma)$. Since $\Sigma$ represents a non-trivial homology class, $\Sigma$ is not separating, and $\comp(M \backslash \Sigma)$ is a compact manifold with boundary equal to either two copies of $\Sigma$ (when $\Sigma$ is two-sided) or one copy of the double cover of $\Sigma$ (when $\Sigma$ is one-sided), see figure \ref{fig:metriccompletion}. Under the assumption of the corollary, if $\Sigma$ is degenerate, then $\tilde{\Sigma}$ admits a contracting neighborhood. Note that $\tilde{\Sigma}$ automatically admits a contracting neighborhood when $\Sigma$ is strictly stable, as strict stability passes to $\tilde{\Sigma}$ and one can construct a contracting neighborhood on $\tilde{\Sigma} \times [0, \hat{t}] \subseteq \comp(M \backslash \Sigma)$, see e.g. \cite[\S 3]{song2018existence}. \nl 
\indent Consider $h$ satisfying the assumptions in \ref{hCompactSupport} or \ref{hZeroBoundary} on $\comp(M \backslash \Sigma)$. We apply theorems \ref{mainTheorem} (or \ref{ZeroBoundaryTheorem}) to produce infinitely many smooth, distinct, almost embedded, PMCs, $\{Y_{h,p}\}$, on $\text{comp}(M \backslash \Sigma)$, where $Y_{h,p}$ is closed in $\comp(M \backslash \Sigma)$ and has the appropriate density depending on assumption \ref{hCompactSupport} vs. \ref{hZeroBoundary}. Moreover, each $Y_{h,p}$ is disjoint from $\tilde{\Sigma}$. Since $\comp(M \backslash \Sigma) \backslash \tilde{\Sigma}$ is isometrically diffeomorphic to $M \backslash \Sigma$ and each $Y_{h,p} \subseteq \comp(M \backslash \Sigma) \backslash \tilde{\Sigma}$, each $Y_{h,p}$ has the same regularity, multiplicity, and almost embeddedness as a closed PMC in $M$. See figures \ref{fig:metriccompletion}, \ref{fig:hnpmc} for visualizations.
\end{proof}
\begin{proof}[Proof of corollary \ref{mainCorollaryFrankel}]
%
\noindent Note that if $H_n(M, \Z_2) \neq 0$, then we apply corollary \ref{mainCorollaryHomology} with the corresponding assumptions on the homology minimizer, $\Sigma$. So WLOG, assume $H_n(M, \Z_2) = 0$. \newline 
\indent The existence of an embedded, strictly stable minimal surface follows as in \cite[Lemma 12, (1)]{song2018existence} (see also \cite[Thm 9.1]{meeks2008stable}), which we include: if $(M, g)$ does not satisfy the Frankel property, then there exist two disjoint connected minimal hypersurfaces, $\Gamma, \Gamma'$. If either $\Gamma, \Gamma'$ are stable, then by the assumptions of the corollary, both are strictly stable and we are done. Thus, assume that both hypersurfaces are unstable. \nl 
\indent Define $N$ to be a component of $\comp(N \backslash (\Gamma \sqcup \Gamma'))$ which has at least two different new boundary components coming from $\Gamma, \Gamma'$. Let $S_0$ be a component of $\partial N$ coming from $\Gamma$. We now minimize the n-volume in the homological class of $S_0$ inside of $N$. Again, by \cite[\S 5.1.6]{federer2014geometric} and the maximum principle with $S_0$ as a barrier, minimization in this homology class yields a smoothly embedded, strictly stable minimal hypersurface such that one of its components is two-sided and contained in the interior of $N$. Call this component $\Sigma$ and note that since $\Sigma$ arises via minimization in homology, it admits a contracting neighborhood. \newline 
\indent Now we can apply theorems \ref{mainTheoremClosed} on $M$ to yield infinitely many distinct, almost embedded PMCs, $\{Y_{h,p}\}$ when $h$ satisfies conditions \ref{hCompactSupport} or \ref{hZeroBoundary}. \nl 
\indent Note that because $H_n(M, \Z_2) = 0$, $\Sigma$ is separating i.e., $M = M^+ \sqcup_{\Sigma} M^-$. Thus, we could have also applied theorem \ref{mainTheorem} to either component, $M^{\pm}$ (see figure \ref{fig:dumbbell}), to yield infinitely many distinct, almost embedded PMCs when $h$ satisfies conditions \ref{hCompactSupport}. 
\end{proof}
%
%
%
\section{Appendix}
\subsection{An almost-Euclidean monotonicity formula for varifolds with bounded first variation} \label{MonotonicityForm}
In this section, we prove lemma \ref{monotonicityLemma}, which is a generalization of \cite[Lemma 2]{song2018existence}. See also \cite[Thm 3.17]{simon1983lectures} for a similar theorem.
\begin{lemma}
Let $\eta \in (0, 1)$. Consider a metric $\tilde g$ on $B_3 \subset \R^{n+1}$ such that $\|\tilde g - g_{eucl}\|_{C^1(B_3)} \leq \eta$ and let $b \geq 0$. Then there exist nonnegative real numbers $\mathfrak{c} = \mathfrak{c}(\eta,b)$ and $\mathfrak{a} = \mathfrak{a}(\eta,b)$ such that the following holds. 

Let $V$ be a $n$-dimensional varifold in $B^3$ which has $b$-bounded first variation with respect to $\tilde g$, namely:
\[
|\delta^{\tilde g}V(X)| = \left| \int_{B_3 \times G(n,n+1)}\mathrm{div}^{\tilde g}_SX \,dV(x,S) \right| \leq b \int|X|_{\tilde g}\,d\mu_V, \qquad \text{for any $C^1$ vector field $X$ on $B^3$}.
\]

Then for any $\xi \in B_1$ and for any $0 < \sigma \leq \rho < 1$:
\[
e^{\mathfrak{c}\sigma}\cdot\frac{\|V\|(B(\xi,\sigma))}{\sigma^n} \leq (1+\mathfrak{a})e^{c(1+\mathfrak{a})\rho}\cdot\frac{\|V\|(B(\xi,(1+\mathfrak{a})\rho))}{\rho^n}
\]
where $B(\xi, r)$ denotes the Euclidean ball of radius $r$ centered at $\xi$. 

More generally, for any smooth, nonnegative function $h$ defined on $B$ and $\xi$, $\sigma$, and $\rho$ as above,
\[
\frac{e^{\mathfrak{c}\sigma}}{\sigma^n}\int_{B(\xi,\sigma)}h\,d\|V\| \leq (1+\mathfrak{a})e^{c(1+\mathfrak{a})\rho}\cdot\left[\frac{1}{\rho^n} \int_{B(\xi,(1+\mathfrak{a})\rho)} h\,d\|V\| + \int_\sigma^{(1+\mathfrak{a})\rho}\int_{B(\xi,\tau)\times G(n+1,n)}\frac{|\nabla_Sh|}{\tau^n}\,dV(x,S)\,d\tau\right]
\]
where $|\nabla_Sh|$ is the norm of the gradient of $h$ along $S$ computed with respect to $g_{eucl}$. Moreover,  $\mathfrak{a}(\eta,b)$ and $|\mathfrak{c}(\eta,b) - b|$ converge to $0$ as $\eta \to 0$, uniformly with respect to $b$.
\end{lemma}
%
%
\begin{proof}
We follow closely Antoine’s proof of the monotonicity formula, Lemma 2, which is inspired by Simon’s proof in the case of stationary/bounded variation.

For $\xi \in B_1$, we let $g_\xi$ be the flat metric $\tilde g(\xi)$ on $B_3$. We also write $r_\xi(x) = \|x-\xi\|_{g_\xi}$ to denote the distance function from $\xi$ with respect to $g_\xi$, and denote by $\tilde \nabla$, $\tilde \nabla_S$ and $\tilde{\mathrm{div}}_S$ the Levi-Civita connection of the metric $\tilde g$, its projection onto an $n$-plane, $S$, and the divergence operator along $S$ computed with the connection $\tilde \nabla$, respectively. Note that by the assumption on the metrics, we have $(1-\eta)g_{eucl}\leq g_\xi \leq (1+\eta)g_{eucl}$, in the sense of bilinear forms. 

Pick $\epsilon>0$ such that $(1+\epsilon)\rho<2$, and let $\gamma(t)=\varphi(t/\rho)$, where $\varphi \colon \R \to [0,1]$ is a smooth cutoff function such that $\varphi'\leq 0$, $\varphi(t) \equiv 1$ for $t \leq 1$ and $\varphi(t) \equiv 0$ for $t\geq 1+\epsilon$.

Let $h$ be a smooth nonnegative function on $B_3$, and consider the smooth, compactly supported vector field $X$ defined by $X_x = \gamma(r_\xi(x))h(x)\frac{1}{2} \tilde \nabla r_\xi^2(x)$, for $x \in B_3$. The divergence of $X$ was computed in \cite[equation (27)]{song2018existence} as follows: let $(x,S)\in B_3 \times G(n,n+1)$ be an $n$-plane at $x \neq \xi$, and pick an $\tilde g$-orthonormal basis $e_1,\ldots,e_n$ for $S$.  Then
\[
\tilde{\mathrm{div}}_SX \geq r_\xi \gamma'(r_\xi) h \|\tilde\nabla_Sr_\xi\|_{\tilde g}^2 + r_\xi\gamma(r_\xi)\tilde g(\tilde \nabla r_\xi,\tilde \nabla_S h) + n \gamma(r_\xi)h\|\tilde\nabla r_\xi\|_{\tilde g}^2-2Cr_\xi\gamma(r_\xi)h\| \tilde \nabla r_\xi\|_{\tilde g}^2,
\]
where $C$ is a positive constant depending only on $\eta$, and which can be chosen arbitrarily small as $\eta \to 0$.

Since $t\gamma^\prime(t) = \frac{t}{\rho}\varphi^\prime(t/\rho) = -\rho \frac{\partial}{\partial \rho}\left(\varphi(t/\rho) \right)$, we compute, using the bounded first variation (and also that $r_\xi \gamma(r_\xi) \leq (1+\epsilon) \rho \varphi(r_\xi/\rho)$):

\begin{align*}
    & n\int \varphi(r_\xi/\rho)h\|\tilde\nabla r_\xi\|_{\tilde g}^2\,dV - \rho \frac{\partial}{\partial \rho} \int \varphi(r_\xi/\rho) h \|\tilde\nabla_Sr_\xi\|_{\tilde g}^2\,dV \\[3pt] & \qquad \leq b\int r_\xi\gamma(r_\xi)h\|\tilde\nabla r_\xi\|_{\tilde g}\,d\mu_V - \int r_\xi\gamma(r_\xi)\tilde g(\tilde \nabla r_\xi,\tilde \nabla_S h)\,dV + 2C(1+\epsilon)\rho \int \varphi(r_\xi/\rho)h\|\tilde\nabla r_\xi\|^2_{\tilde g}\,dV\\[3pt]
    &\qquad \leq (1+\epsilon)b\cdot\rho\int \varphi(r_\xi/\rho)h\|\tilde\nabla r_\xi\|_{\tilde g}\,dV - \int r_\xi\gamma(r_\xi)\tilde g(\tilde \nabla r_\xi,\tilde \nabla_S h)\,dV + 2C(1+\epsilon)\rho \int \varphi(r_\xi/\rho)h\|\tilde\nabla r_\xi\|^2_{\tilde g}\,dV
\end{align*}
Let $\tilde I(\rho) = \int \varphi(r_\xi/\rho)h\|\tilde\nabla r_\xi\|_{\tilde g}^2\,dV$. If we write $\tilde \nabla=\tilde \nabla_{S^\perp} + \tilde \nabla_S$ (orthogonal decomposition w.r.t. $\tilde g$), then

\[
n\tilde I(\rho) - \rho \frac{\partial}{\partial \rho}\tilde I(\rho) = n\int \varphi(r_\xi/\rho)h\|\tilde\nabla r_\xi\|_{\tilde g}^2\,dV-\rho \frac{\partial}{\partial \rho}\int \varphi(r_\xi/\rho) h \left(\|\tilde\nabla_Sr_\xi\|_{\tilde g}^2+\|\tilde\nabla_{S^\perp}r_\xi\|_{\tilde g}^2\right)\,dV 
\]
and hence
\begin{align*}
 n\tilde I(\rho) - \rho \frac{\partial}{\partial \rho}\tilde I(\rho) & \leq-\rho \frac{\partial}{\partial \rho}\int \varphi(r_\xi/\rho) h \|\tilde\nabla_{S^\perp}r_\xi\|_{\tilde g}^2\,dV +(1+\epsilon)b\cdot\rho\int \varphi(r_\xi/\rho)h\|\tilde\nabla r_\xi\|_{\tilde g}\,dV \\[3pt] & \qquad \qquad - \int r_\xi\gamma(r_\xi)\tilde g(\tilde \nabla r_\xi,\tilde \nabla_S h)\,dV  + 2(1+\epsilon)C\rho \tilde I(\rho)
\end{align*}

Using the comparison between the metrics $g_\xi$ and $g_{eucl}$, we see that there exists a positive constant $c'=c'(\eta)$ such that 
\[1 \leq (1+c^\prime)\|\tilde \nabla r_\xi\|_{\tilde g}^2\leq (1+c^\prime)^2 \quad \text{and} \quad \lim_{\eta \to 0^+}c^\prime(\eta)=0.\] 
(see also \cite[p. 31]{song2018existence}). Thus, $\|\tilde\nabla r_\xi\|_{\tilde g} \leq \sqrt{1+c'} \cdot \|\tilde\nabla r_\xi\|_{\tilde g}^2$ and
\[
n\tilde I(\rho) - \rho \frac{\partial}{\partial \rho}\tilde I(\rho) \leq-\rho \frac{\partial}{\partial \rho}\int \varphi(r_\xi/\rho) h \|\tilde\nabla_{S^\perp}r_\xi\|_{\tilde g}^2\,dV - \int r_\xi\gamma(r_\xi)\tilde g(\tilde \nabla r_\xi,\tilde \nabla_S h)\,dV  + (1+\epsilon)(2C+b\cdot\sqrt{1+c'})\rho \tilde I(\rho)
\]
Therefore,
\begin{align*}
\frac{\partial}{\partial \rho}\left[ e^{(1+\epsilon)(2C+b\cdot\sqrt{1+c'})\cdot \rho} \frac{\tilde I(\rho)}{\rho^n} \right] &= \frac{e^{(1+\epsilon)(2C+b\cdot\sqrt{1+c'})\rho}}{\rho^{n+1}}\cdot \left[\rho\frac{\partial}{\partial\rho}\tilde I(\rho) - n\tilde I(\rho) +(1+\epsilon)(2C+b\cdot\sqrt{1+c'})\rho\tilde I(\rho)   \right] \\[3pt]
&\geq \frac{e^{(1+\epsilon)(2C+b\cdot\sqrt{1+c'})\rho}}{\rho^{n+1}}\Big[\rho \frac{\partial}{\partial \rho}\int \varphi(r_\xi/\rho) h \|\tilde\nabla_{S^\perp}r_\xi\|_{\tilde g}^2\,dV \\
& \qquad \qquad \qquad \qquad \qquad \qquad  + \int r_\xi\gamma(r_\xi)\tilde g(\tilde \nabla r_\xi,\tilde \nabla_S h)\,dV \Big]
\end{align*}

Choose $\mathfrak{c}=2C+b\cdot\sqrt{1+c'}$ and use Cauchy-Schwarz $\tilde g(\tilde \nabla r_\xi,\tilde \nabla_S h) \geq -\|\tilde\nabla r_\xi\|_{\tilde g}\|\tilde\nabla_S h\|_{\tilde g}$  (and once again $r_\xi \gamma(r_\xi) \leq (1+\epsilon) \rho \varphi(r_\xi/\rho)$) to get

\[
\frac{\partial}{\partial \rho}\left[ e^{(1+\epsilon)\mathfrak{c}\cdot \rho} \frac{\tilde I(\rho)}{\rho^n} \right]\geq \frac{e^{(1+\epsilon)\mathfrak{c}\rho}}{\rho^{n}}\left[\frac{\partial}{\partial \rho}\int \varphi(r_\xi/\rho) h \|\tilde\nabla_{S^\perp}r_\xi\|_{\tilde g}^2\,dV- \int \varphi(r_\xi/\rho) \|\tilde\nabla r_\xi\|_{\tilde g}\|\tilde\nabla_S h\|_{\tilde g}\,dV \right]
\]

Integrating this inequality from $\sigma$ to $\rho$, and taking the limit $\epsilon \to 0$ to make $\varphi \to \chi_{(-\infty,1]}$, we get (see also Simon’s, page 90): 
\begin{align*}
 e^{(1+\epsilon)\mathfrak{c}\cdot \rho} \frac{\tilde I(\rho)}{\rho^n} -  e^{(1+\epsilon)\mathfrak{c}\cdot \sigma} \frac{\tilde I(\sigma)}{\sigma^n} & \geq e^{\mathfrak{c}\sigma}\int_{(B^{\tilde g}(\xi,\rho)\setminus B^{\tilde g}(\xi,\sigma))\times G(n,n+1)} h(x) \frac{\|\tilde\nabla_{S^\perp}r_\xi(x)\|_{\tilde g}^2}{r_\xi^n(x)}\,dV(x,S)  \\[3pt]& \qquad - e^{\mathfrak{c}\rho}\int_\sigma^\rho\frac{1}{\tau^{n}}\int_{B^{\tilde g}(\xi,\tau)\times G(n,n+1)} \|\tilde\nabla r_\xi(x)\|_{\tilde g}\|\tilde\nabla_S h(x)\|_{\tilde g}\,dV(x,S)\,d\tau
\end{align*}
%
%
The proof is concluded using the same estimates and metric comparisons as in \cite[p. 32]{song2018existence}, choosing $\mathfrak{a}=(1+c’)^{n+2}-1$.
\end{proof}
\subsection{A quantitative Maximum Principle in bumpy metrics} \label{quantMPSection}
In this section, present a quantitative maximum principle for PMCs in $M$ - i.e. if $Y$ is an $h$ PMC satisfying assumptions \ref{hCompactSupport} or \ref{hZeroBoundary}, then if $Y \cap \Sigma = \emptyset$, it must lie a finite distance from $\Sigma$ depending on $h$, the index of $Y$, and the area of $Y$. \nl \nl 
First, suppose that $\Sigma = \partial M$ is minimal with a foliation given by a contracting neighborhood, $\{\Sigma_t\}$, such that an open neighborhood $U \supseteq \Sigma$ is diffeomorphic to $U \cong \Sigma \times [0, \hat{t}]$ in fermi coordinates. 
%
\begin{theorem}[Quantitative Maximum Principle] \label{QMPThm}
Let $Y \subseteq M$ be a closed, almost embedded PMC with curvature $H = \pm h$ such that $Y \cap \Sigma = \emptyset$, $Y = \partial \Omega$ 
for some $\Omega \in \CC(M)$, $A(Y) \leq K$, $\text{Ind}(Y) \leq I$, $||h||_{C^{1,\alpha}} \leq C$, and $h$ satisfying assumptions \ref{hCompactSupport} or \ref{hZeroBoundary}. There exists $d = d(K, I, C) > 0$ such that $\text{dist}(Y, \Sigma) \geq d$
\end{theorem}
\begin{proof}
Suppose no such $d$ exists, then for some combination of $K, I, C$, we have a sequence of closed PMCs, $\{Y_i = \partial \Omega_i\}$, with $A(Y_i) \leq K$, $\text{Ind}(Y_i) \leq I$, $||h_i||_{C^{1,\alpha}} \leq C$, and 
\[
\lim_{i \to \infty} \text{dist}(Y_i, \Sigma) = 0
\]
By Arzela--Ascoli, $h_i \xrightarrow{C^1} h \in C^1(M)$ satisfying the analogous conditions. Moreover, by compactness, there exists $p_i \in Y_i$ such that $\text{dist}(Y_i, \Sigma) = \text{dist}(p_i, \Sigma)$. Applying theorem \ref{thm:compactness for FPMC}, we get that $Y_i \xrightarrow{i \to \infty} V_{\infty}$, $\text{supp}(V_{\infty}) = Y_{\infty}$ a closed $C^1$ hypersurface in $M$ with mean curvature $\pm h$, and $p_i \to p_{\infty} \in \Sigma \cap Y_{\infty}$. Note that theorem \ref{thm:compactness for FPMC} applies verbatim as long as $Y_i = \partial \Omega_i$ (i.e. even if the prescribing function $h_i \not \in \mathcal{S}(g)$) . \nl  
\indent However, now by the maximum principle, we have that some component of $Y_{\infty}$ must be $\Sigma$, and so 
\[
V_{\infty} = \sum_{i = 1}^m a_i \Sigma_i + W_{\infty}
\]
where $\text{supp}(W_{\infty}) \cap \Sigma = \emptyset$ and $a_i \geq 0$ with $\sum_{i = 1}^m a_i \geq 1$. This would imply that some component of the $Y_i$ converges to $\Sigma$, and by theorem \ref{thm:compactness for FPMC} this convergence happens graphically away from finitely many points. However, by the ``no-pinching"/tethering to the core argument of proposition \ref{NoPinchingLemma}, this cannot happen, as every component of $Y_i$ has a point, $q_i$, with $t(q_i) \geq \hat{t}$. This is a contradiction.
\end{proof}
\noindent We can also make the assumption that $\Sigma$ is non-degenerate and relax our conditions on $h$:
\begin{theorem}[Non-degenerate Quantitative] \label{NonDegenMPThm}
Suppose $(M^{n+1}, g)$ such that $\Sigma = \partial M$ is a non-degenerate minimal surface. Let $Y \subseteq M$ be a closed, almost embedded PMC with curvature $H = \pm h$ such that $Y = \partial \Omega$ for some $\Omega \in \CC(M)$, $h\Big|_{\Sigma} = \p_{\nu} h \Big|_{\Sigma} = 0$,  $Y \cap \Sigma = \emptyset$, $A(Y) \leq K$, $\text{Ind}(Y) \leq I$, and $||h||_{C^{1,\alpha}} \leq C$. There exists $d = d(K, I, C) > 0$ such that $\text{dist}(Y, \Sigma) \geq d$.
\end{theorem}
\begin{proof}
The proof proceeds similarly: for some $K, I, C > 0$, we can find surfaces, $Y_i$, and prescribing functions, $h_i$, such that $H_{Y_i} = \pm h_i$ and the $Y_i$ are arbitrarily close to $\Sigma$. Up to subsequence, $h_i \xrightarrow{C^1} h$, and using the compactness of PMCs with bounded area and index, we get our decomposition of 
\[
Y_{i} \to V_{\infty} = \sum_{j = 1}^m a_j \Sigma_j + W_{\infty}
\]
we note that $a_i \geq 1$ for some $i$, and the graphical convergence (away from $I$ points) from theorem \ref{thm:compactness for FPMC} allows us to construct a Jacobi field of $\Sigma_i$ as follows: consider the lowest sheet of the connected component of $Y_i$ that converges graphically to $\Sigma_j$ with $a_j \geq 1$. Then we can represent this sheet, $S_i$, as 
\[
S_i = \exp_{\Sigma_j}(u_{i,j}(s) \nu)
\]
Considering the mean curvature of $S_i$ in comparison to $\Sigma$, we get 
\begin{align*}
H_{S_i} - H_{\Sigma} &= J_{\Sigma} (u_{i,j}) + O(u_{i,j}^2) \\
H_{S_i} - H_{\Sigma} &= h_i(s, u_{i,j}(s)) - 0 \\
&= h_i(s, u_{i,j}(s)) - h_i(s, 0) \\
&= (\p_{\nu} h_i)(s) \cdot u_{i,j} + O(u_{i,j}^2) \\
&= O(u_{i,j}^2)
\end{align*}
where $J_{\Sigma}$ is the jacobi operator on $\Sigma$ and we've used that
all $h_i$ satisfy $h_i \Big|_{\Sigma} = \p_{\nu} h_i \Big|_{\Sigma} \equiv 0$. Normalizing both sides by $||u_{i,j}||_{C^0}$ and sending $i \to \infty$ yields a Jacobi field, a contradiction. 
\end{proof}
In theorems \ref{QMPThm} \ref{NonDegenMPThm}, we can actually replace $||h||_{C^{1,\alpha}(M)}$ with $||h||_{C^{1,\alpha}(\Sigma \times [0, \hat{t}])}$, i.e. we only need $C^{1,\alpha}$ bounds near the boundary. This yields the same theorem.
\bibliographystyle{amsplain}
\bibliography{main}

\end{document}